\documentclass[a4paper, reqno, 10pt]{amsart}
\usepackage{amssymb}
\usepackage{verbatim}
\usepackage{extarrows}
\usepackage{bm,stmaryrd}
\usepackage[centering]{geometry}
\usepackage{enumitem}
\usepackage{longtable}
\usepackage{multirow,tikz}
\usepackage{array}
\usepackage{graphics}
\usepackage{quiver}
\usepackage[all]{xy}
\usepackage{makecell}
\usepackage{diagbox}

\numberwithin{equation}{section}
\theoremstyle{definition}
\newtheorem{defn}{Definition}[section]
\newtheorem{rem}[defn]{Remark}

\newtheorem{exm}[defn]{Example}

\theoremstyle{plain}

\newtheorem{thm}[defn]{Theorem}
\newtheorem{lem}[defn]{Lemma}

\newtheorem{prop}[defn]{Proposition}
\newtheorem{fact}[defn]{Fact}
\newtheorem{defn-thm}[defn]{Definition-Theorem}

\newtheorem*{Freyd-Mitchell}{Freyd-Mitchell embedding Theorem}
\def\Fib{\operatorname{Fib}}
\def\CoFib{\operatorname{CoFib}}
\def\Weq{\operatorname{Weq}}
\def\Coker{\operatorname{Coker}}
\def\Hom{\operatorname{Hom}}

\def\Ker{\operatorname{Ker}}
\def\Id{\operatorname{Id}}
\def\TFib{\operatorname{TFib}}
\def\TCoFib{\operatorname{TCoFib}}
\def\Ho{\operatorname{Ho}}

\def\A{\mathcal{A}}

\def\M{\mathcal{M}}
\def\E{\mathcal{E}}

\def\Mor{\operatorname{Mor}}
\def\Iso{\operatorname{Iso}}

\newcommand{\xra}{\xlongrightarrow}
\title[Orthogonal Model Structures]{Orthogonal Model Structures}
\author[Yang Gao, Yu-xiao Yang, Pu Zhang] {Yang Gao, Yu-xiao Yang, Pu Zhang$^*$ \\ \\  School of Mathematical Sciences \\
Shanghai Jiao Tong University,  \ Shanghai \ 200240, \ China }
\thanks{$^*$ Corresponding author}
\thanks{gyang112358$\symbol{64}$sjtu.edu.cn \ \ \ \ yuxiaoyyx@sjtu.edu.cn\ \ \ \ pzhang$\symbol{64}$sjtu.edu.cn}
\thanks{This work was supported by the National Natural Science Foundation of China Grant No. 12131015.}

\begin{document}

\begin{abstract} This paper studies orthogonal model structures, i.e., model structures such that a lifting in the Lifting axiom is unique.
Cofibrant (fibrant) objects are defined without initial (terminal) objects.
For an orthogonal model structure on a category with enough cofibrant objects and fibrant objects, it is proved that the homotopy category
is equivalent to the full subcategory of cofibrant-fibrant objects. TTF model structures, bi-reflective model structures,
and torsion model structures, are introduced.
They are all orthogonal. One to one correspondences
between TTF model structures and TTF triples  in an abelian category, bi-reflective model structures and bi-reflective pairs in any category,
and torsion model structures and twin torsion pairs in an abelian category, are established in a constructive way.
Torsion model structures on a poset, on the category of $G$-sets, and on the category of topological groups, are also constructed.
The homotopy categories of all these model structures are computed. In particular, the homotopy category of a TTF model structure is an abelian category.

\vskip5pt


\vskip5pt


\end{abstract}

\maketitle

\vspace{-20pt}

\section{\bf Introduction}

\subsection{Orthogonal model structures}  Model structures and their homotopy categories, introduced by Quillen, have been widely used in representation theory, homological algebra,
topology, and algebraic geometry (see e.g. \cite{Q1, Q2, H1, MV, HSS, H2, Hir, BR, G2}).
Various variants of model structures are also of much interest. Orthogonal model structure is one of them.

\vskip 5pt

By definition, a model structure on a category is  \emph{an orthogonal model structure}, provided that a lifting in the Lifting axiom is unique.
Such a model structure has been called \emph{a Quillen factorization system} by Pultr and Tholen \cite{PT},
and  \emph{a one-dimensional model category} by Balchin and Garner \cite{BG}.

\vskip 5pt

As a model structure can be reformulated by two weak factorization systems (\cite{AHRT, Joy, St, DLL}),
an orthogonal model structure can be reformulated by two orthogonal factorization systems (\cite{FK, PT, Riehl}),
which is one of the main ways for constructing an orthogonal model structure. See Proposition \ref{prop:weq-two-out-of-three}.

\vskip 5pt

This paper is devoted to the fundamental problems of constructing  several orthogonal model structures, and determining their homotopy categories.

\subsection{\bf Homotopy categories of orthogonal model structures}
The homotopy category $\Ho(\mathcal A)$ of a model structure on category $\A$ is an important object of study in algebra and topology.

\vskip5pt

For a model structure on a category with initial (respectively, terminal) object, one has the notion of cofibrant (respectively, fibrant) object.
However, many categories possibly have no initial (respectively, terminal) objects.
For studying orthogonal model structures on such categories (for examples, bi-reflective model structures on an arbitrary category;
torsion model structures on a poset),
we need  the notion of cofibrant (respectively, fibrant) objects, without initial or terminal objects (see Definition \ref{defn:cofibrant-fibrant-objects}).
For an orthogonal model structure on category $\mathcal A$ with zero object, this coincides with the one in the usual sense in \cite{Q1}, and
$\mathcal A$ always has enough cofibrant objects and enough fibrant objects.

\vskip 5pt

\noindent\textbf{Theorem A.} (Theorem~\ref{thm:orthogonal-ho-CF}) \ {\it Let $\mathcal A$ be a category with enough cofibrant objects and enough fibrant objects, and  $(\CoFib,\Fib,\Weq)$ an orthogonal model structure on $\mathcal A$. Then
the restriction of the localization functor $\mathcal A\longrightarrow\Ho(\mathcal A)$ induces an equivalence $\mathcal C\cap\mathcal F\cong \Ho(\mathcal A)$ of categories,
where $\mathcal C$ $($respectively, $\mathcal F)$ is equivalent to the full subcategory of cofibrant $($respectively, fibrant$)$ objects.}

\vskip5pt

\noindent\textbf{Corollary.} \ {\it Let $\mathcal A$ be a category with zero object. Then the homotopy category of an orthogonal model structure on $\mathcal A$ is
the full subcategory of cofibrant-fibrant objects.}

\vskip5pt

For an orthogonal model structure on a category with zero object, one has $\mathcal C\cap\mathcal F\cap\mathcal W =0$ (see Fact~\ref{fact:orthogonal-zero}),
where $\mathcal W$ is the class of trivial objects.
Thus, if $\mathcal{A}$ is a weakly idempotent complete additive category, then Theorem A is known by \cite[Theorem 1.1]{LZ}, which claims that
the homotopy category $\Ho(\mathcal A)\simeq
(\mathcal C\cap\mathcal F)/(\mathcal C\cap\mathcal F\cap\mathcal W)$ for any model structure on $\mathcal{A}$.

\subsection{\bf TTF model structures} Let $\mathcal A$ be an abelian category.
From a torsion pair in $\mathcal A$, Everaert and Gran \cite[Proposition 1.4]{EG} have constructed two orthogonal factorization systems in $\mathcal A$ (see Lemma \ref{lem:torsion-pair-ofs}).

\vskip 5pt

By a TTF triple (\cite{J1}) in $\mathcal A$ one means a \emph{torsion-torsionfree triple} $(\mathcal C,\mathcal W,\mathcal F)$, i.e.,
$(\mathcal C, \mathcal W)$ and $(\mathcal W, \mathcal F)$ are torsion pairs in $\mathcal A$. In this case,
$\mathcal W$  is a Serre subcategory.

\vskip 5pt

By definition a TTF \emph{model structure} on abelian category $\mathcal A$ is an orthogonal model structure $(\CoFib,\Fib,\Weq)$ such that
\[
\CoFib=\{f\in {\rm Mor}(\mathcal A) \ \mid \ \operatorname{Coker}f\in\mathcal C\},
\qquad
\Fib=\{f\in {\rm Mor}(\mathcal A) \ \mid \ \Ker f\in\mathcal F\},
\]
where $\mathcal C$ and $\mathcal F$ are the classes of cofibrant objects and fibrant objects, respectively.

\vskip 5pt

As shown in the following theorem, there is a one to one correspondence between TTF triples and TTF model structures.

\vskip 5pt

\noindent\textbf{Theorem B.} (Theorem~\ref{thm:ttf-triples-ttf-orthogonal-model-structures}) \ {\it
Let $\mathcal A$ be an abelian category. Then
$$\Phi:\{{\rm TTF} \ \mbox{triples in} \ \mathcal A\} \longrightarrow \{{\rm TTF} \ \mbox{model structures on} \ \mathcal A\},  \ (\mathcal C,\mathcal W,\mathcal F) \longmapsto (\CoFib,\Fib,\Weq)$$
is a bijection,  where
\begin{align*}
\CoFib&=\{f\in {\rm Mor}(\mathcal A) \ \mid \ \operatorname{Coker}f\in\mathcal C\}, \qquad
\Fib =\{f\in {\rm Mor}(\mathcal A) \ \mid \ \Ker f\in\mathcal F\}\\
\Weq&=\{f\in {\rm Mor}(\mathcal A) \ \mid \ \Ker f\in\mathcal W, \ \operatorname{Coker}f\in\mathcal W\}
\end{align*}
with the inverse $\Psi:(\CoFib, \Fib, \Weq) \longmapsto (\mathcal C,\mathcal W,\mathcal F)$, where \ $\mathcal C$, $\mathcal F$, and $\mathcal W$ are respectively the classes of cofibrant objects, fibrant objects, and trivial objects.

\vskip5pt

In this case the homotopy category $\Ho(\mathcal A)$ is an abelian category, and it is equivalent to $\mathcal C\cap\mathcal F$ as a category.
In particular, $\mathcal C\cap\mathcal F$ enjoys a structure of an abelian category.}

\vskip 10pt

TTF model structures appear widely and naturally (see Remark \ref{remTTF}).
In contrast to the fact that the homotopy category of an abelian model structure (see \cite{H2}; also \cite{G2}) is a triangulated category (see \cite{NP}), the homotopy category of a TTF model structure is an abelian category.

\subsection{\bf Bi-reflective model structures} A full subcategory of a category is reflective (respectively, coreflective) if the inclusion functor
has the left (respectively, right) adjoint, which is called the reflector (respectively, coreflector). This is a fundamental notion (\cite {Mac}).
Cassidy - H\'{e}bert - Kelly \cite[Theorem 4.1]{CHK85} have introduced {\it semi-left-exactness} (respectively, {\it semi-right-exactness}) of the reflector (respectively, coreflector), and then constructed an associated orthogonal factorization system
(see Lemma \ref{lem:sle-reflective-ofs}).

\vskip 5pt

Based on this, bi-reflective pairs and bi-reflective model structures are introduced in this paper (see Definitions~\ref{defn:bi-reflectivepair} and \ref{defn:bi-reflective-orthogonal-model-structure}).
By definition bi-reflective model structures are  orthogonal. As shown in the following theorem, there is a one to one correspondence between bi-reflective pairs and bi-reflective model structures.

\vskip 5pt

\noindent\textbf{Theorem C.} (Theorem~\ref{thm:bi-reflective-pairs-bi-reflective-orthogonal-model-structures}) \
{\it Let $\mathcal A$ be a  category. Then
\begin{equation*}
\begin{aligned}
\Phi:\{\text{bi-reflective pairs in} \ \mathcal A\} &\longrightarrow \{\text{bi-reflective model structures on} \ \mathcal A\}
\\ (\mathcal C,\mathcal F) &\longmapsto (\mathcal E_t, \ \mathcal M_r, \ \mathcal M_t\circ \mathcal E_r)
\end{aligned}
\end{equation*}
is a bijection, where $t$ is the coreflector of $\mathcal C$ with the counit $\varepsilon$, $r$ is the reflector of $\mathcal F$ with the unit $\eta$,
\begin{equation*}
\begin{aligned}
\mathcal E_t &= \left\{
f:X\longrightarrow Y\ \middle|\begin{array}{c}
\xymatrix@C=0.9cm@R=0.4cm{
tX \ar[r]^-{\varepsilon_{_X}} \ar[d]_{t(f)} & X \ar[d]^{f} \\
tY \ar[r]^-{\varepsilon_{_Y}} & Y
}
\end{array}\text{ is a pushout}
\right\}
\\
\mathcal M_t
&= \{\ f \in {\rm Mor}(\mathcal A)\ \mid \  t(f)\text{ is an isomorphism}\}
\end{aligned}
\end{equation*}
\begin{equation*}
\begin{aligned}
\mathcal E_r &= \{f \in {\rm Mor}(\mathcal A)\ \mid \  r(f)\text{ is an isomorphism}\}
\\
\mathcal M_r
&=
\left\{
f: X\longrightarrow Y\ \middle|\begin{array}{c}
\xymatrix@C=0.9cm@R=0.4cm{
X\ar[r]^-{\eta_{_X}}\ar[d]_f & rX\ar[d]^{r(f)}\\
Y\ar[r]^-{\eta_{_Y}} & rY
}
\end{array}\text{is a pullback}
\right\}.
\end{aligned}
\end{equation*}

\vskip 5pt
\noindent
The inverse of $\Phi$ is $\Psi:(\CoFib, \Fib, \Weq) \longmapsto (\mathcal C,\mathcal F)$, where \ $\mathcal C$ and $\mathcal F$ are respectively the classes of cofibrant objects and fibrant objects.}

\vskip 10pt

Bi-reflective model structures appear widely and naturally. Even on an abelian category, a TTF model structure is in general not bi-reflective (see Remark \ref{both-TTF-and-bi-reflective}).

\subsection{\bf Torsion model structures}
Rosick\'y and Tholen \cite{RT} have defined \emph{a torsion factorization system} in a category as an orthogonal factorization system $(\mathcal E, \mathcal M)$ such that $\mathcal E$ and $\mathcal M$ have the two-out-of-three property;
and then for a homological category $\mathcal A$, they proved there is a one to one correspondence between normal torsion factorization systems and torsion pairs (\cite[Theorem~5.2]{RT}).

\vskip5pt

Motivated by this, a \emph{torsion model structure} on a category is introduced in this paper as an orthogonal model structure $(\CoFib,\Fib,\Weq)$ such that $\CoFib$ and $\Fib$ have the two-out-of-three property.
Torsion model structures on abelian category $\mathcal A$ enjoy a more pleasant property, in the sense that there is a one to one correspondence between torsion model structures on $\mathcal A$ and twin torsion pairs in $\mathcal A$.

\vskip5pt

By {\it a twin torsion pair} (\cite{T}) in an abelian category $\mathcal A$ one means a quadruple $(\mathcal T_1, \mathcal F_1; \mathcal T_2, \mathcal F_2)$,  where
$(\mathcal T_1, \mathcal F_1)$ and $(\mathcal T_2, \mathcal F_2)$ are torsion pairs in $\mathcal A$ with $\mathcal T_2 \subseteq \mathcal T_1$.
A torsion pair $(\mathcal T,\mathcal F)$ will be written as $(\mathcal T, \mathcal F, t, r)$, where $t:\mathcal A\longrightarrow \mathcal T$ and $r:\mathcal A\longrightarrow \mathcal F$ are functors such that $t(X)$ and $r(X)$
are  respectively the torsion part and the torsionfree part of $X$. In this way,
a twin torsion pair  $(\mathcal T_1, \mathcal F_1; \mathcal T_2, \mathcal F_2)$ will be written as $(\mathcal T_1, \mathcal F_1, t_1, r_1; \mathcal T_2, \mathcal F_2, t_2, r_2).$

\vskip 5pt

\vskip 5pt

\noindent\textbf{Theorem D.} (Theorem~\ref{thm:abelian-bijection-models-torsion-pairs}) \ {\it Let $\mathcal A$ be an abelian category.

\vskip 5pt

$(1)$ \ There is a one to one correspondence between torsion model structures on $\mathcal A$ and twin torsion pairs in $\mathcal A$.
Explicitly,
$$\Phi:  \left\{\text{\rm twin torsion pairs} \ \mbox{in} \ \mathcal A\right\} \longrightarrow \left\{\text{\rm torsion model structures on }\mathcal A \right\}$$
$$(\mathcal T_1, \mathcal F_1, t_1, r_1; \mathcal T_2, \mathcal F_2, t_2, r_2)  \longmapsto (\CoFib, \Fib, \Weq)$$
is a bijection, where
\begin{equation*}
\begin{aligned}
\CoFib & =\{f\in {\rm Mor}(\mathcal A) \ \mid \ r_1(f) \text{ is an isomorphism}\ \}
\\
\Fib &=\{f\in {\rm Mor}(\mathcal A) \ \mid \ t_2(f) \text{ is an isomorphism}\ \}
\\
\Weq &=\{f\in {\rm Mor}(\mathcal A) \ \mid \ H(f) \text{ is an isomorphism}\ \}
\end{aligned}
\end{equation*}
with $ H\cong t_1\circ r_2 \cong r_2\circ t_1: \mathcal A \longrightarrow \mathcal T_1\cap \mathcal F_2;$
and the inverse is  $$\Psi:(\CoFib, \Fib, \Weq) \longmapsto (\mathcal T_1, \mathcal F_1; \mathcal T_2, \mathcal F_2)$$ where
\begin{equation*}
\begin{aligned}
&\mathcal T_1:=\{X\in \mathcal A \ \mid \ 0\longrightarrow X \ \mbox{is in} \ \CoFib\},\ \ \ \ \ \mathcal F_1:=\{X\in \mathcal A \ \mid \ X\longrightarrow 0 \ \mbox{is in} \ \TFib\}
\\
&\mathcal T_2:=\{X\in \mathcal A \ \mid \ 0\longrightarrow X \ \mbox{is in} \ \TCoFib\},\ \ \ \mathcal F_2:=\{X\in \mathcal A \ \mid \ X \longrightarrow 0 \ \mbox{is in} \ \Fib\}.
\end{aligned}
\end{equation*}

\vskip 5pt

$(2)$ \ Any torsion model structure on $\mathcal A$ is a bi-reflective model structure.

\vskip 5pt

$(3)$ \ Let $(\CoFib, \Fib, \Weq)$ be a torsion model structure given by a twin torsion pair $(\mathcal T_1, \mathcal F_1; \mathcal T_2, \mathcal F_2)$.
Then the class of cofibrant objects is $\mathcal T_1$, the class of fibrant objects is $\mathcal F_2$, and the class of trivial objects is
$$\mathcal T_2*\mathcal F_1: = \{X\in \mathcal A \mid \exists\ \text{an exact sequence } 0\longrightarrow T_2 \longrightarrow X \longrightarrow F_1 \longrightarrow 0, \ T_2 \in \mathcal T_2, \ F_1 \in \mathcal F_1\}.$$
The homotopy category $\Ho(\mathcal A)$ is a quasi-abelian category, and it is equivalent to $\mathcal T_1\cap\mathcal F_2$ as a category.
In particular, every torsion class and every torsion-free class can be realized as the homotopy category of a torsion model structure.}

\vskip 10pt

On an abelian category $\mathcal{A}$, a TTF model structure is in general not a torsion model structure, although both of them are related with torsion pairs. See Remark \ref{both-TTF-and-torsion}.

\vskip 5pt

Torsion model structures on a poset, on the category $G\text{-}\mathsf{Set}$ of $G$-sets, and on the category $\mathsf{TopGrp}$ of topological groups, are constructed in Propositions \ref{prop:two-cut-points-poset-model}, \ref{prop:gset-normal-subgroup-torsion-model} and  \ref{prop:topological-group-torsion-model-structure}, respectively.
A torsion model structure which has neither cofibrant nor fibrant objects is also constructed in  Example \ref{exm:Z3-no-enough-CF}.
Some information on these torsion model structures is listed as follows:

\begin{table}[h!]
\begin{center}
\begin{tabular}{|c|c|c|c|c|}
\hline
\textbf{Category} & \textbf{\makecell[c]{Initial, terminal\\ and zero objects}} & \textbf{\makecell[c]{Has enough\\ cofibrant and \\fibrant objects}} & \textbf{Bi-reflectivity} & \textbf{\makecell[c]{Homotopy \\category}}\\
\hline
Poset $(P, \le)$ & \makecell[c]{In general, no initial object \\ no terminal object} & True & bi-reflective & $\mathcal C\cap\mathcal F$\\
\hline
$G\text{-}\mathsf{Set}$ & \makecell[c]{Has initial object\\ has terminal object \\ but no zero object}& True & Not bi-reflective & $\mathcal C\cap\mathcal F$\\
\hline
$\mathsf{TopGrp}$ & Has zero object & True & Not bi-reflective & $\mathcal C\cap\mathcal F$\\
\hline
Poset $(\mathbb Z^3, \le)$ & \makecell[c]{No initial object \\ no terminal object} & False & Not bi-reflective & Poset $(\mathbb Z, \le)$\\
\hline
\end{tabular}
\end{center}
\end{table}

\subsection{\bf The organization} The paper is organized as follows.

\vskip5pt

1 \ Introduction

2 \ Preliminaries

\hskip15pt 2.1 \ Model structures

\hskip15pt 2.2 \ Weak factorization systems

\hskip15pt 2.3 \ Orthogonal model structures

\hskip15pt 2.4 \ Localization of categories

3  \ Homotopy categories of orthogonal model structures

4 \  TTF model structures

5 \  Bi-reflective model structures

\hskip15pt 5.1 \  Orthogonal factorization systems via reflective subcategories

\hskip15pt 5.2 \  Bi-reflective pairs and  bi-reflective model structures

\hskip15pt 5.3 \ Examples

6 \ Torsion model structures

\hskip15pt 6.1 \  Torsion model structures on abelian categories

\hskip15pt 6.2 \ Torsion model structures on posets

\hskip15pt 6.3 \ Torsion model structures on $G\text{-}\mathsf{Set}$

\hskip15pt 6.4 \ Torsion model structures on $\mathsf{TopGrp}$

\hskip15pt 6.5 \ A torsion model structure without enough fibrant or cofibrant objects


\section{\bf Preliminaries}

We will recall the main notion and facts on (orthogonal) model structures, (weak, orthogonal) factorization systems, and localization of categories.

\subsection{\bf Model structures}

Let $\mathcal A$ be a category. A morphism $f:X\longrightarrow Y$ is a {\it retract} of a morphism
$g:X'\longrightarrow Y'$ if there is a commutative diagram
\[\xymatrix@R=0.4cm{X \ar[r]^{\varphi_1} \ar[d]_{f} & X' \ar[r]^{\psi_1} \ar[d]_{g} & X \ar[d]^{f} \\ Y \ar[r]^{\varphi_2} & Y' \ar[r]^{\psi_2}& Y}\]
such that $\psi_1\circ \varphi_1=\mathrm{Id}_X$ and
$\psi_2\circ \varphi_2=\mathrm{Id}_Y$.

\vskip 5pt

A class $\mathcal S$ of morphisms in $\mathcal A$ is \emph{closed under retracts}, provided that every retract of a morphism in $\mathcal S$ is in $\mathcal S$.

\vskip 5pt

The class $\mathcal S$  has the \emph{two-out-of-three property}, provided that for two composable morphisms $f: X \longrightarrow Y$ and $g: Y\longrightarrow Z$, whenever two of $f$, $g$, and $g\circ f$ are in $\mathcal S$, so is the third.

\begin{defn}[\cite{Q1,Q2}]\label{def:model-structure} \ A \emph{model structure} on a category $\mathcal A$ is a triple $(\CoFib,\Fib,\Weq)$ of classes of morphisms, satisfying the following axioms {\rm (CM1)} - {\rm (CM4)}:

\vskip5pt

{\rm (CM1)} \ {\rm(}\emph{Two-out-of-three axiom}{\rm)} \ The class $\Weq$ has the two-out-of-three property.

\vskip5pt

{\rm (CM2)} \ {\rm(}\emph{Retract axiom}{\rm)} \ The classes $\CoFib$, $\Fib$, and $\Weq$ are closed under retracts.

\vskip5pt

{\rm (CM3)} \ {\rm(}\emph{Lifting axiom}{\rm)} \  For every commutative square
$$\xymatrix@R=0.4cm{A\ar[r]^-a \ar[d]_-i & X \ar[d]^-p \\
B\ar[r]^-b \ar@{-->}[ru]^-s & Y }$$
with $i\in\CoFib$ and $p\in\Fib$, if either $i\in\CoFib\cap\Weq$ or $p\in\Fib\cap\Weq$, then there exists a morphism $s: B\longrightarrow X$ such that $a=s\circ i$ and $b=p\circ s$.
Such a morphism $s$ is called {\it a lifting}.

\vskip5pt

{\rm (CM4)} \ {\rm(}\emph{Factorization axiom}{\rm)} \ Every morphism $f:X\longrightarrow Y$ admits factorizations $f=p\circ i=q\circ j$, where $i\in\CoFib\cap\Weq$, \ $p\in\Fib$, \ $j\in\CoFib$, and $q\in\Fib\cap\Weq$.
\end{defn}

If  this is the case, then the morphisms in  $\CoFib, \ \Fib, \ \Weq$ are called \emph{cofibrations}, \emph{fibrations}, and \emph{weak equivalences}, respectively. Put $\TCoFib:=\CoFib\cap\Weq$ and $\TFib:=\Fib\cap\Weq$.
The morphisms in  $\TCoFib$ and $\TFib$ are called \emph{trivial cofibrations} and \emph{trivial fibrations}, respectively.

\vskip5pt

\begin{defn}[\cite{Q1}] \label{def:fibrantobject} \ Let  $(\CoFib$,  $\Fib$,  $\Weq)$ be a model structure on category $\mathcal A$.

\vskip5pt

$(1)$ \ Suppose that $\mathcal A$ has an initial object $\emptyset$.  An object $C$ is {\it a cofibrant object},  if $\emptyset\longrightarrow C$ is a cofibration.
Denote by $\mathcal C$ the class of cofibrant objects.

\vskip5pt

$(2)$ \ Suppose that  $\mathcal A$ has a terminal object $*$.  An object $F$ is {\it a fibrant object}, if  $F\longrightarrow *$ is a fibration.
Denote by $\mathcal F$ the class of fibrant objects.

\vskip5pt

$(3)$ \ Suppose that $\mathcal A$ has a zero object $0$. An object $X$ is {\it trivial} if $0 \longrightarrow X $ is a weak equivalence, or, equivalently,
$X\longrightarrow 0$ is a weak equivalence. Denote by $\mathcal W$ the class of trivial objects.

\vskip5pt

An object is {\it trivially cofibrant} (respectively, {\it trivially fibrant}) if it is both trivial and  cofibrant (respectively, fibrant).
\end{defn}

\vskip5pt

We also need to treat model structures on those categories, which have no initial objects (respectively, terminal objects).
In this case, the notion of cofibrant objects  (respectively, fibrant objects)
will be introduced in Definition \ref{defn:cofibrant-fibrant-objects}, and this notion will be important in this paper.

\vskip5pt

A morphism $u:A\longrightarrow B$ has the \emph{left lifting property} (LLP) with respect to a morphism $f:X\longrightarrow Y$, equivalently, $f$ has the \emph{right lifting property} (RLP) with respect to $u$, if every commutative square
$$\xymatrix@R=0.5cm@C=0.8cm{A\ar[r]^-x \ar[d]_-u & X \ar[d]^-f \\
B\ar[r]_-y \ar@{-->}[ru]^-d & Y }$$
admits a lifting $d:B\longrightarrow X$ such that $d\circ u=x$ and $f\circ d=y$.

\vskip 5pt

For classes of morphisms $\mathcal S$, denote by ${\rm LLP}(\mathcal S)$ the class of morphisms which have the left lifting property with respect to every morphism in $\mathcal S$, and
by ${\rm RLP}(\mathcal S)$ the class of morphisms which have the right lifting property with respect to every morphism in $\mathcal S$.

\vskip 5pt

For classes of morphisms $\mathcal S$ and $\mathcal T$, set $\mathcal T\circ\mathcal S:=\{f=p\circ u \mid u\in\mathcal S, \ p\in\mathcal T\}$.

\vskip 5pt

Denote by $\operatorname{Iso}(\mathcal A)$ the class of isomorphisms in $\mathcal A$.

\vskip 5pt

We need the following facts.

\begin{fact} \label{elementpropmodel} \ {\rm (\cite{Q1, Q2})} \ Let  $(\CoFib$, \ $\Fib$,  \ $\Weq)$ be a
model structure on category $\mathcal A$. Then

$(1)$ \ \ Any two classes of $\CoFib$, \  $\Fib$, \  $\Weq$ determine the third uniquely. Namely,
$$\CoFib = {\rm LLP}(\TFib), \  \TFib = {\rm RLP}(\CoFib); \ \ \Fib = {\rm RLP}(\TCoFib), \ \TCoFib = {\rm LLP}(\Fib)$$ and
$\Weq = \TFib\circ \TCoFib.$

$(2)$ \ Both the classes $\CoFib$ and $\Fib$ are closed under compositions.

$(3)$ \ $\CoFib \cap \Fib \cap \Weq = \operatorname{Iso}(\mathcal A)$.

$(4)$ \ Cofibrations are closed under pushouts, i.e., if the following commutative square
$$\xymatrix@R=0.4cm{\bullet\ar[r]^-i\ar[d] & \bullet \ar@{.>}[d] \\
\bullet \ar@{.>}[r]^-{i'} & \bullet}
$$
is a pushout square with $i\in \CoFib$, then $i'\in \CoFib$.

Also, trivial cofibrations are closed under pushouts.

$(4')$ \ Fibrations are closed under pullbacks$;$ and trivial fibrations are closed under pullbacks.\end{fact}

\subsection{\bf Weak factorization systems} \ Let $\operatorname{Mor}(\mathcal A)$ and $\operatorname{Iso}(\mathcal A)$ denote the class of morphisms and the class of isomorphisms in $\mathcal A$, respectively.

\begin{defn} {\rm(\cite[Appendix D]{Joy})} \label{defn:weak factorization system}\ A pair $(\mathcal L, \mathcal R)$ of classes of morphisms in a category $\mathcal A$ is a \emph{weak factorization system}, provided that the following conditions are satisfied$:$
\begin{enumerate}
    \item \ $\operatorname{Mor}(\mathcal A) = \mathcal R \circ \mathcal L.$
    \item \ $\mathcal L= {\rm LLP}(\mathcal R)$ and $\mathcal R={\rm RLP}(\mathcal L)$.
\end{enumerate}
\end{defn}

If this is the case, then one has $\mathcal L \cap \mathcal R = \operatorname{Iso}(\mathcal A)$.

\vskip5pt

\begin{prop}\label{prop:weak factorization system-equivalent}{\rm (\cite[Proposition D.1.9]{Joy})} \
A pair $(\mathcal L,\mathcal R)$ of classes of morphisms in a category $\mathcal A$ is a weak factorization system if and only if the following conditions are satisfied$:$

{\rm (i)} \ $\operatorname{Mor}\mathcal A=\mathcal R\circ\mathcal L;$

{\rm (ii)} \ every morphism in $\mathcal L$ has the left lifting property with respect to every morphism in $\mathcal R;$

{\rm (iii)} \ $\mathcal L$ and $\mathcal R$ are closed under retracts.
\end{prop}

\begin{prop}\label{prop:weak factorization system-model-structure} {\rm(\cite[Proposition E.1.11]{Joy})} \ For a triple $(\CoFib,\Fib,\Weq)$ of classes of morphisms in $\mathcal A$, write $\TFib:=\Fib\cap\Weq$ and $\TCoFib:=\CoFib\cap\Weq$. Then $(\CoFib,\Fib,\Weq)$ is a model structure if and only if $(\CoFib,\TFib)$ and $(\TCoFib,\Fib)$ are weak factorization systems,  $\Weq$ has the two-out-of-three property, and $\Weq$ is closed under retracts.
\end{prop}

We need the following property of weak factorization systems.

\begin{prop}\label{prop:retract-argument} \ {\rm (\cite{Joy}, \cite{Sat})} \ Let $(\CoFib,\TFib)$ and $(\TCoFib,\Fib)$ be weak factorization systems on category $\mathcal A$
such that $\TCoFib\subseteq\CoFib$. Put $\Weq: =\TFib\circ\TCoFib$. Then $$\CoFib\cap\Weq=\TCoFib, \ \ \Fib\cap\Weq=\TFib.$$
\end{prop}

\begin{proof}  \ This result appears in the proof of \cite[Proposition E.1.11]{Joy} and is for finitely complete and cocomplete categories;
and in \cite[Lemma 2.1]{Sat} this result is stated for finitely complete and cocomplete categories. For convenience we include a proof.

\vskip5pt The inclusion $\TCoFib\subseteq\CoFib\cap\Weq$ is immediate. For $f:X\longrightarrow Y$ in $\CoFib\cap\Weq$, factorize $f=p\circ i$ with $i:X\longrightarrow Z$ in $\TCoFib$ and $p:Z\longrightarrow Y$ in $\TFib$. The lifting property gives $s:Y\longrightarrow Z$ in the diagram
\[
\xymatrix@R=0.5cm@C=0.8cm{
X\ar[d]_f\ar[r]^i & Z\ar[d]_p\\
Y\ar@{=}[r]\ar@{-->}[ur]^s & Y
}
\]
such that $s\circ f=i$ and $p\circ s=\Id_Y$. Thus,  $f$ is a retract of $i$:
\[
\xymatrix@R=0.5cm@C=0.8cm{
X\ar@{=}[r]\ar[d]_f & X\ar@{=}[r]\ar[d]_i & X\ar[d]_f\\
Y\ar[r]^s & Z\ar[r]^p & Y.
}
\]
Since $\TCoFib$ is closed under retracts (cf. Proposition \ref{prop:weak factorization system-equivalent}),  $f\in\TCoFib$. This proves $\CoFib\cap\Weq=\TCoFib$.
Dually, one has $\Fib\cap\Weq=\TFib.$
\end{proof}

\subsection{\bf Orthogonal model structures} We fix the following notations throughout this paper.

\vskip5pt

For morphisms $u$ and $f$, we write $u\perp f$ provided that  $u$ has the left lifting property with respect to $f$,  and the lifting is unique.

\vskip5pt

Write $\mathcal S\perp\mathcal T$ if $u\perp f$ for every $u\in\mathcal S$ and $f\in\mathcal T$.

\vskip5pt

For a class $\mathcal S$ of morphisms, let ${}^{\perp}\mathcal S := \{f\mid f\perp u,\ \forall \ u\in\mathcal{S}\}$ and $\mathcal S^{\perp}:= \{f\mid u\perp f, \ \forall \ u\in\mathcal{S}\}$.

\begin{defn} {\rm(\cite{FK}, \cite{Riehl})} \ A pair $(\mathcal E,\mathcal M)$ of classes of morphisms in a category $\mathcal A$ is an {\it orthogonal factorization system}, provided that the following conditions are satisfied$:$
\begin{enumerate}
    \item \ $\operatorname{Mor}\mathcal A=\mathcal M\circ\mathcal E$.
    \item \ $\mathcal E={}^{\perp}\mathcal M$ and $\mathcal M=\mathcal E^{\perp}$.
\end{enumerate}
\end{defn}

It is clear that an orthogonal factorization system is a weak factorization system.

\begin{prop} {\rm(\cite[Lemma 1.8]{Riehl}, \cite[Proposition C.0.23]{Joy})} \label{prop:OFS-criterion} \ Let $(\mathcal E,\mathcal M)$ be classes of morphisms in a category $\mathcal A$. The following conditions are equivalent$:$

\vskip5pt

$(1)$ \  $(\mathcal E,\mathcal M)$ is an orthogonal factorization system.

\vskip5pt

$(2)$ \ $\operatorname{Mor}\mathcal A=\mathcal M\circ\mathcal E$, \ $\mathcal E\perp\mathcal M$, and $\mathcal E$ and $\mathcal M$ are closed under isomorphisms.

\vskip5pt

$(3)$ \ {\rm (i)} \  Every morphism admits a factorization $f=m\circ e$ with $e\in\mathcal E$ and $m\in\mathcal M;$ and this factorization is unique up to a unique isomorphism, i.e., if $f=m_1\circ e_1=m_2\circ e_2$ with $e_1,e_2 \in \mathcal{E}$ and $m_1,m_2 \in \mathcal{M}$, then there exists a unique isomorphism $\theta$ such that the following diagram commutes$:$
\[
\xymatrix@R=0.5cm@C=0.6cm{
  {} & Z\ar[dr]^-{m_1}\ar@{-->}[dd]^-{\exists !\theta} & {} \\
  X\ar[ur]^-{e_1}\ar[dr]_-{e_2} & {} & Y \\
  {} & Z'\ar[ur]_-{m_2} & {}
}\]

\hskip10pt {\rm (ii)} \ $\mathcal E$ and $\mathcal M$ are closed under compositions.

\hskip10pt {\rm (iii)} \ $\operatorname{Iso}(\mathcal A) \subseteq \mathcal E\cap\mathcal M$.
\end{prop}

\begin{exm}
In $\mathsf{Set}$, and in any abelian category, the classes of epimorphisms and monomorphisms form an orthogonal factorization system, denoted by $(\operatorname{Epic},\operatorname{Monic})$.
\end{exm}

\begin{defn} \ {\rm(\cite{PT}, \cite{BG})} \label{defn:orthogonal-model-structure} \ A model structure $(\CoFib,\Fib,\Weq)$ on a category $\mathcal A$ is  an \emph{orthogonal model structure}, provided that
$\TCoFib\perp\Fib$ and $\CoFib\perp\TFib$.
\end{defn}

In fact, Pultr and Tholen \cite{PT} call an orthogonal model structure as \emph{a Quillen factorization system}; Balchin and Garner \cite{BG} called it a \emph{one-dimensional model category}.

\vskip5pt

\begin{prop} {\rm (\cite[Proposition~3.3]{PT})} \ \label{prop:weq-two-out-of-three} \ Let $\mathcal A$ be an arbitrary category,  $(\CoFib,\Fib,\Weq)$ a triple of classes of morphisms in $\mathcal A$.
Then   $(\CoFib, \Fib, \Weq)$ is an orthogonal model structure if and only if $(\CoFib, \Fib\cap\Weq)$ and $(\CoFib\cap\Weq, \Fib)$ are orthogonal factorization systems and $\Weq$ has the two-out-of-three property.
\end{prop}
\begin{proof} \ In fact, this result is not explicitly stated in \cite{PT}. For convenience we include a proof.

\vskip5pt

We only need to justify the sufficiency. Assume that $\Weq$ has the two-out-of-three property. By Proposition~\ref{prop:weak factorization system-model-structure}, it remains to prove that $\Weq$ is closed under retracts.
Let $f:X\longrightarrow Y$ be a retract of weak equivalence $f':X'\longrightarrow Y'$. Put $\TFib:=\Fib\cap\Weq$ and $\TCoFib:=\CoFib\cap\Weq$.
Factorizing $f=m\circ e$ and $f'=m'\circ e'$,  with $e\in\CoFib,$  $m\in\TFib$, $e'\in\TCoFib$ and $m'\in\TFib$,  we claim that there exists the following commutative diagram:
\[
\xymatrix@C=1cm@R=0.5cm{
X \ar[r]^{\varphi_1} \ar[d]_{e} & X' \ar[r]^{\psi_1} \ar[d]_{e'} & X \ar[d]^{e} \\
M \ar@{-->}[r]^{\varphi_0} \ar[d]_{m} & M' \ar@{-->}[r]^{\psi_0} \ar[d]_{m'} & M \ar[d]^{m} \\
Y \ar[r]^-{\varphi_2} & Y' \ar[r]^-{\psi_2} & Y
}
\]
In fact, by the commutative squares
\[
\xymatrix@R=0.7cm@C=0.8cm{
X\ar[d]_e\ar[r]^{e' \circ \varphi_1} & M'\ar[d]^{m'} \\
M\ar[r]_{\varphi_2 \circ m}\ar@{-->}[ur]^{\varphi_0} & Y',
}
\qquad\qquad\qquad\qquad
\xymatrix@R=0.7cm@C=0.8cm{
X'\ar[d]_{e'}\ar[r]^{e \circ \psi_1} & M\ar[d]^{m}\\
M'\ar[r]_{\psi_2 \circ m'}\ar@{-->}[ur]^{\psi_0} & Y
}
\]
there exist a unique morphism $\varphi_0$ and a unique morphism $\psi_0$ such that $$\varphi_0\circ e =e'\circ \varphi_1, \ \ m'\circ\varphi_0 = \varphi_2\circ m, \ \ \psi_0\circ e' =e\circ \psi_1, \ \ m\circ\psi_0 = \psi_2\circ m'.$$
This justifies the claim. And then one has
$$\psi_0\circ \varphi_0\circ e = \psi_0\circ e'\circ \varphi_1 = e\circ \psi_1\circ \varphi_1 = e, \ \ \ m\circ \psi_0\circ \varphi_0 = \psi_2\circ m'\circ \varphi_0 =  \psi_2\circ\varphi_2\circ m = m.$$
By the following commutative square
\[
\xymatrix@R=0.7cm@C=1.1cm{
X\ar[d]_e\ar[r]^{e} & M\ar[d]^{m} \\
M\ar[r]_{m}\ar@{-->}[ur]^{\psi_0\circ \varphi_0}_{\operatorname{Id}_M} & Y
}
\]
It follows from the definition of orthogonal factorization system that $\psi_0\circ \varphi_0=\operatorname{Id}_M$. Therefore $e$ is a retract of $e'$. It follows from Proposition \ref{prop:weak factorization system-equivalent} that $e\in\TCoFib$.
Thus $f=m\circ e\in\TFib\circ\TCoFib \subseteq \Weq\circ \Weq \subseteq \Weq$. Thus $\Weq$ is closed under retracts. \end{proof}

\subsection{\bf Localization of categories}
Let $\mathcal S$ be a class of morphisms in a category $\mathcal A$. By definition the localization category ${\mathcal A}[\mathcal S^{-1}]$ of $\mathcal A$ with respect to $\mathcal S$
is a category together with localization functor $\gamma:\mathcal A\longrightarrow {\mathcal A}[\mathcal S^{-1}]$,
which sends morphisms in $\mathcal S$ to isomorphisms, and has the universal property with this respect.

\vskip 5pt

By definition the \emph{homotopy category} (\cite{Q1}) of a model structure $(\CoFib,\Fib,\Weq)$ on category $\mathcal A$ is the localization category $\mathcal A[\Weq^{-1}]$, denoted by $\Ho(\mathcal A)$.

\vskip 5pt

By \cite{GZ}, the localization category ${\mathcal A}[\mathcal S^{-1}]$ always exists,  for any category $\mathcal A$ and for any class $\mathcal S$ of morphisms.
For later application, we need to recall the construction of ${\mathcal A}[\mathcal S^{-1}]$,
via the path category ${\rm Path}Q(\mathcal A, \mathcal S)$ of quiver $Q(\mathcal A, \mathcal S)$,
given by Gabriel and Zisman \cite{GZ}. See also Krause \cite{K}.
Vertices  of the quiver $Q(\mathcal A, \mathcal S)$ are the objects of $\mathcal A$.
For two vertices $X$ and $Y$,
the set of arrows of $Q(\mathcal A, \mathcal S)$ from $X$ to $Y$ is the disjoint union
$$\Hom_\mathcal A(X, Y) \ \overset\bullet \bigcup \ \mathcal S(Y, X)$$
where $\mathcal S(Y, X)$ is the set of morphisms
$g: Y\longrightarrow X$ in $\mathcal S$. Thus, if $g: Y\longrightarrow X$ is a morphism in $\mathcal S$, then in  $Q(\mathcal A, \mathcal S)$ one has
an arrow $g: Y\longrightarrow X$, and also a new added arrow $X\longrightarrow Y$, denoted by $g^{-1}: X\longrightarrow Y$. Thus, for each vertex $X$ there is a loop
${\rm Id}_X: X\longrightarrow X$ in $Q(\mathcal A, \mathcal S)$, corresponding to the identity morphism of $X$.

\vskip 5pt

By definition, objects of the path category ${\rm Path}Q(\mathcal A, \mathcal S)$ are vertices of $Q(\mathcal A, \mathcal S)$
(thus, objects of $\mathcal A$); and morphisms of ${\rm Path}Q(\mathcal A, \mathcal S)$ from $X$ to $Y$ are paths from $X$ to $Y$ in $Q(\mathcal A, \mathcal S)$.
Then ${\mathcal A}[\mathcal S^{-1}]$ is
exactly the quotient category of ${\rm Path}Q(\mathcal A, \mathcal S)$, with respect to
the equivalent relations generated by the following relations:

\vskip5pt

(1) \ For each object $X$ of $\mathcal A$,  one has ${\rm Id}_X\sim e_X,$ where ${\rm Id}_X$ is the morphism
of ${\rm Path}Q(\mathcal A, \mathcal S)$ given by the loop at $X$, and $e_X$ is the morphism given by the path of length $0$ at $X$;

\vskip5pt

(2) \ For composable morphisms $s: X\longrightarrow Y$ and $t: Y\longrightarrow Z$ of $\mathcal A$,  one has
$t\circ s \sim ts,$ where $t\circ s$ is the morphism  in ${\rm Path}Q(\mathcal A, \mathcal S)$ given by the concatenation of arrows $s$ and $t$, and $ts$ is the morphism given by arrow $ts$.

\vskip5pt

(3) \ For any morphism \ $g: Y\longrightarrow X$ in $\mathcal S$, one has $g\circ g^{-1} \sim {\rm Id}_X$ and $g^{-1}\circ g \sim {\rm Id}_Y$.

\vskip5pt Thus, if $s\in \mathcal S$ and $t\in {\rm Mor}(\mathcal A)$ such that $ts\in \mathcal S$ (respectively, $st\in \mathcal S$),
then $(ts)^{-1}\circ t=s^{-1}$ (respectively, $t \circ (st)^{-1}=s^{-1}$) in ${\mathcal A}[\mathcal S^{-1}]$.

\vskip5pt

In conclusion, the localization category ${\mathcal A}[\mathcal S^{-1}]$ has the same objects as $\mathcal A$,
and a morphism $\delta: X\longrightarrow Y$ in ${\mathcal A}[\mathcal S^{-1}]$ is an equivalence class. If $\mathcal S$ is closed under compositions,
then $\delta$ can be written as a zigzag $\delta=f_n{w_{n-1}}^{-1}\cdots {w_1}^{-1}f_1$, where
\[
\xymatrix{
X\ar[r]^-{f_1}& A_1& B_1\ar[l]_-{w_1}\ar[r]^-{f_2} & A_2& B_2\ar[l]_-{w_2}\ar[r]^-{f_3}& \ldots \ar[r]^-{f_{n}}& A_{n}=Y
}
\]
is a sequence of morphisms of $\mathcal A$ and $w_1,\ldots, w_{n-1}\in \mathcal S$ (it is of this form, since if necessary one can add identities at the beginning or at the end). The localization functor $\gamma: \mathcal A \longrightarrow {\mathcal A}[\mathcal S^{-1}]$
is the identity on objects, and sends a morphism $f$ of $\mathcal A$ to the equivalence class where $f$ lies in.

\section{\bf Homotopy categories of orthogonal model structures}

This section is devoted to study the  homotopy category of orthogonal model structures.

\vskip5pt

Many categories possibly have no initial (respectively, terminal) objects.
For example, the category of posets. For studying orthogonal model structures on such categories,
we need to define the notion of cofibrant (respectively, fibrant) objects, without the existence of initial or terminal objects.

\begin{defn}\label{defn:cofibrant-fibrant-objects} \ Let $(\CoFib,\Fib,\Weq)$ be an orthogonal model structure on $\mathcal A$.

\vskip5pt

$(1)$ \ An object $C$ is \emph{a cofibrant object} (with respect to this model structure) if for every trivial fibration $p:X\longrightarrow Y$ and every morphism $u:C\longrightarrow Y$, there exists a unique morphism $v:C\longrightarrow X$ such that $p\circ v=u$:
\[
\xymatrix@C=1.2cm@R=0.8cm{
& X\ar[d]^p \\
C\ar[r]^u\ar@{-->}[ur]^v & Y.
}
\]

Denote by $\mathcal C$ the full subcategory of cofibrant objects.

\vskip 5pt

$(2)$ \ The category $\mathcal A$ has \emph{enough cofibrant objects} (with respect to this model structure) if for every object $A\in\mathcal A$, there exists a trivial fibration $q:C\longrightarrow A$ with $C\in\mathcal C$.
Such a trivial fibration $q$ will be called a \emph{cofibrant replacement} of $A$.

\vskip5pt

$(1')$ \ Dually, an object $F$ is \emph{a fibrant object} (with respect to this model structure) if for every trivial cofibration $i:X\longrightarrow Y$ and every morphism $u:X\longrightarrow F$, there exists a unique morphism $v:Y\longrightarrow F$ such that $v\circ i=u$:
\[
\xymatrix@C=1.2cm@R=0.8cm{
X\ar[r]^u\ar[d]_i & F \\
Y.\ar@{-->}[ur]_v &
}
\]

Denote by $\mathcal F$ the full subcategory of fibrant objects.

\vskip 5pt

$(2')$ \  The category $\mathcal A$ has \emph{enough fibrant objects} (with respect to this model structure) if for every object $A\in\mathcal A$, there exists a trivial cofibration $r:A\longrightarrow F$ with $F\in\mathcal F$.
Such morphism $r$ will be called a \emph{fibrant replacement} of $A$.
\end{defn}

\begin{rem} \ $(1)$ \ In case a category $\mathcal A$ equipped with an orthogonal model structure has an initial object $\emptyset$, an object $C$ is cofibrant if and only if it is cofibrant in the usual sense of \cite{Q1}, i.e.,
$\emptyset\longrightarrow C$ is a cofibration. In this case, $\mathcal A$ always has enough cofibrant objects.

\vskip5pt

$(1')$ \ Dually, in case the category $\mathcal A$ equipped with an orthogonal model structure has a terminal object $*$,
an object $F$ is fibrant if and only if it is fibrant in the usual sense of \cite{Q1}, i.e.,  $F\longrightarrow *$ is a fibration.
In this case,  $\mathcal A$ always has enough fibrant objects.

\begin{fact}\label{fact:orthogonal-zero} \ Let $(\CoFib,\Fib,\Weq)$ be an orthogonal model structure on a category $\mathcal A$ with a zero object. Then every cofibrant-fibrant-trivial object is isomorphic to the zero object, i.e.,
$\mathcal C\cap\mathcal F\cap\mathcal W =0.$
\end{fact}

\begin{proof} \ Let $X\in\mathcal C\cap\mathcal F\cap\mathcal W$. Then $0\longrightarrow X$ is a trivial cofibration and $X\longrightarrow0$ is a trivial fibration. In the commutative square
\[
\xymatrix@R=0.4cm@C=0.7cm{
0\ar[r]\ar[d] & X\ar[d] \\
X\ar[r] & 0
}
\]
both $\Id_X$ and the $0$ are liftings. Thus $\Id_X=0$ and $X\cong0$.
\end{proof}

\end{rem}

\begin{prop}\label{prop:elementary-CF} \ Let $(\CoFib,\Fib,\Weq)$ be an orthogonal model structure on $\mathcal A$. Then one has

\vskip5pt

$(1)$ \ If $C$ is cofibrant and $i:C\longrightarrow C'$ is a cofibration, then $C'$ is cofibrant.

\vskip5pt

$(1')$ \  If $F$ is fibrant and $p:F'\longrightarrow F$ is a fibration, then $F'$ is fibrant.

\vskip5pt

$(2)$ \  If $C$ and $C'$ are cofibrant, then every trivial fibration $p:C\longrightarrow C'$ is an isomorphism.

\vskip5pt

$(2')$ \  If $F$ and $F'$ are fibrant, then every trivial cofibration $i:F\longrightarrow F'$ is an isomorphism.

\vskip5pt

$(3)$ \  If $C$ and $C'$ are cofibrant, then every morphism $f:C\longrightarrow C'$ is a cofibration.

\vskip5pt

$(3')$ \ If $F$ and $F'$ are fibrant, then every morphism $f:F\longrightarrow F'$ is a fibration.

\vskip5pt

$(4)$ \ Cofibrant replacement of an object $X$ is unique up to unique isomorphism, i.e., if $q:C\longrightarrow X$ and $q':C'\longrightarrow X$ are cofibrant replacements of an object $X$, then there is a unique isomorphism $a:C\longrightarrow C'$ such that $q'\circ a=q$.

\vskip5pt

$(4')$ \ Fibrant replacement of an object $X$ is unique up to unique isomorphism, i.e., if $r:X\longrightarrow F$ and $r':X\longrightarrow F'$ are fibrant replacements of an object $X$, then there is a unique isomorphism $b:F\longrightarrow F'$ such that $b\circ r=r'$.
\vskip5pt

$(5)$ \  If $X$ and $Y$ are cofibrant-fibrant objects, then every weak equivalence $u:X\longrightarrow Y$ is an isomorphism.
\end{prop}

\begin{proof} \ {\rm (1)} \ Let $p:X\longrightarrow Y$ be an arbitrary  trivial fibration and $u:C'\longrightarrow Y$ an arbitrary  morphism. Since $C$ is cofibrant,
there exists a unique morphism $v:C\longrightarrow X$ such that $p\circ v=u\circ i:$
\[\xymatrix@C=1.2cm@R=0.8cm{
& X\ar[d]^p \\
C\ar[r]^{u\circ i}\ar@{-->}[ur]^v & Y.
}
\]
Since $i\in\CoFib$ and $p\in\TFib$,  one has a unique lifting  $w:C'\longrightarrow X$ such that $p\circ w=u$ and $w\circ i=v:$
\[
\xymatrix@C=1.2cm@R=0.8cm{
C\ar[r]^v\ar[d]_i & X\ar[d]^p \\
C'\ar[r]^u\ar@{-->}[ur]^w & Y.
}
\]
We claim that the morphism $w: C' \longrightarrow X$ with  $p\circ w=u$ is unique. In fact, if $w': C' \longrightarrow X$ also satisfies $p\circ w'=u$, then $p\circ (w'\circ i) =u\circ i$.
By the uniqueness of $v$ one has $w'\circ i= v$.  By the uniqueness of the lifting one has $w' = w.$ By definition $C'$ is cofibrant.

\vskip 5pt

{\rm (2)} \ Let $p:C\longrightarrow C'$ be a trivial fibration between cofibrant objects. Since $C'$ is cofibrant, by definition there is a unique morphism $s: C' \longrightarrow C$ such that
$p\circ s=\Id_{C'}$:
\[
\xymatrix@C=1.2cm@R=0.8cm{
& C\ar[d]^p \\
C'\ar@{=}[r]\ar@{-->}[ur]^s & C'.
}
\]
Since both $s\circ p$ and $\Id_C$  make the diagram
\[
\xymatrix@C=1.2cm@R=0.8cm{
& C\ar[d]^p \\
C\ar[r]_p\ar@{-->}[ur]^{s\circ p}_{\Id_C} & C'
}
\]
commutes, by the uniqueness one has  $s\circ p=\Id_C$. Thus $p$ is an isomorphism.

\vskip 5pt

{\rm (3)} \ Factorize $f:C\longrightarrow C'$ as $f=qi$ with $i:C\longrightarrow Z$ in $\CoFib$ and $q:Z\longrightarrow C'$ in $\TFib$. By {\rm (1)}, $Z$ is cofibrant; and by {\rm (2)}, $q$ is an isomorphism. Hence $f$ is a cofibration.

\vskip 5pt

{\rm (4)} \ Let $q:C\longrightarrow X$ and $q':C'\longrightarrow X$ be cofibrant replacements of $X$. Then one has a unique morphism $a:C\longrightarrow C'$ such that $q'\circ a=q$. Also, there is a unique morphism
$b:C'\longrightarrow C$ such that $q\circ b=q'$:
\[
\xymatrix@C=1.2cm@R=0.8cm{
& C'\ar[d]^{q'} \\
C\ar[r]_q\ar@{-->}[ur]^a & X
}
\qquad \qquad \qquad \qquad
\xymatrix@C=1.2cm@R=0.8cm{
& C\ar[d]^q \\
C'\ar[r]_{q'}\ar@{-->}[ur]^b & X.
}
\]
Since both $ba$ and $\Id_C$ make  the following diagram
\[
\xymatrix@C=1.2cm@R=0.8cm{
& C\ar[d]^q \\
C\ar[r]_q\ar@{-->}[ur]^{b\circ a}_{\Id_C} & X
}
\]
commutes, by the uniqueness one has $ba=\Id_C$. Similarly, $ab=\Id_{C'}$. Thus $a$ is the unique isomorphism.

\vskip 5pt

{\rm (5)} \ Factorize $u:X\longrightarrow Y$ as $u = q\circ i$ with $i:X\longrightarrow Z$ in $\TCoFib$ and $q:Z\longrightarrow Y$ in $\TFib$. By ${\rm (1)}$ and ${\rm (1')}$, one has $Z\in\mathcal{C}\cap\mathcal{F}$.
By ${\rm (2')}$,  $i$ is an isomorphism;  by {\rm (2)}, $q$ is an isomorphism.  Thus $u$ is an isomorphism.

\vskip 5pt

The statements ${\rm (1')}, {\rm (2')}, {\rm (3')}$, and ${\rm (4')}$ are the dual of {\rm (1)}, {\rm (2)}, {\rm (3)}, and {\rm (4)}, respectively.
\end{proof}

\begin{rem} \ For a model structure on $\mathcal A$ with zero object, the assertions $(1)$ and $(1')$ in Proposition \ref{prop:elementary-CF} still hold true;
however, $(2), (2'), (3), (3'), (4), (4')$ and $(5)$ no longer hold true.
\end{rem}

\begin{thm}\label{thm:orthogonal-ho-CF} \ Let $(\CoFib,\Fib,\Weq)$ be an orthogonal model structure on $\mathcal A$ with enough cofibrant and enough fibrant objects. Then
the restriction of the localization functor $\gamma: \mathcal A\longrightarrow\Ho(\mathcal A)$ induces an equivalence $\mathcal C\cap\mathcal F\cong \Ho(\mathcal A)$ of categories.

\end{thm}

\begin{proof} \ For each object $X\in\mathcal A$, one fixes a cofibrant replacement $q_X: QX\longrightarrow X$, such that if $X\in\mathcal C$ then $QX=X$ and $q_X=\Id_X$.
By definition $q_X$ is a trivial fibration and $QX\in \mathcal C$.  For any morphism $f:X\longrightarrow Y$, since $QX$ is a cofibrant object and $q_Y$ is a  trivial fibration, it follows from
Definition \ref{defn:cofibrant-fibrant-objects}$(1)$ that there is a unique morphism $Qf: QX \longrightarrow QY$ such that the following diagram commutes:
\[
\xymatrix@C=1.2cm@R=0.8cm{
& QY\ar[d]^{q_Y} \\
QX\ar[r]_{f\circ q_X}\ar@{-->}[ur]^{Qf} & Y
}
\]
i.e., $q_Y\circ Qf=f\circ q_X$. In particular, if $f = \Id_X$, then  $q_X\circ Q\Id_X= q_X = q_X\circ \Id_{QX};$ and then the uniqueness in Definition \ref{defn:cofibrant-fibrant-objects}$(1)$ gives
$Q\Id_X=\Id_{QX}$.

\vskip5pt

Given morphisms $f:X\longrightarrow Y$ and $g:Y\longrightarrow Z$,  by construction one has $q_Z\circ Q(g\circ f)=(g\circ f)\circ q_X$ and
$$q_Z\circ Qg\circ Qf = g\circ q_Y\circ Qf = g\circ f \circ q_X$$
i.e., one has the commutative diagram
\[
\xymatrix@C=1.2cm@R=0.8cm{
&& QZ\ar[d]^{q_Z} \\
QX\ar[rr]_{g\circ f\circ q_X}\ar@{-->}[urr]^{Q(g\circ f)}_{Qg\circ Qf} && Z
}
\]
it follows from the uniqueness in Definition \ref{defn:cofibrant-fibrant-objects}$(1)$ that
$Q(g\circ f)=Qg\circ Qf$. Thus $Q: \mathcal A\longrightarrow\mathcal C$ is a functor. If $f\in\Weq$, then $Qf\in\Weq$,  by $q_Y\circ Qf=f\circ q_X$ and the Two-out-of-three axiom.

\vskip 5pt

Dually, one has a functor $R:\mathcal A\longrightarrow\mathcal F$, having the following properties:

(i) \ for each object $X\in\mathcal A$, there is a  trivial cofibration $r_X:X\longrightarrow RX$;

(ii) \ if $X\in\mathcal F$, then $RX=X$ and $r_X=\Id_X$;

(iii) \ For  $f:X\longrightarrow Y$,  $Rf:RX\longrightarrow RY$ is the unique morphism such that
\[
\xymatrix@C=1.2cm@R=0.8cm{
X\ar[r]^{r_Y\circ f}\ar[d]_{r_X} & RY \\
RX\ar@{-->}[ur]_{Rf} &,
}
\]
commutes, i.e.,  $Rf\circ r_X=r_Y\circ f;$

(iv) \ $R$ preserves weak equivalences.

\vskip 5pt

Consider the composition of the functors: $\mathcal A \stackrel R \longrightarrow \mathcal F \stackrel i \hookrightarrow \mathcal A \stackrel Q \longrightarrow \mathcal C$, where $i: \mathcal F \hookrightarrow \mathcal A$ is the inclusion.
For each object $X\in \mathcal A$,  $RX$ is fibrant and $q _{_{RX}}:QRX\longrightarrow RX$ is a fibration. By Proposition~\ref{prop:elementary-CF}$(1')$,  $QRX$ is fibrant, and hence
$QRX\in\mathcal C\cap \mathcal F$. Thus one gets a functor   $QR:\mathcal{A}\longrightarrow \mathcal C\cap\mathcal F$.
If $f\in\Weq$, then $QRf$ is a weak equivalence between cofibrant-fibrant objects, and it is an isomorphism, by Proposition~\ref{prop:elementary-CF}(5).
By the universal property of $\gamma: \mathcal A \longrightarrow \Ho(\mathcal A)$, there is a unique functor $H:\Ho(\mathcal A)\longrightarrow\mathcal C\cap\mathcal F$ such that $H\circ \gamma=QR.$

\vskip 5pt

Consider the functors $G=\gamma \circ j: \mathcal C\cap\mathcal F \longrightarrow \Ho(\mathcal A)$,  where
$j:\mathcal C\cap\mathcal F\hookrightarrow\mathcal A$ is the inclusion. We claim that
$G$ is an equivalence of categories with quasi-inverse $H$.

\vskip 5pt

Since $Q$ and $R$ are the identity on $\mathcal C\cap\mathcal F$, one has $H\circ G=H \circ \gamma \circ j=QR\circ j=\Id_{\mathcal C\cap\mathcal F}.$
It remains to construct a natural isomorphism $\alpha: \Id_{\Ho(\mathcal A)}\longrightarrow G\circ H$ of functors.

\vskip 5pt

For every $X\in\mathcal A$, define a morphism $\alpha_X: X\longrightarrow (G\circ H)(X)$ in $\Ho(\mathcal A)$ by
$$\alpha_X:=\gamma(q _{_{RX}})^{-1}\circ \gamma(r_X): X\longrightarrow (G\circ H)(X)=(QR)X.$$ This is an isomorphism in $\Ho(\mathcal A)$, for each object $X\in \mathcal A$.
It suffices to prove the square

\[
\xymatrix@C=1cm@R=0.8cm{
X\ar[r]^{f}\ar[d]_{\alpha_X} & Y\ar[d]^{\alpha_Y}\\
(QR)X\ar[r]^{(G\circ H)f} & (QR)Y}
\]
commutes for any morphism $f: X\longrightarrow Y$ in $\Ho(\mathcal A)$.

\vskip5pt

For this purpose, first let $f: X\longrightarrow Y$ be a morphism in $\mathcal A$. Consider the commutative diagram in $\mathcal A$:
\[
\xymatrix@C=1.4cm@R=0.8cm{
X\ar[r]^{r_X}\ar[d]_f & RX\ar[d]^{Rf} & QRX\ar[l]_{q _{_{RX}}}\ar[d]^{QRf} \\
Y\ar[r]^{r_Y} & RY & QRY\ar[l]_{q _{_{RY}}} .
}
\]

\vskip5pt
\noindent Applying $\gamma$ one has $\gamma(Rf)\circ \gamma(r_X)=\gamma(r_Y)\circ \gamma(f)$, and $\gamma(q_{_{RY}})\circ \gamma(QRf)=\gamma(Rf)\circ \gamma(q_{_{RX}})$. Thus
\begin{align*}(G\circ H)(\gamma(f))\circ \alpha_X & =\gamma(QRf)\circ \gamma(q _{_{RX}})^{-1}\circ \gamma(r_X) =\gamma(q _{_{RY}})^{-1}\circ \gamma(Rf) \circ \gamma(r_X) \\ & = \gamma(q _{_{RY}})^{-1}\circ \gamma(r_Y)\circ \gamma(f)=\alpha_Y\circ \gamma(f).\end{align*}

\vskip5pt

In general, notice that every morphism of $\Ho(\mathcal A)$ can be represented by a zigzag
\[
\xymatrix@C=1.0cm{
X\ar[r]^-{f_1} & A_1 & B_1\ar[l]_-{w_1}\ar[r]^-{f_2} & A_2 & B_2\ar[l]_-{w_2}\ar[r]^-{f_3} & \cdots & B_{n-1}\ar[l]_-{w_{n-1}}\ar[r]^-{f_n} & A_n=Y
}
\]
where $w_i:B_i\longrightarrow A_i$ belongs to $\Weq$ for $1\leq i\leq n-1$. The remaining arrows are $f_1:X\longrightarrow A_1$ and $f_{i+1}:B_i\longrightarrow A_{i+1}$ for $1\leq i\leq n-1$ (if necessary, one can add identity morphisms).
For every arrow $h:U\longrightarrow V$ in this zigzag, one has $(G\circ H)(\gamma(h))\circ \alpha_{_U}=\alpha_{_V}\circ \gamma(h)$.
In particular, for $w_i:B_i\longrightarrow A_i$ one has
\[
(G\circ H)(\gamma(w_i))\circ\alpha_{_{B_i}}
=\alpha_{_{A_i}}\circ\gamma(w_i).
\]
Since $w_i\in\Weq$, $\gamma(w_i)$ is invertible in $\Ho(\mathcal A)$, and hence $(G\circ H)(\gamma(w_i)^{-1})\circ\alpha_{_{A_i}} =\alpha_{_{B_i}}\circ\gamma(w_i)^{-1}.$
Thus the following diagram commutes in $\Ho(\mathcal A)$:
\[
\xymatrix@C=1.4cm{
X\ar[r]^{\gamma(f_1)}\ar[d]_{\alpha_{_X}} & A_1\ar[d]^{\alpha_{_{A_1}}} & B_1\ar[l]_{\gamma(w_1)}\ar[d]^{\alpha_{_{B_1}}}\ar[r]^{\gamma(f_2)} &  \cdots & B_{n-1}\ar[l]\ar[d]^{\alpha_{B_{n-1}}}\ar[r]^{\gamma(f_n)} & Y\ar[d]^{\alpha_{_Y}} \\
GHX\ar[r]^{GH(\gamma(f_1))} & GHA_1 & GHB_1\ar[l]_{GH(\gamma(w_1))}\ar[r]^{GH(\gamma(f_2))} & \cdots & GHB_{n-1}\ar[l]\ar[r]^{GH(\gamma(f_n))} & GHY.
}
\]
Thus, for any morphism $f:X\longrightarrow Y$ in $\Ho(\mathcal A)$ represented by this zigzag, one has $(G\circ H)(f)\circ\alpha_{_X}=\alpha_{_Y}\circ f.$
This completes the proof.
\end{proof}

\section{\bf TTF model structures}

In this section we will establish a one to one correspondence between TTF model structures and TTF triples. All these model structures are orthogonal model structures, on abelian categories.

\vskip 5pt

Let $\mathcal A$ be an abelian category.
Recall that a \emph{torsion pair} (\cite{Dickson}) in $\mathcal A$ is a pair $(\mathcal T,\mathcal F)$ of full subcategories such that $\Hom_{\mathcal A}(T,F)=0$ for $T\in\mathcal T$ and $F\in\mathcal F$, and every object $X\in\mathcal A$
admits an exact sequence $0\longrightarrow T\longrightarrow X\longrightarrow F\longrightarrow0$ with $T\in\mathcal T$ and $F\in\mathcal F.$
In this case, $\mathcal T$ and $\mathcal F$ are  respectively called the torsion class  and the torsion-free class.

\vskip 5pt

A \emph{torsion-torsionfree triple}  (\cite{J1}), or simply, a TTF \emph{triple} in $\mathcal A$,  is a triple $(\mathcal C,\mathcal W,\mathcal F)$ of full subcategories
such that $(\mathcal C,\mathcal W)$ and $(\mathcal W,\mathcal F)$ are torsion pairs in $\mathcal A$.
The subcategory $\mathcal W$ in a TTF triple $(\mathcal C,\mathcal W,\mathcal F)$ is a Serre subcategory.

\vskip 5pt

\begin{defn}\label{defn:ttf-orthogonal-model} \ Let $\mathcal A$ be an abelian category. A TTF \emph{model structure} on $\mathcal A$ is an orthogonal model structure $(\CoFib,\Fib,\Weq)$ such that
\[
\CoFib=\{f\in {\rm Mor}(\mathcal A) \ \mid \ \operatorname{Coker}f\in\mathcal C\},
\qquad
\Fib=\{f\in {\rm Mor}(\mathcal A) \ \mid \ \Ker f\in\mathcal F\},
\]
where $\mathcal C$ and $\mathcal F$ are the classes of cofibrant objects and fibrant objects, respectively.
\end{defn}

\begin{lem} {\rm (\cite[Proposition 1.4]{EG})} \label{lem:torsion-pair-ofs} \ Let $(\mathcal T,\mathcal F)$ be a torsion pair in an abelian category $\mathcal A$. Put
\begin{align*}
\mathcal E& =\{f \ \text{is an epimorphism} \ \mid \ \Ker f\in\mathcal T\}, \ \ \
\mathcal M =\{f\in {\rm Mor}(\mathcal A) \ \mid \ \Ker f\in\mathcal F\}
\\
\mathcal E'&=\{f\in {\rm Mor}(\mathcal A) \ \mid \ \operatorname{Coker}f\in\mathcal T\}, \ \ \ \ \ \ \ \ \ \
\mathcal M'=\{f\text{ is a monomorphism} \ \mid \ \operatorname{Coker}f\in\mathcal F\}.
\end{align*}
Then

\vskip5pt

$(1)$ \ $(\mathcal E,\mathcal M)$ is an orthogonal factorization system.

\vskip5pt

$(1')$ \ $(\mathcal E', \mathcal M')$ is an orthogonal factorization system.
\end{lem}


\begin{proof} \ For convenience we include an alternative proof. To prove $(1)$,  by Proposition~\ref{prop:OFS-criterion}, it remains to prove $\operatorname{Mor}\mathcal A=\mathcal M\circ\mathcal E$ and $\mathcal E\perp\mathcal M$.

\vskip 5pt

Let $f:X\longrightarrow Y$ be a morphism with kernel $k:K\longrightarrow X$. Choose an exact sequence
$0\longrightarrow T_K\xra{u}K\xra{v}F_K\longrightarrow0$ with $T_K\in\mathcal T$ and $F_K\in\mathcal F$.
Let $e:X\longrightarrow L$ be the cokernel of $k \circ u:T_K\longrightarrow X$. Then there exists $\sigma:F_K\longrightarrow L$ such that $e \circ k=\sigma \circ v$.
Let $p:L\longrightarrow C$ be the cokernel of $\sigma$. Then there is $s:{\rm Im }f \longrightarrow C$ such that $p \circ e = s\circ \tilde{f}.$ See the following commutative diagram with exact rows and columns:
\[\xymatrix@C=0.7cm@R=0.9cm{
&T_K\ar@{=}[r]\ar@{^{(}->}[d]_-{u} & T_K\ar@{^{(}->}[d]_-{k\circ u} & {} & {} \\
 0\ar[r]& K\ar[r]^-{k}\ar@{->>}[d]_-{v} & X\ar[r]^-{\tilde{f}}\ar@{->>}[d]_-{e} & {\rm Im }f\ar@{-->}[d]_-{s}\ar[r]& 0 \\
 0\ar[r]& F_K\ar[r]^-{\sigma} & L\ar[r]^-{p} & C\ar[r]& 0.
}\]
Choose an Epi-Mono factorization of $f = j \circ \tilde{f}$. By the snake lemma, $s$ is an isomorphism. One has $f= j \circ \tilde{f}= j \circ (s^{-1}\circ p \circ e) $.
Put $m=j\circ s^{-1}\circ p: L\longrightarrow Y$. Then $f=m\circ e$, where $\Ker e=T_K\in\mathcal T,\ \Ker m\cong \Ker p=F_K\in\mathcal F.$
Thus $e\in\mathcal E$ and $m\in\mathcal M$, by definition. This proves $\operatorname{Mor}\mathcal A=\mathcal M\circ\mathcal E$.

\vskip 5pt

Given a commutative square as the square at the right hand side below, with $e:A\longrightarrow B$ in $\mathcal E$ and $m:X\longrightarrow Y$ in $\mathcal M$.
By definition $\Ker e \in \mathcal T$ and $\Ker m \in \mathcal F$.
Then one gets the following commutative diagram with exact rows:
\[
\xymatrix@C=1.1cm@R=0.9cm{
0\ar[r] & \Ker e\ar[r]^{\sigma}\ar@{-->}[d]_c & A\ar[r]^e\ar[d]_a & B\ar[d]^b\ar[r]\ar@{-->}[dl]_-{s} & 0\\
0\ar[r] & \Ker m\ar[r]^{\tau} & X\ar[r]^m & Y.
}
\]
As $\Ker e\in\mathcal T$ and $\Ker m\in\mathcal F$, one has $c=0$, and hence $a\circ \sigma=\tau \circ c=0$. Since $e$ is the cokernel of $\sigma$, there is a morphism $s:B\longrightarrow X$ such that $s\circ e=a$.
Since $e$ is epic and $m\circ s\circ e=m\circ a=b\circ e$, one has $m\circ s=b$.

\vskip 5pt

Assume that $s':B\longrightarrow X$ also satisfies $s'\circ e=a$. Since $e$ is epic and $s'\circ e = a = s\circ e$, one has $s' = s$.
Thus $\mathcal E\perp\mathcal M$.
By Proposition~\ref{prop:OFS-criterion}, $(\mathcal E,\mathcal M)$ is an orthogonal factorization system.

\vskip 5pt

Dually $(\mathcal E',\mathcal M')$ is an orthogonal factorization system.
\end{proof}

\begin{prop}\label{prop:ttf-triple-gives-orthogonal-model} \ Let $(\mathcal C, \mathcal W, \mathcal F)$ be a {\rm TTF} triple in an abelian category $\mathcal A$. Put
\begin{align*}
\CoFib&=\{f\in {\rm Mor}(\mathcal A) \ \mid \ \operatorname{Coker}f\in\mathcal C\}, \qquad
\Fib =\{f\in {\rm Mor}(\mathcal A) \ \mid \ \Ker f\in\mathcal F\},\\
\Weq&=\{f\in {\rm Mor}(\mathcal A) \ \mid \ \Ker f\in\mathcal W, \ \operatorname{Coker}f\in\mathcal W\}.
\end{align*}
Then $(\CoFib,\Fib,\Weq)$ is a {\rm TTF} model structure on $\mathcal A$, with $\mathcal{C}$ the class of cofibrant objects, $\mathcal{F}$ the class of fibrant objects, and $\mathcal{W}$ the class of trivial objects of $(\CoFib,\Fib,\Weq)$.
\end{prop}

\begin{proof} \ Since $(\mathcal{W}, \mathcal{F})$ is a torsion pair, $\mathcal{W}\cap \mathcal{F} = 0.$ Thus
$$\Fib\cap\Weq = \{\ f \ \text{is a monomorphism} \ \mid \ \operatorname{Coker} f\in\mathcal W\}.$$
Applying Lemma \ref{lem:torsion-pair-ofs}$(1')$ to the torsion pair $(\mathcal{C}, \mathcal{W})$, one knows that $(\CoFib, \Fib\cap\Weq)$ is an orthogonal factorization system.

\vskip5pt

Similarly, since $(\mathcal{C}, \mathcal{W})$ is a torsion pair, $\mathcal{C}\cap \mathcal{W} = 0.$ Thus
$$\CoFib\cap\Weq = \{\ f \ \text{is an epimorphism} \ \mid \ \operatorname{Ker} f\in\mathcal W\}.$$
Applying Lemma \ref{lem:torsion-pair-ofs}$(1)$ to the torsion pair $(\mathcal{W}, \mathcal{F})$, one knows that $(\CoFib\cap\Weq, \Fib)$ is an orthogonal factorization system.

\vskip 5pt

To see that $(\CoFib,\Fib,\Weq)$ is an orthogonal model structure, by Proposition \ref{prop:weq-two-out-of-three},
it remains to verify that $\Weq$ has the two-out-of-three property. In fact, for $f:X\longrightarrow Y$ and $g:Y\longrightarrow Z$, one has kernel-cokernel sequence (see e.g. \cite[Corollary 1.5]{I}):
\[
0\longrightarrow \Ker f\longrightarrow \Ker(g\circ f)\longrightarrow \Ker g
\longrightarrow \operatorname{Coker}f\longrightarrow
\operatorname{Coker}(g\circ f)\longrightarrow \operatorname{Coker}g
\longrightarrow0.
\]
Since $\mathcal W$ is a Serre subcategory, it follows that one has

(i) \ if $\Ker f\in\mathcal W, \ \Coker f\in\mathcal W, \ \Ker g\in\mathcal W$ and $\Coker g\in\mathcal W,$ then $\Ker (g\circ f)\in\mathcal W, \ \Coker (g\circ f)\in\mathcal W$;

(ii) \ if $\Ker f\in\mathcal W, \ \Coker f\in\mathcal W, \ \Ker (g\circ f)\in\mathcal W$ and $\Coker (g\circ f)\in\mathcal W,$ then $\Ker g\in\mathcal W, \ \Coker g\in\mathcal W$;

(iii) \ if $\Ker g\in\mathcal W, \ \Coker g\in\mathcal W, \ \Ker (g\circ f)\in\mathcal W$ and $\Coker (g\circ f)\in\mathcal W,$ then $\Ker f\in\mathcal W, \ \Coker f\in\mathcal W$.

\noindent This implies that $\Weq=\{f\in {\rm Mor}(\mathcal A) \ \mid \ \Ker f\in\mathcal W, \ \operatorname{Coker}f\in\mathcal W\}$ satisfies (CM1).
So far $(\CoFib,\Fib,\Weq)$ is an orthogonal model structure.

\vskip 5pt

Hence  $0 \longrightarrow X$ is a cofibration if and only if $X \in \mathcal C$.
Similarly, $Y \longrightarrow 0$ is a fibration if and only if $Y \in \mathcal F$.
Thus $\mathcal{C}$ is precisely the class of cofibrant objects of $(\CoFib,\Fib,\Weq)$, and $\mathcal{F}$ is precisely the class of fibrant objects of $(\CoFib,\Fib,\Weq)$.
By definition $(\CoFib,\Fib,\Weq)$ is a TTF model structure.

\vskip5pt

By $\Weq =\{f\in {\rm Mor}(\mathcal A) \ \mid \ \Ker f\in\mathcal W, \ \operatorname{Coker}f\in\mathcal W\}$, one sees that
$0 \longrightarrow W$ is a weak equivalence if and only if $W \in \mathcal W$. Thus, $\mathcal W$ is precisely the class of trivial objects of $(\CoFib,\Fib,\Weq)$.
\end{proof}

\vskip 5pt

For a full subcategory $\mathcal S$ of additive category $\mathcal A$, set
$$\mathcal S^{\perp_0}=\{X\in\mathcal A\mid\Hom(S,X)=0,  \forall  S\in\mathcal S\}, \ \ \ {}^{\perp_0}\mathcal S=\{X\in\mathcal A\mid\Hom(X,S)=0,  \forall  S\in\mathcal S\}.$$

\begin{prop}\label{prop:ttf-orthogonal-model-gives-ttf-triple} \ Let $(\CoFib,\Fib,\Weq)$ be a {\rm TTF} model structure on an abelian category $\mathcal A$.
Let \ $\mathcal C$, $\mathcal F$, and $\mathcal W$ denote its full subcategories of cofibrant objects, fibrant objects, and trivial objects, respectively. Then $(\mathcal C,\mathcal W,\mathcal F)$ is a {\rm TTF} triple.
\end{prop}

\begin{proof} \ First, we claim $\TFib=\{f \ \text{is a monomorphism}\ \mid \ \Coker f\in \mathcal C^{\perp_0}\}$.

\vskip 5pt

Since $(\CoFib, \TFib)$ and $(\operatorname{Epic},\operatorname{Monic})$ are orthogonal factorization systems and $\operatorname{Epic}\subseteq\CoFib$, it follows that $\TFib\subseteq\operatorname{Monic}$.
Let $p:X\longrightarrow Y$ lie in $\TFib$ and $h:C\longrightarrow \Coker p$ be an arbitrary morphism with $C\in\mathcal C$. Consider the pullback diagram
\[
\xymatrix@C=1.0cm@R=0.7cm{
0\ar[r] & X\ar@{=}[d]\ar[r]^j & Y'\ar[d]^b\ar[r]^q\ar@{-->}[ld]_s & C\ar[d]^h\ar[r] & 0\\
0\ar[r] & X\ar[r]^p & Y\ar[r]^-{g} & \Coker p\ar[r] & 0.
}
\]
By definition  $j\in \CoFib$. Since $(\CoFib,\TFib)$ is an orthogonal factorization system, there is a unique lifting $s: Y'\longrightarrow X$ such that $b=p\circ s$.
Thus $h\circ q=g\circ b=g\circ p\circ s=0$. Since $q$ is epic, $h=0$. Thus $\Coker p\in\mathcal C^{\perp_0}$.

\vskip 5pt

Conversely, let $f:X\longrightarrow Y$ be a monomorphism with $\operatorname{Coker}f\in\mathcal C^{\perp_0}$. We need to prove $f\in \TFib.$  It remains to prove that $f$ has the right lifting property with respect to $\CoFib$.
Assume that the square at the left hand side below is commutative with $i\in\CoFib$. Then one has the following diagram with exact rows:
\[
\xymatrix@C=1.0cm@R=0.7cm{
{} & A\ar[d]_a\ar[r]^i & B\ar[d]^b\ar[r]^-q\ar@{-->}[ld]_s & \operatorname{Coker}i\ar[d]^h\ar[r] & 0\\
0\ar[r] & X\ar[r]^f & Y\ar[r]^-g & \operatorname{Coker}f\ar[r] & 0
}
\]
with $\operatorname{Coker}i\in\mathcal C$. Since $\operatorname{Coker}f\in\mathcal C^{\perp_0}$, one has $h = 0$ and hence $g\circ b = h\circ q = 0$.
Thus, there is a morphism $s: B\longrightarrow X$ such that $f\circ s=b$.
Since $f\circ s\circ i = b\circ i = f\circ a$ and $f$ is a monomorphism, thus $a = s\circ i$. This proves that $f$ has the right lifting property with respect to $\CoFib$, and hence
$f\in\TFib$.

\vskip 5pt

This proves $\TFib=\{f \ \text{is a monomorphism}\ \mid \ \Coker f\in \mathcal C^{\perp_0}\}$.

\vskip 5pt

Now we prove that $(\mathcal{C},\mathcal{W})$ is a torsion pair. If $W\in \mathcal W$, the morphism $0\longrightarrow W$ lies in $\Weq$. By the definition of a TTF model, $0\longrightarrow W$ lies in $\Fib$,  and hence it lies in $\TFib$. By the claim one has $W \in \mathcal C^{\perp_0}$.
This proves that $\Hom_{\mathcal A}(C,W)=0$ with $C \in \mathcal C$ and $W \in \mathcal W$.

\vskip 5pt

If $W\in \mathcal C^{\perp_0}$, then $0\longrightarrow W$ lies in $\TFib\subseteq \Weq$. This implies $W\in\mathcal W$. Thus $\mathcal C^{\perp_0} \subseteq \mathcal W$.

\vskip 5pt

Factorizing $0\longrightarrow X$ with respect to $(\CoFib,\TFib)$ gives a trivial fibration $f: C_X\longrightarrow X$.
By the claim,  $f$ is a monomorphism and $\Coker f\in \mathcal C^{\perp_0}$.
Thus one has a short exact sequence $0\longrightarrow C_X\stackrel f \longrightarrow X\longrightarrow \Coker f\longrightarrow0$, where $C_X\in\mathcal C$ and $\Coker f\in \mathcal C^{\perp_0} \subseteq \mathcal W$.
Thus $(\mathcal C,\mathcal W)$ is a torsion pair.

\vskip 5pt

Dually, $(\mathcal W,\mathcal F)$ is a torsion pair. Thus $(\mathcal C,\mathcal W,\mathcal F)$ is a TTF triple. \end{proof}

The main result on TTF model structures is as follows.

\begin{thm}\label{thm:ttf-triples-ttf-orthogonal-model-structures} \ Let $\mathcal A$ be an abelian category. Then
$$\Phi:\{{\rm TTF} \ \mbox{triples in} \ \mathcal A\} \longrightarrow \{{\rm TTF} \ \mbox{model structures on} \ \mathcal A\},  \ (\mathcal C,\mathcal W,\mathcal F) \longmapsto (\CoFib,\Fib,\Weq)$$
is a bijection,  where
\begin{align*}
\CoFib&=\{f\in {\rm Mor}(\mathcal A) \ \mid \ \operatorname{Coker}f\in\mathcal C\}, \qquad
\Fib =\{f\in {\rm Mor}(\mathcal A) \ \mid \ \Ker f\in\mathcal F\}\\
\Weq&=\{f\in {\rm Mor}(\mathcal A) \ \mid \ \Ker f\in\mathcal W, \ \operatorname{Coker}f\in\mathcal W\}
\end{align*}
with the inverse $\Psi:(\CoFib, \Fib, \Weq) \longmapsto (\mathcal C,\mathcal W,\mathcal F)$, where \ $\mathcal C$, $\mathcal F$, and $\mathcal W$ are respectively the classes of cofibrant objects, fibrant objects, and trivial objects.
\vskip5pt

In this case the homotopy category $\Ho(\mathcal A)$ is an abelian category, and it is equivalent to $\mathcal C\cap\mathcal F$ as a category.
In particular, $\mathcal C\cap\mathcal F$ enjoys a structure of an abelian category.
\end{thm}

\begin{proof} \ By Propositions \ref{prop:ttf-triple-gives-orthogonal-model} and \ref{prop:ttf-orthogonal-model-gives-ttf-triple}, $\Phi$ and $\Psi$ are well-defined.

\vskip5pt

For any TTF triple $(\mathcal C,\mathcal W,\mathcal F)$, by Propositions \ref{prop:ttf-triple-gives-orthogonal-model},
$\mathcal{C}$, $\mathcal{F}$ and $\mathcal{W}$ are precisely the classes of cofibrant objects, fibrant objects and trivial objects of TTF model structure $\Phi(\mathcal C,\mathcal W,\mathcal F)$, respectively;
and then by definition $(\Psi\circ\Phi)(\mathcal C,\mathcal W,\mathcal F)=(\mathcal C,\mathcal W,\mathcal F).$

\vskip 5pt
Conversely, for any TTF model structure $(\CoFib, \Fib, \Weq)$, by definition one has
\[
\CoFib=\{f\in {\rm Mor}(\mathcal A) \ \mid \ \operatorname{Coker}f\in\mathcal C\},
\qquad
\Fib=\{f\in {\rm Mor}(\mathcal A) \ \mid \ \Ker f\in\mathcal F\}.
\]
By definition $\Psi(\CoFib, \Fib, \Weq) = (\mathcal C,\mathcal W,\mathcal F),$
where \ $\mathcal C$, $\mathcal F$, and $\mathcal W$ are respectively the classes of cofibrant objects, fibrant objects, and trivial objects, of $(\CoFib, \Fib, \Weq)$, and hence by definition one has
$$(\Phi\circ\Psi)(\CoFib, \Fib, \Weq) = \Phi(\Psi(\CoFib, \Fib, \Weq)) =  \Phi( (\mathcal C,\mathcal W,\mathcal F)) = (\CoFib', \Fib', \Weq')$$
where
\begin{align*}
\CoFib'&=\{f\in {\rm Mor}(\mathcal A) \ \mid \ \operatorname{Coker}f\in\mathcal C\} = \CoFib \\
\Fib' &=\{f\in {\rm Mor}(\mathcal A) \ \mid \ \Ker f\in\mathcal F\} = \Fib.
\end{align*}
By Fact~\ref{elementpropmodel}{\rm (1)} one has
$$(\Phi\circ\Psi)(\CoFib, \Fib, \Weq) = (\CoFib', \Fib', \Weq') = (\CoFib, \Fib, \Weq).$$
This justifies the bijection.

\vskip 5pt

By Theorem~\ref{thm:orthogonal-ho-CF}, one has $\Ho(\mathcal A) \simeq \mathcal C\cap\mathcal F$ as a category. Since $\mathcal W$ is a Serre subcategory of an abelian category $\mathcal A$ and
$$\Weq =\{f\in {\rm Mor}(\mathcal A) \ \mid \ \Ker f\in\mathcal W, \ \operatorname{Coker}f\in\mathcal W\}$$
it follows that $\mathcal A[\Weq^{-1}]$ is an abelian category (see \cite{Gabriel}). Thus
$\Ho(\mathcal A)= \mathcal A[\Weq^{-1}]$ is an abelian category. \end{proof}

\begin{rem} \ We include a sketch of a direct proof of the following interesting result:

\vskip5pt

Let $(\CoFib,\Fib,\Weq)$ be a {\rm TTF} model structure on an abelian category $\mathcal A$,
$\mathcal C$ and $\mathcal F$ the classes of cofibrant objects and fibrant objects, respectively. Then $\mathcal C\cap\mathcal F$ enjoys a structure of an abelian category.
We stress that this structure of abelian category on  $\mathcal C\cap\mathcal F$ is not directly inherited from the one on $\mathcal A$.

\vskip 5pt

First, we clarify kernels in $\mathcal C\cap\mathcal F$. For a morphism $f:X\longrightarrow Y$ in $\mathcal C\cap\mathcal F$, let $k:K\longrightarrow X$ be the kernel of $f$ in $\mathcal A$.
By Proposition \ref{prop:ttf-orthogonal-model-gives-ttf-triple},  $(\mathcal C,\mathcal W,\mathcal F)$ is a {\rm TTF} triple. Thus, there is an exact sequence
$0 \longrightarrow Q\overset{q}{\longrightarrow} K\overset{p}{\longrightarrow} M\longrightarrow 0$ with $Q \in \mathcal C$ and $M\in\mathcal{W}$. It is clear that $k\circ q: Q\longrightarrow X$ is a monomorphism in $\mathcal C\cap\mathcal F$.
We claim that $k\circ q: Q\longrightarrow X$ is the kernel of $f$ in $\mathcal C\cap\mathcal F$.

\vskip 5pt

In fact, since $X \in \mathcal F$ and  $\mathcal F$ is closed under subobjects, it follows that $K \in \mathcal F$ and $Q\in \mathcal F$. Thus
$Q \in \mathcal C\cap\mathcal F$.  For any morphism $u:U \longrightarrow X$ in $\mathcal{C}\cap\mathcal{F}$ with $f\circ u = 0$, since $k$ is the kernel of $f$, there exists a morphism $s:U\longrightarrow K$ such that $k\circ s = u$.
Since $U\in\mathcal{C}$ and $M\in\mathcal{W}$, one has $p\circ s = 0$.
Thus there exists a morphism $t:U\longrightarrow Q$ such that $q\circ t = s$:
\[\xymatrix{
  Q\ar@{^{(}->}[d]_-{q} & U\ar@{-->}[l]_-{t}\ar@{-->}[dl]_-{s}\ar[d]^-{u} & {} \\
  K\ar@{^{(}->}[r]^-{k}\ar@{->>}[d]_-{p} & X\ar[r]^-{f} & Y,\\
  M
}
\]
where $(k \circ q) \circ t=k \circ s = u$.
Thus $k\circ q:Q\longrightarrow X$ is the kernel of $f$ in $\mathcal C\cap\mathcal F$.

\vskip 5pt

Dually, for a morphism $f:X\longrightarrow Y$ in $\mathcal C\cap\mathcal F$,
let $c:Y\longrightarrow Z$ be cokernel of $f\in\mathcal{A}$ and $0 \longrightarrow N\overset{n}{\longrightarrow} Z\overset{r}{\longrightarrow} R\longrightarrow 0$ be the exact sequence with $N \in \mathcal W$ and $R\in\mathcal F$.
Then $r\circ c:Y\longrightarrow R$ is the cokernel of $f$ in $\mathcal C\cap\mathcal F$.

\vskip 5pt

Now we prove that the coimage in $\mathcal C\cap\mathcal F$ is canonically isomorphic to the image in $\mathcal C\cap\mathcal F$.
Let $e:X\longrightarrow \Coker k$ be the cokernel of $k$ in $\mathcal A$, and $m:\Ker c\longrightarrow Y$ be the kernel of $c$ in $\mathcal A$.
Since $\operatorname{can}:\Coker k\longrightarrow \Ker c$ is canonical isomorphism in $\mathcal A$, and since $X\in\mathcal{C}$ and  $Y\in\mathcal{F}$, it follows that $\Coker k\cong\Ker c\in \mathcal C\cap\mathcal F$.
We claim that $e$ is the cokernel of $k\circ q$ in $\mathcal C\cap\mathcal F$.

\vskip 5pt

In fact, let $d:X\longrightarrow \Coker (k\circ q)$ be the cokernel of $k\circ q$ in $\mathcal A$.
Applying the kernel-cokernel sequence for $q:Q \longrightarrow K$ and $k:K \longrightarrow X$ one gets an exact sequence in $\mathcal A$
$$0 = \Ker k \longrightarrow M= \Coker q\overset{\alpha}{\longrightarrow} \Coker (k\circ q) \overset{\beta}{\longrightarrow} \Coker k \longrightarrow 0.$$
See the following commutative diagram:
\[\xymatrix{
Q\ar@{=}[r]\ar@{^{(}->}[d]_-{q} & Q\ar@{^{(}->}[d]^-{k \circ q} & {} & {} & {} & {} \\
K\ar@{^{(}->}[r]^-{k}\ar@{->>}[d]_-{p} & X\ar[rrr]^-{f}\ar@{->>}[d]^-{d}\ar@{->>}[dr]^-{e} & {} & {} & Y\ar@{->>}[r]^-{c} & Z \\
M\ar@{^{(}->}[r]^-{\alpha} & {\rm Coker\ }(k \circ q)\ar@{->>}[r]^-{\beta} & {\rm Coker\ }k\ar@{.>}[r]^-{\rm can} & {\rm Ker\ }c\ar@{^{(}->}[ur]^m & {} & {}}
\]
For any morphism $w: X\longrightarrow W$ in $\mathcal C\cap\mathcal F$ with $w\circ (k\circ q) = 0$, there exists a morphism $j:\Coker(k\circ q)\longrightarrow W$ in $\mathcal A$ such that $j\circ d = w$.
Since $M\in\mathcal{W}$ and $(\mathcal W, \mathcal F)$ is a torsion pair, one has $j\circ\alpha = 0$.
Thus there exists a morphism $l:\Coker k\longrightarrow W$ such that $l\circ \beta = j$.
\[\xymatrix{
 {} & {} & M\ar@{^{(}->}[d]^-{\alpha} \\
 Q\ar@{^{(}->}[r]^-{k\circ q} & X\ar@{->>}[r]^-{d}\ar[d]^-{w} & {\rm Coker\ }(k\circ q)\ar@{-->}[dl]_-{j}\ar@{->>}[d]^-{\beta} \\
 {} & W & {\rm Coker} \ k.\ar@{-->}[l]^-{l}}
\]
Thus $l\circ e = l\circ(\beta\circ d) = j\circ d = w$. This proves the claim that $e$ is cokernel of $k\circ q$ in $\mathcal C\cap\mathcal F$.

\vskip 5pt

Dually, $m$ is the kernel of $r\circ c$ in $\mathcal C\cap\mathcal F$. This completes the proof. \hfill $\square$ \end{rem}

\vskip5pt

As an application of Theorem~\ref{thm:ttf-triples-ttf-orthogonal-model-structures}, one has the following equivalent description of TTF model structures.

\begin{prop}\label{prop:TTF-orthogoanl-equivalent-condition}
Let $\mathcal A$ be an abelian category, $(\CoFib,\Fib,\Weq)$ be an orthogonal model structure on $\mathcal A$. Then it is a {\rm TTF} model structure if and only if
\begin{align*}
\TCoFib&=\{ f \ \text{is an epimorphism}\ \mid \ \Ker f\in\mathcal W\},\\
\TFib&=\{ f \ \text{is a monomorphism}\ \mid \ \Coker f\in\mathcal W\},
\end{align*}
where $\mathcal W$ is the class of trivial objects.
\end{prop}
\begin{proof} \ Assume that $(\CoFib,\Fib,\Weq)$ is a {\rm TTF} model structure.  By Theorem~\ref{thm:ttf-triples-ttf-orthogonal-model-structures} one has
\begin{align*}
\CoFib&=\{ f\in {\rm Mor}(\mathcal A) \ \mid \ \operatorname{Coker}f\in\mathcal C\}, \qquad
\Fib =\{ f\in {\rm Mor}(\mathcal A) \ \mid \ \Ker f\in\mathcal F\},\\
\Weq&=\{ f\in {\rm Mor}(\mathcal A) \ \mid \ \Ker f\in\mathcal W, \ \operatorname{Coker}f\in\mathcal W\}
\end{align*}
and $(\mathcal C, \mathcal W, \mathcal F)$ is a TTF triple, where $\mathcal C$ and $\mathcal F$ are the classes of cofibrant objects and fibrant objects, respectively.
Since $(\mathcal C, \mathcal W)$ and $(\mathcal W, \mathcal F)$ are torsion pairs, it follows that $\mathcal C\cap \mathcal W = 0 = \mathcal W\cap \mathcal F$. Thus
\begin{align*}
\TCoFib &= \CoFib\cap\Weq=\{ f \ \text{is an epimorphism}\ \mid \ \Ker f\in\mathcal W\}\\
\TFib &= \Fib\cap\Weq=\{ f \ \text{is a monomorphism}\ \mid \ \Coker f\in\mathcal W\}.
\end{align*}

\vskip 5pt

Conversely, assume that $(\CoFib, \Fib, \Weq)$ is an orthogonal model structure such that $\TCoFib$ and $\TFib$ are given as above.
To see that $(\CoFib,\Fib,\Weq)$ is a {\rm TTF} model structure, by Theorem~\ref{thm:ttf-triples-ttf-orthogonal-model-structures} it suffices to
show that $(\mathcal{C},\mathcal{W})$ and $(\mathcal{W},\mathcal{F})$ are torsion pairs.

\vskip 5pt

We will prove that $(\mathcal{W}, \mathcal{F})$ is a torsion pair. For any morphism $f:W\longrightarrow F$ with $W\in\mathcal{W}$ and $F\in\mathcal{F}$, the morphism  $W\longrightarrow 0$ lies in $\TCoFib$, and $F\longrightarrow 0$ lies in $\Fib$.
By the Lifting axiom one has $f = 0$:
\[
\xymatrix@R=0.5cm@C=0.7cm{
W\ar[r]^f\ar[d] & F\ar[d] \\
0\ar@{=}[r]\ar@{-->}[ur] & 0.
}
\]
Thus $\operatorname{Hom}_{\mathcal{A}}(\mathcal{W},\mathcal{F}) = 0$. For any object $X\in\mathcal A$, factorize the morphism $X\longrightarrow 0$ as the composition $X\overset{p}{\longrightarrow} F\longrightarrow 0$ with $p\in \TCoFib$ and $F \in \mathcal{F}$.
Thus one gets an exact sequence
\[
0\longrightarrow \Ker p\longrightarrow X\overset{p}{\longrightarrow} F\longrightarrow 0
\]
with $\Ker p \in \mathcal{W}$ and $F \in \mathcal{F}$. This shows that $(\mathcal{W},\mathcal{F})$ is a torsion pair.

\vskip 5pt

Dually, $(\mathcal{C},\mathcal{W})$ is a torsion pair. This completes the proof.
\end{proof}

\begin{rem} \label{remTTF} \ Let $R$ be a ring. By \cite[2.1, 2.2]{J1}, TTF triples in $R\text{-}\operatorname{Mod}$ are in bijection with idempotent two-sided ideals of $R$. Thus, by Theorem~\ref{thm:ttf-triples-ttf-orthogonal-model-structures}
one has a bijection between idempotent two-sided ideals of $R$ and TTF model structures on $R\text{-}\operatorname{Mod}$.
\end{rem}

\section{\bf Bi-reflective model structures}

The aim of this section is to establish a one to one correspondence between bi-reflective model structures and bi-reflective pairs. All these model structures are orthogonal.

\subsection{\bf Orthogonal factorization systems via reflective subcategories}

In this subsection we will recall the process of getting an orthogonal factorization system via a reflective subcategory with the semi-left-exact reflector,
given by Cassidy--H\'ebert--Kelly \cite{CHK85}. This is a basis for introducing bi-reflective pairs and bi-reflective model structures, in the next subsection.

\begin{defn} \label{defn:strong-approximation} \ Let $\mathcal A$ be a category.

\vskip5pt

$(1)$ \ \ {\rm (\cite{F})} \  Let $\mathcal B$ be a full subcategory of $\mathcal A$. A morphism $u:A\longrightarrow B$ is a \emph{left $\mathcal B$-strong-approximation} of $A$,
provided that $B\in\mathcal B$, and for each object $B'\in\mathcal B$ and each morphism $f: A\longrightarrow B'$, there is a unique morphism $\overline f: B\longrightarrow B'$ with $f=\overline f\circ  u$:
\[\xymatrix@C=1.0cm@R=0.5cm{
A\ar[r]^u\ar[d]_f & B\ar@{-->}[ld]^-{\overline f} \\
B'.}\]

\vskip5pt

$(1')$ \ Dually, a morphism $c: B\longrightarrow A$ is a \emph{right $\mathcal B$-strong-approximation} of $A$, provided that $B\in\mathcal B$,  and for each object $B'\in\mathcal B$ and each morphism $g:B'\longrightarrow A$, there is a unique morphism $\overline g:B'\longrightarrow B$ with $g=c\circ \overline g$:
\[\xymatrix@C=1.0cm@R=0.5cm{
& B'\ar@{-->}[ld]_-{\overline g}\ar[d]^g & {} \\
B\ar[r]^c & A.
}\]

\vskip10pt

$(2)$ \ {\rm (\cite[IV, 3]{Mac})} \ A full subcategory $\mathcal F\subseteq\mathcal A$ is \emph{reflective},  provided that the inclusion $i:\mathcal F\longrightarrow\mathcal A$ admits a left adjoint $r:\mathcal A\longrightarrow\mathcal F$.
If this is the case, $r$ is called a \emph{reflector} of $\mathcal F$, and write $\eta:\Id_{\mathcal A}\longrightarrow i\circ r$ for the unit.

\vskip5pt

$(2')$ \  Dually, a full subcategory $\mathcal C\subseteq\mathcal A$ is  \emph{coreflective}, provided that the inclusion $i:\mathcal C\longrightarrow\mathcal A$ admits a right adjoint $t:\mathcal A\longrightarrow\mathcal C$.
If this is the case,  $t$ is called a \emph{coreflector} of $\mathcal C$, and write $\varepsilon: i\circ t\longrightarrow\Id_{\mathcal A}$ for the counit.

\vskip5pt

$(3)$ \ {\rm (\cite[Theorem~4.3]{CHK85})} \ A reflector $r:\mathcal A\longrightarrow\mathcal F$ of a reflective subcategory $\mathcal F\subseteq\mathcal A$ is \emph{semi-left-exact},
provided that for every left $\mathcal F$-strong-approximation $u:A\longrightarrow F$ of $A$ and every $g:E\longrightarrow F$ with $E\in\mathcal F$, there exists a pullback square in $\mathcal A$
\[
\xymatrix@C=0.9cm@R=0.6cm{
P\ar[r]\ar[d]_{u'} & A\ar[d]^u\\
E\ar[r]^-g & F
}
\]
such that $u':P\longrightarrow E$ is a left $\mathcal F$-strong-approximation of $P$.

\vskip 5pt

\vskip5pt

$(3')$ \ Dually, a coreflector $t:\mathcal A\longrightarrow\mathcal C$ of a coreflective subcategory $\mathcal C\subseteq\mathcal A$ is \emph{semi-right-exact}, provided that for every right $\mathcal C$-strong-approximation $c:C\longrightarrow A$ of $A$ and every $h:C\longrightarrow D$ with $D\in\mathcal C$, there exists a pushout square in $\mathcal A$
\[
\xymatrix@C=0.9cm@R=0.6cm{
C\ar[r]^{h}\ar[d]_{c} & D\ar[d]^{c'}\\
A\ar[r] & Q
}
\]
such that $c':D\longrightarrow Q$ is a right $\mathcal C$-strong-approximation of $Q$.
\end{defn}

Note that in Definition~\ref{defn:strong-approximation}{\rm (2)}, the counit $\varepsilon: r\circ i\longrightarrow\Id_{\mathcal F}$ of the adjoint pair $(r, i)$ is a natural isomorphism. Dually, in Definition~\ref{defn:strong-approximation}{$(2')$},
the unit $\eta:\Id_{\mathcal C}\longrightarrow t\circ i$ of the adjoint pair $(i, t)$ is a natural isomorphism.

\vskip5pt

We omit the proof of the following fact, which is direct.

\begin{fact}\label{fact:approximation-basic}\  Let $\mathcal A$ be a category.

\vskip5pt

$(1)$ \ Let $\mathcal F\subseteq\mathcal A$ be a reflective subcategory with reflector $r:\mathcal A\longrightarrow\mathcal F$ and the unit $\eta:\Id_{\mathcal A}\longrightarrow i\circ r$. Then for every $A\in\mathcal A$ one has

\vskip5pt

\hskip20pt ${\rm (i)}$ \  $\eta_{_A}:A\longrightarrow rA$ is a left $\mathcal F$-strong-approximation of $A$.
\vskip5pt

\hskip20pt ${\rm (ii)}$ \  If $u:A\longrightarrow F$ and $u':A\longrightarrow F'$ are left $\mathcal F$-strong-approximation of $A$, then there is a unique isomorphism $v:F\longrightarrow F'$ such that $u'=v\circ u$.
\vskip5pt

\hskip20pt ${\rm (iii)}$ \ $r(\eta_{_A}):rA\longrightarrow r^2A$ is an isomorphism.
\vskip5pt

$(1')$ \ Dually, let $\mathcal C\subseteq\mathcal A$ be a coreflective subcategory with coreflector $t:\mathcal A\longrightarrow\mathcal C$. Then for every $A\in\mathcal A$ one has

\vskip5pt

\hskip20pt ${\rm (i')}$ \  $\varepsilon_A:tA\longrightarrow A$ is a right $\mathcal C$-strong-approximation of $A$.
\vskip5pt

\hskip20pt ${\rm (ii')}$ \ If $c:C\longrightarrow A$ and $c':C'\longrightarrow A$ are right $\mathcal C$-strong-approximations of $A$, then there is a unique isomorphism $v:C\longrightarrow C'$ such that $c'\circ v=c.$
\vskip5pt

\hskip20pt ${\rm (iii')}$ \ $t(\varepsilon_A):t^2A\longrightarrow tA$ is an isomorphism.
\end{fact}

\vskip5pt

\begin{lem} {\rm(\cite[Theorem 4.1]{CHK85})}\label{lem:sle-reflective-ofs} \ $(1)$ \ Let $\mathcal F\subseteq\mathcal A$ be a reflective subcategory with semi-left-exact reflector $r:\mathcal A\longrightarrow\mathcal F$.
Set
\begin{equation*}
\begin{aligned}
\mathcal E_r &= \{f \in {\rm Mor}(\mathcal A)\ \mid \  r(f)\text{ is an isomorphism}\}
\\
\mathcal M_r
&=
\left\{
f: X\longrightarrow Y\ \middle|\begin{array}{c}
\xymatrix@C=0.9cm@R=0.4cm{
X\ar[r]^-{\eta_{_X}}\ar[d]_f & rX\ar[d]^{r(f)}\\
Y\ar[r]^-{\eta_{_Y}} & rY
}
\end{array}\text{is a pullback}
\right\}.
\end{aligned}
\end{equation*}
Then $(\mathcal E_r,\mathcal M_r)$ is an orthogonal factorization system, and $\mathcal E_r$ satisfies the two-out-of-three property.

\vskip5pt

$(1')$ \ Dually,  let $\mathcal C\subseteq\mathcal A$ be a coreflective subcategory with semi-right-exact coreflector $t:\mathcal A\longrightarrow\mathcal C$.
Set
\begin{equation*}
\begin{aligned}
\mathcal E_t &= \left\{
f:X\longrightarrow Y\ \middle|\begin{array}{c}
\xymatrix@C=0.9cm@R=0.4cm{
tX \ar[r]^-{\varepsilon_{_X}} \ar[d]_{t(f)} & X \ar[d]^{f} \\
tY \ar[r]^-{\varepsilon_{_Y}} & Y
}
\end{array}\text{ is a pushout}
\right\}
\\
\mathcal M_t
&= \{\ f \in {\rm Mor}(\mathcal A)\ \mid \  t(f)\text{ is an isomorphism}\}.
\end{aligned}
\end{equation*}
Then $(\E_t,\M_t)$ is an orthogonal factorization system, and $\mathcal M_t$ satisfies the two-out-of-three property.
\end{lem}

\vskip5pt

Let $G$ be a group. By the trivial action, any set is a $G$-set. Thus $\mathsf{Set}$ is a full subcategory of
$G\text{-}\mathsf{Set}$, the category of $G$-sets. For a $G$-set $X$, denote by $X/G$ the set of $G$-orbits of $X$,
and for a $G$-map $f: X\longrightarrow Y$, denote by $r(f): X/G \longrightarrow Y/G$ the map given by $Gx \longmapsto Gf(x)$.

\vskip5pt

\begin{prop} \label{reflG-set} \ Let $G$ be a group, and $r: G\text{-}\mathsf{Set} \longrightarrow  \mathsf{Set}$ the functor defined by $r(X) = X/G$.
Then $\mathsf{Set}$ is a reflective subcategory of $G\text{-}\mathsf{Set}$ with reflector $r$, and $r$ is semi-left-exact, and $(\mathcal E_{\mathrm{orb}},\mathcal M_{\mathrm{orb}})$ is an  orthogonal factorization system, where
\begin{equation*}
\begin{aligned}
\mathcal E_{\mathrm{orb}}
&=
\{\ G\text{-map} \ f:X\longrightarrow Y \  \mid \ r(f):X/G\longrightarrow Y/G
 \text{ is a bijection}\}\\
\mathcal M_{\mathrm{orb}}
&=
\left\{G\text{-map} \ f:X\longrightarrow Y \ \middle| \ \begin{array}{c}
\xymatrix@C=1.0cm@R=0.4cm{
X\ar[r]^-{\eta_{_X}}\ar[d]_{f} & X/G\ar[d]^-{r(f)}\\
Y\ar[r]^-{\eta_{_Y}} & Y/G
}
\end{array}\text{ is a pullback in} \ G\text{-}\mathsf{Set}
\right\}.
\end{aligned}
\end{equation*}
\end{prop}

\begin{proof} \ It is clear that $(r, i)$ is an adjoint pair, where $i: \mathsf{Set} \hookrightarrow G\text{-}\mathsf{Set}$ is the inclusion. The unit is given by the $G$-map $\eta_{_X}: X\longrightarrow X/G, \ x\longmapsto Gx$.
By Fact \ref{fact:approximation-basic}(1)(i), $\eta_{_X}: X\longrightarrow X/G$ is a left $\mathsf{Set}$-strong-approximation of $X$,
and by Fact \ref{fact:approximation-basic}(1)(ii) any left $\mathsf{Set}$-strong-approximation of $X$ is isomorphic to $\eta_{_X}$.
Let $E$ be a set regarded as a trivial $G$-set, and $h:E\longrightarrow X/G$ be a map.
Put $$P=\{(x,e)\in X\times E\mid \eta_{_X}(x)=h(e)\}.$$ Since $\eta_{_X}: X\longrightarrow X/G$ is surjective, it follows that $P\ne \emptyset$. Consider a $G$-action on $P$ defined by $\sigma\cdot(x,e)=(\sigma x,e),\ \forall \ \sigma \in G.$
This is well defined, since $\eta_X(\sigma x)=\sigma\eta_X(x)=\eta_X(x)=h(e)$.
Then one gets a commutative square
\[
\xymatrix@C=1.0cm@R=0.7cm{
P\ar[r]^-p\ar[d]_{\rho} & X\ar[d]^{\eta_{_X}} \\
E\ar[r]^-h & X/G
}
\]
where $\rho$ and $p$ are projections. By the construction it is clear that this is a pullback square.

\vskip5pt

It remains to prove that $\rho: P \longrightarrow E$ is a  left $\mathsf{Set}$-strong-approximation of $P$.

\vskip5pt

In fact, for each $e\in E$, $h(e)\in X/G$ is a $G$-orbit of $X$, and
\[
\rho^{-1}(e)
=\{(x,e)\in P\}
=\{(x,e)\mid x\in h(e)\}
= h(e)\times\{e\}.
\]
For elements $(x,e), (x',e)\in \rho^{-1}(e)$, since $h(e)$ is a $G$-orbit of $X$, there is $\sigma\in G$ such that $ x'=\sigma x.$
Thus $\sigma\cdot(x,e)=(\sigma x,e)=(x',e).$ This shows that $\rho^{-1}(e)$ is a $G$-orbit of $P$.

\vskip5pt

By definition the $G$-action on $P$ does not change the second coordinate. It follows that any $G$-orbit of $P$ is of the form $\rho^{-1}(e)$ with $e\in E$.
Thus $P$ is the disjoint union
$$ P = \overset \cdot {\bigcup\limits_{e\in E}} \rho^{-1}(e)$$
and hence $E$ is precisely the set of $G$-orbits of $P$, and has the commutative diagram
\[
\xymatrix@C=1.0cm@R=0.5cm{
P\ar@{=}[r]\ar[d]_{\rho} & P\ar[d]^{\eta_{_P}} \\
E\ar[r]^-\phi & P/G
}
\]
where $\phi: E\longrightarrow P/G$ is a bijection given by $e\longmapsto \rho^{-1}(e).$
By Fact \ref{fact:approximation-basic}(1)(i), the $G$-map $\eta_{_P}: P\longrightarrow P/G$ is a left $\mathsf{Set}$-strong-approximation of $P$, and hence
$\rho:P\longrightarrow E$ is also a left $\mathsf{Set}$-strong-approximation of $P$. This proves $r$ is semi-left-exact.

\vskip 5pt

By Lemma~\ref{lem:sle-reflective-ofs}, one gets the orthogonal factorization system $(\mathcal E_{\mathrm{orb}},\mathcal M_{\mathrm{orb}})$.
\end{proof}

\vskip 5pt
Let $\mathbb Q\text{-Vect}$  denote the category of vector spaces over the field $\mathbb Q$ of rational numbers.

\vskip 5pt

\begin{prop} \ \label{reflAb} \ The category $\mathbb Q\text{-Vect}$ is a reflective subcategory of $\operatorname{Ab}$ with reflector $-\otimes_{\mathbb Z}\mathbb Q: \operatorname{Ab} \longrightarrow \mathbb Q\text{-Vect}$, and $-\otimes_{\mathbb Z}\mathbb Q$ is semi-left-exact,
and $(\mathcal E_{\mathbb Q},\mathcal M_{\mathbb Q})$ is an orthogonal factorization system, where
\begin{equation*}
\begin{aligned}
\mathcal E_{\mathbb Q}
&=
\{f \in {\rm Mor}({\rm Ab})\ \mid \  \Ker f\text{ and }\operatorname{Coker}f
 \text{ are abelian torsion groups}\}\\
\mathcal M_{\mathbb Q}
&=
\left\{
f:A\longrightarrow B\ \middle|\begin{array}{c}
\xymatrix@C=1.0cm@R=0.4cm{
A\ar[r]^{\eta_{_A}}\ar[d]_{f}
 & A\otimes_{\mathbb Z}\mathbb Q
     \ar[d]^{f\otimes_{\mathbb Z}\mathbb Q}\\
B\ar[r]_{\eta_{_B}}
 & B\otimes_{\mathbb Z}\mathbb Q
}
\end{array}\text{ is a pullback in} \ {\rm Ab}
\right\}.
\end{aligned}
\end{equation*}
\end{prop}

\begin{proof} \
It is clear that $(-\otimes_{\mathbb Z}\mathbb Q, i)$ is an adjoint pair, where $i: \mathbb Q\text{-Vect} \hookrightarrow \operatorname{Ab}$ is the inclusion. The unit is given by the $\eta_{_A}: A\longrightarrow A\otimes_{\mathbb Z}\mathbb Q, \ a\longmapsto a\otimes_{\mathbb Z}1$.
By Fact \ref{fact:approximation-basic}(1)(i), $\eta_{_A}: A\longrightarrow A\otimes_{\mathbb Z}\mathbb Q$ is a left $\mathbb Q\text{-Vect}$-strong-approximation of $A$,
and by Fact \ref{fact:approximation-basic}(1)(ii) any left $\mathbb Q\text{-Vect}$-strong-approximation of $A$ is isomorphic to $\eta_{_A}$.
For a $\mathbb Q$-vector space $V$ and a $\mathbb Q$-linear map $h:V\longrightarrow A\otimes_{\mathbb Z}\mathbb Q$, consider the pullback in $\operatorname{Ab}$
\[
\xymatrix@C=1.0cm@R=0.7cm{
P\ar[r]\ar[d]_p & A\ar[d]^{\eta_{_A}} \\
V\ar[r]^-h & A\otimes_{\mathbb Z}\mathbb Q .
}
\]

\vskip5pt\noindent Since $_\mathbb Z\mathbb Q$ is a flat module, it follows that $-\otimes_{\mathbb Z}\mathbb Q: \operatorname{Ab} \longrightarrow \mathbb Q\text{-Vect}$ is an exact functor, and
$-\otimes_{\mathbb Z}\mathbb Q$ preserves pullbacks. Thus
\[\xymatrix@C=1.0cm@R=0.7cm{
P\otimes_{\mathbb Z}\mathbb Q\ar[r]\ar[d]_{p\otimes_{\mathbb Z}\mathbb Q} & A\otimes_{\mathbb Z}\mathbb Q\ar[d]^{\eta_{_A}\otimes_{\mathbb Z}\mathbb Q} \\
V\otimes_{\mathbb Z}\mathbb Q\ar[r]^-{h\otimes_{\mathbb Z}\mathbb Q} & (A\otimes_{\mathbb Z}\mathbb Q)\otimes_{\mathbb Z}\mathbb Q
}
\]
\vskip5pt\noindent is a pullback square in $\mathbb Q\text{-Vect}$.
By Fact \ref{fact:approximation-basic}(1)(iii),  $\eta_{_A}\otimes_{\mathbb Z}\mathbb Q$ is an isomorphism.
Hence $p\otimes_{\mathbb Z}\mathbb Q$ is an isomorphism.
Since $\tau:V\otimes_{\mathbb Z}\mathbb Q\longrightarrow V,\ v\otimes q\mapsto qv,$ is an isomorphism,
one has the isomorphism $\tau\circ(p\otimes_{\mathbb Z}\mathbb Q): P\otimes_{\mathbb Z}\mathbb Q \longrightarrow V$.
By Fact \ref{fact:approximation-basic}(1)(i), $\eta_{_P}: P\longrightarrow P\otimes_{\mathbb Z}\mathbb Q$ is a left $\mathbb Q\text{-Vect}$-strong-approximation of $P$.
It follows from the commutative diagram
\[
\xymatrix@C=1.8cm@R=0.7cm{P\ar@{=}[r]\ar[d]_-{\eta_{_P}}
 &  P\ar[d]^-{p} \\ P\otimes_{\mathbb Z}\mathbb Q
\ar[r]^-{\tau\circ (p\otimes_{\mathbb Z}\mathbb Q)} & V}
\]
that $p$ is a strong left $\mathbb Q\text{-Vect}$-approximation of $P$. This proves that $-\otimes_{\mathbb Z}\mathbb Q$ is semi-left-exact.

\vskip 5pt

Note that $A$ is an abelian torsion group if and only if $A\otimes_{\mathbb Z}\mathbb Q = 0$. In fact, if $A$ is an abelian torsion group, then it is clear that $A\otimes_{\mathbb Z}\mathbb Q = 0$.
Conversely, if $A\otimes_{\mathbb Z}\mathbb Q = 0$ and $A$ is not an abelian torsion group, then there is $a\in A$ whose order is infinite. Thus one gets $\langle a\rangle\cong \mathbb Z$ and an inclusion $j:\mathbb Z \hookrightarrow A$. Applying $-\otimes_{\mathbb Z}\mathbb Q$  one gets $j\otimes_{\mathbb Z}\mathbb Q = 0$, and hence ${\mathbb Z}\otimes_{\mathbb Z}\mathbb Q \hookrightarrow 0$, which is a contradiction.

\vskip 5pt
Thus
\begin{align*}
&\{\ f \in {\rm Mor}({\rm Ab})\
 \mid f\otimes_{\mathbb Z}\mathbb Q
 \text{ is an isomorphism}\} \\=
&\{f \in {\rm Mor}({\rm Ab})\ \mid \  \Ker f\otimes_{\mathbb Z}\mathbb Q = 0 = \operatorname{Coker}f\otimes_{\mathbb Z}\mathbb Q
\} \\=
&\{f \in {\rm Mor}({\rm Ab})\ \mid \  \Ker f\text{ and }\operatorname{Coker}f
 \text{ are abelian torsion groups}\}.
\end{align*}
By Lemma~\ref{lem:sle-reflective-ofs}, $(\mathcal E_{\mathbb Q},\mathcal M_{\mathbb Q})$ is an orthogonal factorization system.
\end{proof}

\subsection{\bf Bi-reflective pairs and  bi-reflective model structures}  We will define bi-reflective pairs and  bi-reflective model structures, and then give the one to one correspondence between them.

\begin{defn}\label{defn:bi-reflectivepair} \ Let $\mathcal A$ be a category. A pair $(\mathcal C,\mathcal F)$ of full subcategories of $\mathcal A$ is an \emph{bi-reflective pair} if the following conditions are satisfied:

\vskip5pt
\hskip20pt {\rm (A1)} \ $\mathcal C$ is a coreflective subcategory with coreflector $t: \mathcal A \longrightarrow \mathcal C$, and $t$ is semi-right-exact$;$
and $\mathcal F$ is a reflective subcategory with reflector $r: \mathcal A \longrightarrow \mathcal F$, and $r$ is semi-left-exact. Consider the orthogonal factorization systems
$(\mathcal E_t,\mathcal M_t)$ and $(\mathcal E_r,\mathcal M_r)$ as given in Lemma \ref{lem:sle-reflective-ofs}, i.e.,
\begin{equation*}
\begin{aligned}
\mathcal E_t &= \left\{
f:X\longrightarrow Y\ \middle|\begin{array}{c}
\xymatrix@C=0.9cm@R=0.4cm{
tX \ar[r]^-{\varepsilon_{_X}} \ar[d]_{t(f)} & X \ar[d]^{f} \\
tY \ar[r]^-{\varepsilon_{_Y}} & Y
}
\end{array}\text{ is a pushout}
\right\}
\\
\mathcal M_t
&= \{\ f \in {\rm Mor}(\mathcal A)\ \mid \  t(f)\text{ is an isomorphism}\}
\end{aligned}
\end{equation*}

 \begin{equation*}
 \begin{aligned}
\mathcal E_r &= \{f \in {\rm Mor}(\mathcal A)\ \mid \  r(f)\text{ is an isomorphism}\}
\\
\mathcal M_r
&=
\left\{
f: X\longrightarrow Y\ \middle|\begin{array}{c}
\xymatrix@C=0.9cm@R=0.4cm{
X\ar[r]^-{\eta_{_X}}\ar[d]_f & rX\ar[d]^{r(f)}\\
Y\ar[r]^-{\eta_{_Y}} & rY
}
\end{array}\text{is a pullback}
\right\}.
\end{aligned}
\end{equation*}

\vskip5pt
\hskip20pt {\rm (A2)} \   $\mathcal E_r\perp\mathcal M_t$. In particular, $\mathcal E_r\subseteq\mathcal E_t$ and $\mathcal M_t\subseteq\mathcal M_r$.

\vskip5pt
\hskip20pt {\rm (A3)} \ $\mathcal E_r\circ\mathcal M_t\subseteq\mathcal M_t\circ\mathcal E_r$.
\end{defn}
\vskip5pt

\begin{defn}\label{defn:bi-reflective-orthogonal-model-structure} \ Let $\mathcal A$ be a category. An orthogonal model structure $(\CoFib,\Fib,\Weq)$ on $\mathcal A$ is an \emph{bi-reflective model structure} if the following conditions are satisfied:

\vskip5pt
\hskip20pt {\rm (M1)} \  $\mathcal A$ has enough cofibrant objects and enough fibrant objects. Let $\mathcal C$ and $\mathcal F$ denote the classes of cofibrant objects and fibrant objects, respectively.
\vskip5pt
\hskip20pt {\rm (M2)} \ The classes $\TCoFib$ and $\TFib$ satisfy the two-out-of-three property.
\vskip5pt
\hskip20pt ${\rm (M3)}$ \  For every $p:C\longrightarrow D$ in $\TFib$ with $C\in\mathcal C$ and every $u:C\longrightarrow E$ in $\CoFib$, there exists a pushout square
\[
\xymatrix@C=0.9cm@R=0.6cm{
C\ar[r]^{u}\ar[d]_{p} & E\ar[d]^{q}\\
D\ar[r] & Q
}
\]
such that $q\in\TFib$.

\vskip5pt
\hskip20pt ${\rm (M3')}$ \  For every $i:X\longrightarrow Z$ in $\TCoFib$ with $Z\in\mathcal F$ and every $v:Y\longrightarrow Z$ in $\Fib$, there exists a pullback square
\[
\xymatrix@C=0.9cm@R=0.6cm{
P\ar[r]\ar[d]_{j} & X\ar[d]^{i}\\
Y\ar[r]^-{v} & Z
}
\]
such that $j\in\TCoFib$.
\end{defn}

The following  result will be useful in latter applications.

\begin{prop}\label{prop:weq-only-composition}\ Let $(\mathcal E,\mathcal M)$ and $(\mathcal E',\mathcal M')$ be orthogonal factorization systems on $\mathcal A$ such that $\mathcal E\perp\mathcal M'$. Suppose that $\mathcal E$ and $\mathcal M'$ satisfy the two-out-of-three property, and put $\Weq:=\mathcal M'\circ\mathcal E$. Then $\Weq$ satisfies the two-out-of-three property if and only if it is closed under compositions.
\end{prop}

\begin{proof}
Let $X\overset{f}{\longrightarrow}Y\overset{g}{\longrightarrow}Z$ be morphisms in $\mathcal A$.
Assume that $\Weq$ is closed under compositions. Suppose that $f\in\Weq$ and $g\circ f\in\Weq$. Factorize $g=p\circ i$ with $i\in\mathcal E$ and $p\in\mathcal M$.
Since any isomorphism is in $\mathcal E\cap \mathcal M = \mathcal E'\cap \mathcal M'$, it follows that  $i\in\Weq$, and hence  $(i\circ f)\in\Weq$.
By $\Weq =\mathcal M'\circ\mathcal E$ one has $i\circ f=q\circ j$ with $j\in\mathcal E$ and $q\in\mathcal M'$.
\[\xymatrix@C=0.8cm@R=0.3cm{
  X\ar[rr]^-{f}\ar[ddrr]_-{j} & {} & Y\ar[rr]^-{g}\ar[dr]_-{i} & {} & Z \\
  {} & {} & {} & W\ar[ur]_-{p} & {} \\
  {} & {} & L\ar[ur]_-{q} & {} &
}\]
Since $q\in\mathcal M'\subseteq\mathcal M$, it follows that $g\circ f =(p\circ q)\circ j$ with $j\in\mathcal E$ and $(p\circ q)\in\mathcal M$.
Since $(g\circ f)\in\Weq = \mathcal M'\circ\mathcal E$, one has  $g\circ f =u\circ v$ with $v:X \longrightarrow M$  in $\mathcal E$ and $u:M \longrightarrow Z$
in $\mathcal M'\subseteq\mathcal M$. Thus $(g\circ f)$ has two $(\mathcal E,\mathcal M)$-factorizations: $g\circ f =(p\circ q)\circ j =u\circ v$.
By Proposition~\ref{prop:OFS-criterion}(3), there exists a unique isomorphism $\theta$ such that $\theta \circ j=v$ and $u\circ \theta=p\circ q:$
\[\xymatrix@C=0.8cm@R=0.3cm{
  {} & {} & L\ar[dr]^-{q}\ar@{-->}[dd]^-{\theta} & {} & {} \\
  {} & {} & {} & W\ar[dr]^-{p} & {} \\
  X\ar[uurr]^-{j}\ar[rr]^-{v} & {} & M\ar[rr]^-{u} & {} & Z.
}\]
Since $q\in\mathcal M'$ and $p\circ q = u\circ \theta\in\mathcal M'$ and $\mathcal M'$ satisfies the two-out-of-three property, it follows that $p\in\mathcal M'$, and hence $g=p\circ i\in \mathcal M'\circ\mathcal E =\Weq$.

\vskip 5pt

Dually, if $g\in\Weq$ and $g\circ f\in\Weq$ then $f\in\Weq$. \end{proof}

\begin{prop} \label{prop:bi-reflective-A3}\
Let $\mathcal C$ be a coreflective subcategory of category $\mathcal A$ with semi-right-exact coreflector $t$, and $\mathcal F$ a reflective subcategory of category $\mathcal A$ with semi-left-exact reflector $r$.
Let  $(\mathcal E_t,\mathcal M_t)$ and $(\mathcal E_r,\mathcal M_r)$ be the orthogonal factorization systems given in {\rm Lemma \ref{lem:sle-reflective-ofs}}, i.e.,
\begin{equation*}
\begin{aligned}
\mathcal E_t &= \left\{
f:X\longrightarrow Y\ \middle|\begin{array}{c}
\xymatrix@C=0.9cm@R=0.4cm{
tX \ar[r]^-{\varepsilon_{_X}} \ar[d]_{t(f)} & X \ar[d]^{f} \\
tY \ar[r]^-{\varepsilon_{_Y}} & Y
}
\end{array}\text{ is a pushout}
\right\}
\\
\mathcal M_t
&= \{\ f \in {\rm Mor}(\mathcal A)\ \mid \  t(f)\text{ is an isomorphism}\}
\end{aligned}
\end{equation*}
 \begin{equation*}
 \begin{aligned}
\mathcal E_r &= \{f \in {\rm Mor}(\mathcal A)\ \mid \  r(f)\text{ is an isomorphism}\}
\\
\mathcal M_r
&=
\left\{
f: X\longrightarrow Y\ \middle|\begin{array}{c}
\xymatrix@C=0.9cm@R=0.4cm{
X\ar[r]^-{\eta_{_X}}\ar[d]_f & rX\ar[d]^{r(f)}\\
Y\ar[r]^-{\eta_{_Y}} & rY
}
\end{array}\text{is a pullback}
\right\}.
\end{aligned}
\end{equation*}

\vskip5pt

Put $\Weq:=\mathcal M_t\circ\mathcal E_r$. If $\mathcal E_r\perp\mathcal M_t$, then the following conditions are equivalent:

\vskip5pt
\hskip20pt ${\rm (1)}$ \   $\mathcal E_r\circ\mathcal M_t\subseteq\mathcal M_t\circ\mathcal E_r$.

\vskip5pt
\hskip20pt ${\rm (2)}$ \   $r\mathcal M_t\subseteq\mathcal M_t$.

\vskip5pt
\hskip20pt ${\rm (2')}$ \   $t\mathcal E_r\subseteq\mathcal E_r$.

\vskip5pt
\hskip20pt ${\rm (3)}$ \   $\Weq$ has the two-out-of-three property.

\end{prop}

\begin{proof}\ {\rm (1)} $\Longrightarrow$ {\rm (3)}: \ If \ $\mathcal E_r\circ\mathcal M_t\subseteq\mathcal M_t\circ\mathcal E_r$, then
$$\Weq\circ\Weq=\mathcal M_t\circ\mathcal E_r\circ\mathcal M_t\circ\mathcal E_r\subseteq\mathcal M_t\circ\mathcal M_t\circ\mathcal E_r\circ\mathcal E_r\subseteq\mathcal M_t\circ\mathcal E_r=\Weq$$ i.e.,  $\Weq$ is closed under compositions,
and hence  $\Weq$ has the two-out-of-three property, by Proposition~\ref{prop:weq-only-composition}.

\vskip 5pt

{\rm (3)} $\Longrightarrow$ {\rm (1)}: \ Since $\mathcal E_r\subseteq\Weq$ and $\mathcal M_t\subseteq\Weq$, it follows from the two-out-of-three property that $\mathcal E_r\circ\mathcal M_t\subseteq\Weq=\mathcal M_t\circ\mathcal E_r$.

\vskip 5pt

{\rm (2)} $\Longrightarrow$ {\rm (3)}: \  Suppose that $r\mathcal M_t\subseteq\mathcal M_t$. It suffices to prove claim:
$$\Weq=\{f \in {\rm Mor}(\mathcal A)\ \mid \  (t\circ r)f\text{ is an isomorphism}\}.$$

\vskip5pt

In fact, if $f=m\circ e\in \Weq =\mathcal M_t\circ\mathcal E_r$ with $e\in\mathcal E_r$ and $m\in\mathcal M_t$, then  $r(e)$ is an isomorphism (cf. the definition of $\mathcal E_r$), and $r(m)\in r\mathcal \M_t \subseteq \M_t$. Thus $(t\circ r)e = t(r(e))$ is an isomorphism, and $(t\circ r)m = t(r(m))$ is an isomorphism (cf. the definition of $\mathcal M_t$). Thus $(t\circ r)f=(t\circ r)(m\circ e)$ is an isomorphism.

\vskip 5pt

Conversely, suppose that $(t\circ r)f$ is an isomorphism.  Factorize $f:X\longrightarrow Y$ as $f=m\circ e$, where $e:X\longrightarrow P$ lies in $\mathcal E_r$ and $m:P\longrightarrow Y$ lies in $\mathcal M_r$.
It suffices to show $m\in\mathcal M_t$, i.e., $t(m)$ is an isomorphism.
Since $m\in\mathcal M_r$, by definition one has a pullback square in $\mathcal A$:
\[
\xymatrix@C=0.9cm@R=0.6cm{
P\ar[r]^{\eta_P}\ar[d]_{m} & rP\ar[d]^{r(m)}\\
Y\ar[r]^{\eta_{_Y}} & rY.
}
\]
Since $(i, t)$ is an adjoint pair, where $i: \mathcal C \hookrightarrow \mathcal A$ is the inclusion, $t$ preserves pullbacks, and hence one has the pullback square in $\mathcal C$:
\[
\xymatrix@C=0.9cm@R=0.6cm{
tP\ar[r]^-{t(\eta_P)}\ar[d]_{t(m)} & (t\circ r)P \ar[d]^{(t\circ r)m}\\
tY\ar[r]^-{t(\eta_{_Y})} & (t\circ r)Y.
}
\]
Since $e\in \mathcal E_r$, one has $(t\circ r)e$ is an isomorphism. By assumption $(t\circ r)f$ is an isomorphism, then so is $(t\circ r)m$. Therefore $t(m)$ is an isomorphism.

\vskip 5pt

{\rm (3)} $\Longrightarrow$ ${\rm (2)}$: \ Assume that $\Weq =\mathcal M_t\circ\mathcal E_r$ has the two-out-of-three property. We need to prove $r\mathcal M_t\subseteq\mathcal M_t$.
For  $f:X\longrightarrow Y$ in $\mathcal{M}_t$, one has commutative diagram:
\[
\xymatrix@C=0.9cm@R=0.6cm{
X\ar[r]^-{\eta_{_X}}\ar[d]_{f} & rX\ar[d]^{r(f)}\\
Y\ar[r]^-{\eta_{_Y}} & rY
}
\]
where $\eta:{\rm Id}_{\mathcal{A}}\longrightarrow j\circ r$ is the unit of adjoint pair $(r, j)$, and $j:\mathcal{F}\longrightarrow \mathcal{A}$ is the inclusion.
By Fact \ref{fact:approximation-basic}(1)(iii), one has $r(\eta_X)$ is isomorphism, thus $\eta_{_X}\in\mathcal E_r$ and $\eta_{_Y}\in\mathcal E_r$. Since $\eta_{_X}, \eta_{_Y}, f$ are in $\Weq =\mathcal M_t\circ\mathcal E_r$ and $\Weq$ has two-out-of-three property, it follows that $r(f)\in\Weq$.

\vskip 5pt

By the commutative square:
\[
\xymatrix@C=0.9cm@R=0.6cm{
rX\ar[r]^-{r(\eta_{_X})}\ar[d]_{r(f)} & r^2X\ar[d]^{r^2(f)}\\
rY\ar[r]^-{r(\eta_{_Y})} & r^2Y
}
\]
with isomorphisms $r(\eta_{_X})$ and $r(\eta_{_Y})$, one sees that it is a pullback square.  Hence $r(f)\in\mathcal M_r$ (cf. the definition of $\mathcal M_r$).
By Proposition~\ref{prop:retract-argument}, one has $r(f)\in\Weq\cap \mathcal M_r = \mathcal M_t$.

\vskip 5pt

The implications ${\rm (2')}$ $\Longrightarrow ${\rm (3)} and {\rm (3)} $\Longrightarrow {\rm (2')}$ follow by duality.
\end{proof}

Every model structure induced by a bi-reflective pair is bi-reflective.

\begin{prop}\label{prop:bi-reflective-pair-gives-bi-reflective-model}\ Let $(\mathcal C,\mathcal F)$ be a bi-reflective pair with the coreflector $t$ of $\mathcal C$ and the reflector $r$ of $\mathcal F$,
$(\mathcal E_t,\mathcal M_t)$ and $(\mathcal E_r,\mathcal M_r)$ the orthogonal factorization systems given in {\rm Lemma \ref{lem:sle-reflective-ofs}}, i.e.,
\begin{equation*}
\begin{aligned}
\mathcal E_t &= \left\{
f:X\longrightarrow Y\ \middle|\begin{array}{c}
\xymatrix@C=0.9cm@R=0.4cm{
tX \ar[r]^-{\varepsilon_{_X}} \ar[d]_{t(f)} & X \ar[d]^{f} \\
tY \ar[r]^-{\varepsilon_{_Y}} & Y
}
\end{array}\text{ is a pushout}
\right\}
\\
\mathcal M_t
&= \{\ f \in {\rm Mor}(\mathcal A)\ \mid \  t(f)\text{ is an isomorphism}\}
\end{aligned}
\end{equation*}
 \begin{equation*}
 \begin{aligned}
\mathcal E_r &= \{f \in {\rm Mor}(\mathcal A)\ \mid \  r(f)\text{ is an isomorphism}\}
\\
\mathcal M_r
&=
\left\{
f: X\longrightarrow Y\ \middle|\begin{array}{c}
\xymatrix@C=0.9cm@R=0.4cm{
X\ar[r]^-{\eta_{_X}}\ar[d]_f & rX\ar[d]^{r(f)}\\
Y\ar[r]^-{\eta_{_Y}} & rY
}
\end{array}\text{is a pullback}
\right\}.
\end{aligned}
\end{equation*}

\vskip5pt
\noindent Then $(\mathcal E_t, \ \mathcal M_r, \ \mathcal M_t\circ \mathcal E_r)$ is a bi-reflective model structure, $\mathcal C$ is the class of cofibrant objects, and $\mathcal F$ is the class of fibrant objects, in the sense of
{\rm Definition \ref{defn:cofibrant-fibrant-objects}}.
\end{prop}
\begin{proof} \ Put $\Weq =\mathcal M_t\circ \mathcal E_r$. Since $(\mathcal C,\mathcal F)$ is a bi-reflective pair, by definition one has $\mathcal E_r\circ\mathcal M_t\subseteq\mathcal M_t\circ\mathcal E_r$.
By Proposition~\ref{prop:bi-reflective-A3}, $\Weq =\mathcal M_t\circ \mathcal E_r$ has the two-out-of-three property.

\vskip5pt

As claimed in {\rm Lemma \ref{lem:sle-reflective-ofs}}, $(\mathcal E_t, \mathcal M_t)$ and $(\mathcal E_r, \mathcal M_r)$ are the orthogonal factorization systems with $\mathcal E_r\subseteq \mathcal E_t$
(cf.  Definition \ref{defn:bi-reflectivepair}(A2)). Applying Proposition~\ref{prop:retract-argument} one has $$\mathcal E_t\cap\Weq=\mathcal E_r, \ \ \mathcal M_r\cap\Weq=\mathcal M_t.$$
Thus, $(\mathcal E_t, \mathcal M_r\cap\Weq)$ and $(\mathcal E_t\cap\Weq, \mathcal M_r)$ are the orthogonal factorization systems and $\Weq$ has the two-out-of-three property.
It follows from Proposition~\ref{prop:weq-two-out-of-three} that $(\mathcal E_t, \ \mathcal M_r, \ \mathcal M_t\circ \mathcal E_r)$ is an orthogonal model structure.

\vskip5pt

We need to prove that the conditions (M1), (M2), (M3) and ${\rm (M3')}$ in Definition \ref{defn:bi-reflective-orthogonal-model-structure} are satisfied.

\vskip 5pt

First, we claim that the class of fibrant objects of $(\mathcal E_t,\mathcal M_r,\mathcal M_t\circ \mathcal E_r)$ is precisely $\mathcal F$.
In fact, for any object $F\in\mathcal F$, let $e:X\longrightarrow Y$ be a morphism in $\mathcal E_r$ and $u:X\longrightarrow F$ a morphism. Let $i: \mathcal F\hookrightarrow \mathcal A$ be the inclusion.
Then one has commutative  diagram:
\[
\xymatrix@C=0.9cm@R=0.6cm{
X\ar[rr]^-{u}\ar[d]_-{e}\ar[dr]^{\eta_{_X}} && F\ar[dr]^{\eta_{_F}} &\\
Y\ar[dr]_-{\eta_{_Y}} & rX\ar[rr]^-{r(u)}\ar[d]^{r(e)} && rF\\
& rY &
}
\]
where $\eta:{\rm Id}_{\mathcal{A}}\longrightarrow i\circ r$ is the unit of the adjoint pair $(r, i)$.
Since $e\in\mathcal E_r$, by definition $r(e)$ is an isomorphism.
Since $i: \mathcal F\hookrightarrow \mathcal A$ is fully faithful, it follows that $\eta_{_{i(F)}}: i(F) \longrightarrow (i\circ r\circ i)(F)$ is an isomorphism for any $F\in\mathcal F$, i.e.,
$\eta_{_F}: F \longrightarrow r(F)$ is an isomorphism. Put $$\alpha:=\eta_{_F}^{-1}\circ r(u)\circ r(e)^{-1}\circ \eta_{_Y}: Y\longrightarrow F.$$ Then
$\alpha\circ e = u.$ Such an $\alpha$ is unique: indeed, if $\beta: Y\longrightarrow F$ also satisfies $\beta\circ e=u$, then  $r(\beta)\circ r(e)=r(u)$ and $r(\beta)=r(u)\circ r(e)^{-1}$.
On the other hand, one has $\eta_{_F} \circ \beta= r(\beta)\circ \eta_{_Y}$.
Thus $$\beta=\eta_{_F}^{-1}\circ r(\beta)\circ \eta_{_Y}=\eta_{_F}^{-1}\circ r(u)\circ r(e)^{-1}\circ \eta_{_Y}=\alpha.$$
By definition $F$ is a fibrant object.

\vskip 5pt

Conversely, let $F'$ be a fibrant object. By Fact \ref{fact:approximation-basic}$(1)${\rm (iii)},  \ $r(\eta_{_{F'}}):rF'\longrightarrow r^2F'$ is an isomorphism.
By the construction $\eta_{_{F'}}:F'\longrightarrow rF'$ is in $\mathcal E_r,$ i.e.,  $\eta_{_{F'}}: F'\longrightarrow rF'$ is a trivial cofibration. By the definition of a fibrant object, there exists a unique $s:rF'\longrightarrow F'$ such that $s\circ \eta_{_{F'}}=\Id_{F'}:$
\[
\xymatrix{
F'\ar@{=}[r]\ar[d]_{\eta_{_{F'}}} & F'\\
rF'\ar@{-->}[ur]_{s}
}
\]
 Since $rF'\in\mathcal F$ is fibrant and both $\Id_{rF'}$ and $\eta_{_{F'}}\circ s$  make the following diagram commute:
\[
\xymatrix{
F'\ar[rr]^{\eta_{_{F'}}}\ar[d]_{\eta_{_{F'}}}& & rF'\\
rF'\ar@{-->}[urr]^{\Id_{rF'}}_{\eta_{_{F'}}\circ s} &
}
\]
by the uniqueness one has $\eta_{_{F'}}\circ s=\Id_{rF'}$. It follows that $\eta_{_{F'}}$ is an isomorphism, and hence $F'\in\mathcal F$. This proves the claim.

\vskip5pt

Dually, the class of cofibrant objects is precisely $\mathcal C$.

\vskip 5pt

{\rm (M1)}: \ For any object $X \in \mathcal A$,  by Fact \ref{fact:approximation-basic}$(1'){\rm (iii')}$,  \ $t(\varepsilon_{_X}): t^2X\longrightarrow tX$ is an isomorphism.
By the construction $\varepsilon_{_X}:tX\longrightarrow X$  is in $\mathcal M_t$ with $tX\in\mathcal C$,  thus  $\varepsilon_{_X}:tX\longrightarrow X$ is a trivial fibration.
By definition $\mathcal A$ has enough cofibrant objects.

Similarly, $\mathcal A$ has enough fibrant objects.

\vskip 5pt

{\rm (M2)}:  By the construction one sees that $\mathcal E_r$ and $\mathcal M_t$ have the two-out-of-three property.

\vskip 5pt

{\rm (M3)}:  Let $p:C\longrightarrow D$ be a morphism in $\mathcal M_t$ with $C\in\mathcal C$ and $u:C\longrightarrow E$ a morphism in $\mathcal E_t$.
By Proposition \ref{prop:elementary-CF}$(1)$ one has $E\in\mathcal C$.
Since $p:C\longrightarrow D$ is a trivial fibration and $\mathcal C$ is the class of cofibrant objects, for any morphism $a:C'\longrightarrow D$ with $C'\in\mathcal C$, by definition of cofibrant objects,
there is a unique $\bar{a}:C'\longrightarrow C$ such that $p\circ\bar{a} = a$:
\[\xymatrix@C=1.0cm@R=0.5cm{
& C\ar[d]^p & {} \\
C'\ar[r]^a \ar@{-->}[ur]^-{\overline a} & D.
}\]
By Definition~\ref{defn:strong-approximation}, $p:C\longrightarrow D$ is a right $\mathcal C$-strong-approximation of $D$.
Since the coreflector $t$ of $\mathcal{C}$ is semi-right-exact, there exists a pushout square
\[
\xymatrix@C=0.9cm@R=0.6cm{
C\ar[r]^{u}\ar[d]_{p} & E\ar[d]^{q}\\
D\ar[r] & Q
}
\]
such that $q:E\longrightarrow Q$ is a right $\mathcal C$-strong-approximation of $Q$. It remains to show $q\in \mathcal M_t$.
Since $\varepsilon_{_Q}:tQ\longrightarrow Q$ is right $\mathcal C$-strong-approximation of $Q$ (cf. Fact~\ref{fact:approximation-basic}$(1'){\rm (i')}$),
it follows from Fact~\ref{fact:approximation-basic}$(1'){\rm (ii')}$ that there is a unique isomorphism $\tau: E\longrightarrow tQ$ such that $q =\varepsilon_{_Q}\circ \tau$.
Since $\varepsilon_{_Q}\in \mathcal M_t$ (cf. the proof of (M1)), one has $q\in \mathcal M_t$.

\vskip 5pt
${\rm (M3')}$:  This is the dual of {\rm (M3)}.
\end{proof}

Conversely, every bi-reflective model structure determines a bi-reflective pair.

\begin{prop}\label{prop:bi-reflective-model-gives-bi-reflective-pair} \ Let $(\CoFib,\Fib,\Weq)$ be a bi-reflective model structure on category $\mathcal A$.
Then $(\mathcal C, \mathcal F)$ is a bi-reflective pair, where $\mathcal C$ and $\mathcal F$ are the classes of cofibrant objects and  fibrant objects, respectively.

\vskip 5pt

Explicitly, $\mathcal C$ is a coreflective subcategory of $\mathcal A$ with coreflector $t: \mathcal A\longrightarrow \mathcal C$ and $t$ is semi-right-exact$;$
$\mathcal F$ is a reflective subcategory of $\mathcal A$ with reflector $r:\mathcal A\longrightarrow \mathcal F$ and $r$ is semi-left-exact.
Let $(\mathcal E_t,\mathcal M_t)$ and $(\mathcal E_r,\mathcal M_r)$ be the orthogonal factorization systems given in {\rm Lemma \ref{lem:sle-reflective-ofs}}, i.e.,
\begin{equation*}
\begin{aligned}
\mathcal E_t &= \left\{
f:X\longrightarrow Y\ \middle|\begin{array}{c}
\xymatrix@C=0.9cm@R=0.4cm{
tX \ar[r]^-{\varepsilon_{_X}} \ar[d]_{t(f)} & X \ar[d]^{f} \\
tY \ar[r]^-{\varepsilon_{_Y}} & Y
}
\end{array}\text{ is a pushout}
\right\}
\\
\mathcal M_t
&= \{\ f \in {\rm Mor}(\mathcal A)\ \mid \  t(f)\text{ is an isomorphism}\}
\end{aligned}
\end{equation*}
 \begin{equation*}
 \begin{aligned}
\mathcal E_r &= \{f \in {\rm Mor}(\mathcal A)\ \mid \  r(f)\text{ is an isomorphism}\}
\\
\mathcal M_r
&=
\left\{
f: X\longrightarrow Y\ \middle|\begin{array}{c}
\xymatrix@C=0.9cm@R=0.4cm{
X\ar[r]^-{\eta_{_X}}\ar[d]_f & rX\ar[d]^{r(f)}\\
Y\ar[r]^-{\eta_{_Y}} & rY
}
\end{array}\text{is a pullback}
\right\}.
\end{aligned}
\end{equation*}
\vskip 5pt
\noindent Then $(\mathcal E_t,\mathcal M_t)=(\CoFib,\TFib)$ and $(\mathcal E_r, \mathcal M_r) = (\TCoFib, \Fib)$.
\end{prop}

\begin{proof}

We first show that $\mathcal{F}$ is a reflective subcategory.

\vskip 5pt

Since $\mathcal A$ has enough fibrant objects (cf. Definition~\ref{defn:bi-reflective-orthogonal-model-structure}(M1)), for every object $X$,
we can choose a trivial cofibration $\eta_{_X}: X\longrightarrow rX$ with $rX\in\mathcal F$ (cf. Definition \ref{defn:cofibrant-fibrant-objects}$(2')$) such that if $X\in\mathcal F$ then $rX=X$ and $\eta_{_X}=\Id_X$.
Given a morphism $f:X\longrightarrow Y$ in $\mathcal A$, consider the morphism $\eta_{_Y}\circ f:X\longrightarrow rY$. Since $rY\in\mathcal F$, by the definition of a fibrant object (cf. Definition \ref{defn:cofibrant-fibrant-objects}$(1')$), there is a unique $r(f):rX\longrightarrow rY$ such that $r(f)\circ \eta_{_X} = \eta_{_Y}\circ f$:
\[
\xymatrix@C=1.0cm@R=0.7cm{
X\ar[r]^f\ar[d]_{\eta_{_X}} & Y\ar[r]^{\eta_{_Y}}&rY.\\
rX\ar@{-->}[urr]_{r(f)} &
}
\]
The uniqueness implies $r({\rm Id}_X) = {\rm Id}_{r(X)}$ and $r(g\circ f) = r(g)\circ r(f)$. This gives a functor $r: \mathcal{A}\longrightarrow \mathcal{F}$.
Let $i:\mathcal{F}\longrightarrow \mathcal{A}$ be the inclusion functor. For any morphism $a:X\longrightarrow F$ in $\mathcal{A}$ with $F\in\mathcal F$,
since $F$ is fibrant, by definition there is a unique morphism $\overline a:rX\longrightarrow F$ such that
$\overline a\circ \eta_{_X}=a$:
\[
\xymatrix@C=1.0cm@R=0.7cm{
X\ar[r]^a\ar[d]_{\eta_{_X}} & F.\\
rX\ar@{-->}[ur]_{\overline a}
}
\]
This gives a natural isomorphism $\Phi_{X,F}:\operatorname{Hom}_{\mathcal{F}}(r(X),F)\longrightarrow \operatorname{Hom}_{\mathcal{A}}(X,i(F)), u\longmapsto u\circ\eta_{_X}$.
Thus $(r,i)$ is an adjoint pair.
Hence $\mathcal F$ is reflective, with reflector $r$ and unit $\eta:{\rm Id}_{\mathcal{A}}\longrightarrow i\circ r$.

\vskip 5pt

{\bf Claim:} \ A morphism $u:X\longrightarrow F$ is a left $\mathcal{F}$-strong-approximation of $X$ if and only if $u$ is a trivial cofibration with $F\in\mathcal F$.

In fact, by the definition of the fibrant objects, every trivial cofibration $u:X\longrightarrow F$ with $F$ fibrant is a left $\mathcal F$-strong-approximation of $X$.
Conversely, let $u:X\longrightarrow F$ be a left $\mathcal F$-strong-approximation of $X$.
As we have known,  $\eta_{_X}:X\longrightarrow rX$ is a left $\mathcal F$-strong-approximation of $X$.
By Fact \ref{fact:approximation-basic}(1)(ii) there is a unique isomorphism $\varphi:rX\longrightarrow F$ such that $u=\varphi\circ \eta_{_X}$.
Since $\eta_{_X}\in\TCoFib$ and $\varphi\in \TCoFib$, one has $u\in\TCoFib$.

\vskip 5pt

Now we prove the reflector $r:\mathcal{A}\longrightarrow \mathcal{F}$ is semi-left-exact (cf. Definition \ref{defn:strong-approximation}(3)).
Let $u:X\longrightarrow F$ be a left $\mathcal F$-strong-approximation of $X$ and $g:E\longrightarrow F$ a morphism with $E\in\mathcal F$.
Then $u\in\TCoFib$, by {\bf Claim}. By Proposition \ref{prop:elementary-CF}$(3')$, $g:E\longrightarrow F$ is a fibration.
By ${\rm (M3')}$ in Definition \ref{defn:bi-reflective-orthogonal-model-structure}, one has a pullback square
\[
\xymatrix@C=0.9cm@R=0.6cm{
P\ar[r]\ar[d]_{u'} & X\ar[d]^{u}\\
E\ar[r]^{g} & F
}
\]
such that  $u': P\longrightarrow E$ is a trivial cofibration.  Thus $u'$ is a left $\mathcal F$-strong-approximation of $P$, by {\bf Claim}.
This proves that $r$ is semi-left-exact.

\vskip 5pt

By Lemma~\ref{lem:sle-reflective-ofs}$(1)$, one gets the orthogonal factorization system $(\mathcal E_r, \mathcal M_r)$.

\vskip 5pt

Now we prove  the equality $(\mathcal E_r, \mathcal M_r) = (\TCoFib, \Fib)$. It suffices to prove $$\TCoFib=\{f \in {\rm Mor}(\mathcal A)\ \mid \  r(f)\text{ is an isomorphism}\}=\mathcal E_r.$$
For every morphism $f:X\longrightarrow Y$, one has the following commutative square:
\[
\xymatrix{
X\ar[r]^{\eta_{_X}}\ar[d]_{f} & rX\ar[d]^{r(f)}\\
Y\ar[r]^{\eta_{_Y}} & rY.
}
\]
Since $\eta_{_X}\in\TCoFib$ and $\eta_{_Y}\in\TCoFib$ and $\TCoFib$ satisfies the two-out-of-three property (cf. Definition~\ref{defn:bi-reflective-orthogonal-model-structure}(M2)), one has $f\in\TCoFib$ if and only if $r(f)\in\TCoFib$.

Thus, if $f\in \mathcal E_r,$ then $r(f)$ is an isomorphism, and hence $r(f)\in\TCoFib$, therefore $f\in\TCoFib$, i.e., $\mathcal E_r \subseteq \TCoFib.$ Conversely, if $f\in\TCoFib$, then $r(f)\in\TCoFib$.
However $r(f)$ is a morphism between fibrant objects, it follows from Proposition~\ref{prop:elementary-CF}$(3')$ that $r(f) \in \Fib$. Thus
$r(f) \in \TCoFib \cap \Fib$. It follows that  $r(f)$ is an isomorphism, i.e., $f\in \mathcal E_r$.  This proves  the equality.

\vskip 5pt

Dually, $\mathcal C$ is a coreflective subcategory such that the coreflector $t: \mathcal A \longrightarrow \mathcal C$ is semi-right-exact,
and one has $(\mathcal E_t, \mathcal M_t)=(\CoFib, \TFib)$.

\vskip 5pt

The condition {\rm (A2)} follows from $\mathcal E_r=\TCoFib\perp\TFib=\mathcal M_t$.

\vskip 5pt

To see the condition {\rm (A3)}, i.e., $\mathcal E_r \circ \mathcal M_t\subseteq \mathcal M_t\circ\mathcal E_r$, note that
$\Weq= \TFib\circ \TCoFib = \mathcal M_t\circ\mathcal E_r$. Since $\Weq = \mathcal M_t\circ\mathcal E_r$ has the two-out-of-three property, it follows from Proposition \ref{prop:bi-reflective-A3} that
$\mathcal E_r \circ \mathcal M_t\subseteq \mathcal M_t\circ\mathcal E_r$.

\vskip 5pt
This proves that $(\mathcal C,\mathcal F)$ is a bi-reflective pair. \end{proof}

Now we are in position to state the main result of this section.

\begin{thm}\label{thm:bi-reflective-pairs-bi-reflective-orthogonal-model-structures}\ Let $\mathcal A$ be a  category. Then
\begin{equation*}
\begin{aligned}
\Phi:\{\text{bi-reflective pairs in} \ \mathcal A\} &\longrightarrow \{\text{bi-reflective model structures on} \ \mathcal A\}
\\ (\mathcal C,\mathcal F) &\longmapsto (\mathcal E_t, \ \mathcal M_r, \ \mathcal M_t\circ \mathcal E_r)
\end{aligned}
\end{equation*}
is a bijection, where $t$ is the coreflector of $\mathcal C$ with the counit $\varepsilon$, $r$ is the reflector of $\mathcal F$ with the unit $\eta$,
\begin{equation*}
\begin{aligned}
\mathcal E_t &= \left\{
f:X\longrightarrow Y\ \middle|\begin{array}{c}
\xymatrix@C=0.9cm@R=0.4cm{
tX \ar[r]^-{\varepsilon_{_X}} \ar[d]_{t(f)} & X \ar[d]^{f} \\
tY \ar[r]^-{\varepsilon_{_Y}} & Y
}
\end{array}\text{ is a pushout}
\right\}
\\
\mathcal M_t
&= \{\ f \in {\rm Mor}(\mathcal A)\ \mid \  t(f)\text{ is an isomorphism}\}
\end{aligned}
\end{equation*}
\begin{equation*}
\begin{aligned}
\mathcal E_r &= \{f \in {\rm Mor}(\mathcal A)\ \mid \  r(f)\text{ is an isomorphism}\}
\\
\mathcal M_r
&=
\left\{
f: X\longrightarrow Y\ \middle|\begin{array}{c}
\xymatrix@C=0.9cm@R=0.4cm{
X\ar[r]^-{\eta_{_X}}\ar[d]_f & rX\ar[d]^{r(f)}\\
Y\ar[r]^-{\eta_{_Y}} & rY
}
\end{array}\text{is a pullback}
\right\}.
\end{aligned}
\end{equation*}

\vskip 5pt
\noindent
The inverse of $\Phi$ is $\Psi:(\CoFib, \Fib, \Weq) \longmapsto (\mathcal C,\mathcal F)$, where \ $\mathcal C$ and $\mathcal F$ are respectively the classes of cofibrant objects and fibrant objects.
\end{thm}

\begin{proof}\ By Proposition~\ref{prop:bi-reflective-pair-gives-bi-reflective-model},  the map $\Phi$ is well-defined;
by Proposition~\ref{prop:bi-reflective-model-gives-bi-reflective-pair}, the map $\Psi$ is well-defined.

\vskip5pt

For any bi-reflective pair $(\mathcal C, \mathcal F)$, by Propositions \ref{prop:bi-reflective-pair-gives-bi-reflective-model},
$\Phi(\mathcal C,\mathcal F) = (\mathcal E_t, \ \mathcal M_r, \ \mathcal M_t\circ \mathcal E_r)$ is a bi-reflective model structure, and
$\mathcal{C}$ and $\mathcal{F}$ are precisely the classes of cofibrant objects and fibrant objects of $(\mathcal E_t, \ \mathcal M_r, \ \mathcal M_t\circ \mathcal E_r)$, respectively.
And then by definition $$(\Psi\circ\Phi)(\mathcal C, \mathcal F) = \Psi(\mathcal E_t, \ \mathcal M_r, \ \mathcal M_t\circ \mathcal E_r) = (\mathcal C, \mathcal F).$$

\vskip 5pt
For any bi-reflective model structure $(\CoFib, \Fib, \Weq)$, by Proposition~\ref{prop:bi-reflective-model-gives-bi-reflective-pair},
$\Psi(\CoFib, \Fib, \Weq) = (\mathcal C, \mathcal F)$ is a bi-reflective pair, where $\mathcal{C}$ and $\mathcal{F}$ are the classes of cofibrant objects and fibrant objects of $(\CoFib, \Fib, \Weq)$, respectively.
And then by definition
$$(\Phi\circ\Psi)(\CoFib, \Fib, \Weq) =  \Phi(\mathcal C, \mathcal F) = (\mathcal E_t, \ \mathcal M_r, \ \mathcal M_t\circ \mathcal E_r)$$
and by Proposition~\ref{prop:bi-reflective-model-gives-bi-reflective-pair} one has $(\mathcal E_t,\mathcal M_t)=(\CoFib,\TFib)$ and $(\mathcal E_r, \mathcal M_r) = (\TCoFib, \Fib)$.
Comparing the two model structures
$$(\CoFib, \ \Fib, \ \Weq) \ \ \ \ \mbox{and} \ \ \ \ (\mathcal E_t, \ \mathcal M_r, \ \mathcal M_t\circ \mathcal E_r)$$
with $\CoFib = \mathcal{E}_t$ and $\Fib = \mathcal{M}_r$, it follows that
$$(\Phi\circ\Psi)(\CoFib, \Fib, \Weq) = (\mathcal E_t, \ \mathcal M_r, \ \mathcal M_t\circ \mathcal E_r) = (\CoFib, \ \Fib, \ \Weq).$$
This completes the proof.
\end{proof}

\begin{rem}\label{both-TTF-and-bi-reflective} \ Even on an abelian category, a TTF model structure is not bi-reflective, in general. In fact, a TTF model structure  on an abelian category $\mathcal{A}$, induced by TTF triple $(\mathcal C, \mathcal W, \mathcal F)$,
is bi-reflective if and only if $\mathcal W = 0$.

\vskip5pt

Indeed, if the TTF model structure is bi-reflective, then by Proposition~\ref{prop:TTF-orthogoanl-equivalent-condition},  $\TCoFib=\{f \in {\rm Mor}(\mathcal A)\ \mid \  f \text{ is epimorphism and} \Ker f\in\mathcal{W} \}$.
By Definition \ref{defn:bi-reflective-orthogonal-model-structure}(M2),  $\TCoFib$ has the two-out-of-three property.
For any $W\in\mathcal W$, since $W\longrightarrow 0$ and $0\longrightarrow 0$ are in $\TCoFib$, it follows that $0\longrightarrow W$ is in $\TCoFib$, which implies $W = 0$.
Conversely, if $\mathcal W=0$, then TTF triple $(\mathcal C, \mathcal W, \mathcal F)$ is TTF triple $(\mathcal A, 0, \mathcal A)$, and then the induced TTF model structure is $(\Mor\mathcal{A},\Mor\mathcal{A},\Iso \mathcal{A})$,
which is bi-reflective given by the bi-reflective pair $(\mathcal{A},\mathcal{A})$.
\end{rem}

\subsection{\bf Examples} \ We will give an application of Theorem \ref{thm:bi-reflective-pairs-bi-reflective-orthogonal-model-structures} in the next section. However, we first include some specific examples.

\begin{exm} \  Let $\mathcal A$ be the category given by a totally ordered set $\{0, c, d, 1\}$ with $0 < c < d < 1.$ Put $\mathcal C=\{0,d,1\}$ and $\mathcal F=\{0,c,1\}$.
Then $(\mathcal C,\mathcal F)$ is a bi-reflective pair.

In fact, the coreflector $t:\mathcal A\longrightarrow\mathcal C$ is given by $t(0)=t(c)=0$, $t(d)=d$, $t(1)=1$.  The right $\mathcal C$-strong-approximation of $c$ is $\varepsilon_c:0\longrightarrow c$.
An analysis in detail shows that $t$ is semi-right-exact. Then Lemma \ref{lem:sle-reflective-ofs} gives an  orthogonal factorization system $(\mathcal E_t, \mathcal M_t)$, where
\begin{equation*}
\begin{aligned}
\mathcal E_t
&=
\operatorname{Iso}\mathcal A
\cup
\{0\longrightarrow d,\ 0\longrightarrow1,\
 c\longrightarrow d,\ c\longrightarrow1,\ d\longrightarrow1\},\\
\mathcal M_t
&=
\operatorname{Iso}\mathcal A
\cup\{0\longrightarrow c\},
\end{aligned}
\end{equation*}
Also, the reflector $r:\mathcal A\longrightarrow\mathcal F$ is given by $r(0)=0$, $r(c)=c$, $r(d)=r(1)=1$. The left $\mathcal F$-strong-approximation of $d$ is $\eta_d: d\longrightarrow 1$.
An analysis in detail shows that $r$ is semi-left-exact. Then Lemma \ref{lem:sle-reflective-ofs} gives an  orthogonal factorization system $(\mathcal E_r, \mathcal M_r)$, where
\begin{equation*}
\begin{aligned}
\mathcal E_r
&=
\operatorname{Iso}\mathcal A
\cup\{d\longrightarrow1\},\\
\mathcal M_r
&=
\operatorname{Iso}\mathcal A
\cup
\{0\longrightarrow c,\ 0\longrightarrow d,\ 0\longrightarrow1,\
 c\longrightarrow d,\ c\longrightarrow1\}.
\end{aligned}
\end{equation*}
Clearly one has \textup{(A2)}: $\mathcal E_r\perp \mathcal M_t$ and \textup{(A3)}: $\mathcal E_r\circ \mathcal M_t\subseteq \mathcal M_t\circ \mathcal E_r$.
This justifies that $(\mathcal C,\mathcal F)$ is a bi-reflective pair.

\vskip5pt

By Proposition~\ref{prop:bi-reflective-pair-gives-bi-reflective-model}, one gets a bi-reflective model structure $(\mathcal E_t,\mathcal M_r,\mathcal M_t\circ \mathcal E_r)$, where
$$\mathcal M_t\circ \mathcal E_r = \operatorname{Iso}\mathcal A
\cup\{d\longrightarrow1, 0\longrightarrow c\}.$$
\end{exm}

\vskip5pt

\begin{exm}\  Let $G$ be a group. Consider the reflective subcategory $\mathsf{Set}$ of $G\text{-}\mathsf{Set}$, with semi-left-exact reflector $r$, as given in Proposition~\ref{reflG-set}.
This gives an  orthogonal factorization system $(\mathcal E_{\mathrm{orb}},\mathcal M_{\mathrm{orb}})$, where
\begin{equation*}
\begin{aligned}
\mathcal E_{\mathrm{orb}}
&=
\{\ G\text{-map} \ f:X\longrightarrow Y \  \mid \ r(f):X/G\longrightarrow Y/G
 \text{ is a bijection}\}\\
\mathcal M_{\mathrm{orb}}
&=
\left\{G\text{-map} \ f:X\longrightarrow Y \ \middle| \ \begin{array}{c}
\xymatrix@C=1.0cm@R=0.4cm{
X\ar[r]^-{\eta_{_X}}\ar[d]_{f} & X/G\ar[d]^-{r(f)}\\
Y\ar[r]^-{\eta_{_Y}} & Y/G
}
\end{array}\text{ is a pullback in} \ G\text{-}\mathsf{Set}
\right\}.
\end{aligned}
\end{equation*}

\noindent Also, $G\text{-}\mathsf{Set}$ is a coreflective subcategory of itself, with identity as the semi-right-exact coreflector. Then Lemma \ref{lem:sle-reflective-ofs} gives
an  orthogonal factorization system $(\Mor G\text{-}\mathsf{Set}, \Iso G\text{-}\mathsf{Set})$.
Then one has a bi-reflective pair $(G\text{-}\mathsf{Set}, \mathsf{Set})$, and then by Proposition~\ref{prop:bi-reflective-pair-gives-bi-reflective-model}, one has a bi-reflective model structure $(\Mor G\text{-}\mathsf{Set},\mathcal{M}_{\mathrm{orb}},\mathcal{E}_{\mathrm{orb}})$
on $G\text{-}\mathsf{Set}$.
This is a cofibrant model structure,  in the sense that each object is cofibrant.
\end{exm}

\begin{exm}\ The category $\mathbb Q\text{-Vect}$ is a reflective subcategory of $\operatorname{Ab}$ with semi-left-exact reflector $-\otimes_{\mathbb Z}\mathbb Q: \operatorname{Ab} \longrightarrow \mathbb Q\text{-Vect}$,
as given in Proposition~\ref{reflAb}. This gives an  orthogonal factorization system $(\mathcal E_{\mathbb Q}, \mathcal M_{\mathbb Q})$, where
\begin{equation*}
\begin{aligned}
\mathcal E_{\mathbb Q}
&=
\{f \in {\rm Mor}({\rm Ab})\ \mid \  \Ker f\text{ and }\operatorname{Coker}f
 \text{ are abelian torsion groups}\}\\
\mathcal M_{\mathbb Q}
&=
\left\{
f:A\longrightarrow B\ \middle|\begin{array}{c}
\xymatrix@C=1.0cm@R=0.4cm{
A\ar[r]^{\eta_{_A}}\ar[d]_{f}
 & A\otimes_{\mathbb Z}\mathbb Q
     \ar[d]^{f\otimes_{\mathbb Z}\mathbb Q}\\
B\ar[r]_{\eta_{_B}}
 & B\otimes_{\mathbb Z}\mathbb Q
}
\end{array}\text{ is a pullback in} \ {\rm Ab}
\right\}.
\end{aligned}
\end{equation*}

\noindent Also, $\operatorname{Ab}$ is a coreflective subcategory of itself, and this leads to an orthogonal factorization system $(\Mor \operatorname{Ab}, \Iso \operatorname{Ab})$.
Then one has a bi-reflective pair $(\operatorname{Ab},\mathbb{Q}\text{-Vect})$, and then by Proposition~\ref{prop:bi-reflective-pair-gives-bi-reflective-model}, one has a bi-reflective model structure
$(\Mor\operatorname{Ab}, \mathcal{M}_{\mathbb{Q}},\mathcal{E}_{\mathbb{Q}})$ on $\operatorname{Ab}$.
This is a cofibrant model structure, however, it is neither an abelian model structure (\cite{H2, G2}), nor an $\omega$-model structure, nor a $\mathcal W$-model structure (cf. \cite{BR}, \cite{LZ}).
\end{exm}

\section{\bf Torsion model structures}

This section introduces torsion model structures, together with their basic properties and examples.

\subsection{\bf Torsion model structures on abelian categories}

\begin{defn}\label{defn:torsion-OFS-and-torsion-model} \ Let $\mathcal A$ be a category.

\vskip5pt

$(1)$ \ {\rm(\cite[4.4]{RT})} \  A \emph{torsion factorization system} in $\mathcal A$ is an orthogonal factorization system $(\mathcal E, \mathcal M)$ such that both $\mathcal E$ and $\mathcal M$ have the two-out-of-three property.

\vskip5pt

$(2)$ \ An orthogonal model structure $(\CoFib,\Fib,\Weq)$ on $\mathcal A$ is a \emph{torsion model structure} if $\CoFib$ and $\Fib$ have the two-out-of-three property.

\vskip5pt

$(3)$ {\rm (\cite{T})}\ Assume that $\mathcal A$ is an abelian category. {\it A twin torsion pair} in $\mathcal A$ is a quadruple $(\mathcal T_1, \mathcal F_1; \mathcal T_2, \mathcal F_2)$,  where
$(\mathcal T_1, \mathcal F_1)$ and $(\mathcal T_2, \mathcal F_2)$ are torsion pairs in $\mathcal A$, such that $\mathcal T_2 \subseteq \mathcal T_1$.

\end{defn}

In analogy with Proposition~\ref{prop:weq-two-out-of-three}, one has the following criterion of torsion model structures in terms of two torsion factorization systems.

\begin{prop}\label{prop:weq-torsion}\ Let $\mathcal A$ be a category,  $(\TCoFib, \Fib)$ and $(\CoFib,\TFib)$ be torsion factorization systems in $\mathcal A$, with $\TCoFib\subseteq\CoFib$. Put $\Weq:=\TFib\circ\TCoFib$. Then $(\CoFib,\Fib,\Weq)$ is a torsion model structure if and only if $\Weq$ is closed under compositions.
\end{prop}

\begin{rem}\label{prop:torsion-ofs-trivial-torsion-model}\
$(1)$ \ Every torsion factorization system $(\mathcal E,\mathcal M)$ in $\mathcal A$ determines the three torsion model structures:
$(\mathcal E,\mathcal M,\operatorname{Mor}\mathcal A), \ (\operatorname{Mor}\mathcal A,\mathcal M,\mathcal E), \ (\mathcal E,\operatorname{Mor}\mathcal A,\mathcal M).$

\vskip5pt
$(2)$ \
There is a one to one correspondence between torsion factorization systems and trivial torsion model structures on $\mathcal A$ with:
\[
(\mathcal E,\mathcal M)\longmapsto(\mathcal E,\mathcal M,\operatorname{Mor}\mathcal A),
\qquad
(\CoFib,\Fib,\operatorname{Mor}\mathcal A)\longmapsto(\CoFib,\Fib).
\]
\end{rem}

\vskip5pt

In an abelian category, we will see that, there is a one to one correspondence between torsion model structures and twin torsion pairs.
This is based on a one to one correspondence between torsion factorization systems and torsion pairs, as follows.

\vskip 5pt

{\bf Notations:} \  Let $(\mathcal T,\mathcal F)$ be a torsion pair in an abelian category $\mathcal A$.
For $X\in \mathcal A$, fix an exact sequence
$$0\longrightarrow t(X)\longrightarrow X \longrightarrow r(X)\longrightarrow 0 $$
with $t(X)\in\mathcal{T}$ and  $r(X)\in\mathcal{F}$.
For $f:X\longrightarrow Y$, there are unique morphisms $t(f)$ and $r(f)$ such that the diagram
\[
\xymatrix@C=0.9cm@R=0.7cm{
0\ar[r] & t(X)\ar[r]\ar[d]_{t(f)} & X\ar[r]\ar[d]^{f} & r(X)\ar[r]\ar[d]^{r(f)} & 0\\
0\ar[r] & t(Y)\ar[r] & Y\ar[r] & r(Y)\ar[r] & 0
}
\]

\vskip5pt \noindent commutes. This gives functors $t:\mathcal A\longrightarrow \mathcal T$ and $r:\mathcal A\longrightarrow \mathcal F$.
When the two functors need to be specified, we will denote a torsion pair $(\mathcal T,\mathcal F)$ by $(\mathcal T, \mathcal F, t, r)$, and a twin torsion pair  $(\mathcal T_1, \mathcal F_1; \mathcal T_2, \mathcal F_2)$ by
$$(\mathcal T_1, \mathcal F_1, t_1, r_1; \mathcal T_2, \mathcal F_2, t_2, r_2).$$

\vskip 5pt

\begin{prop}\label{prop:abelian-torsion-OFS} {\rm (\cite[Theorem~5.2]{RT})} \ Let $\mathcal A$ be an abelian category.

\vskip5pt

$(1)$ \  Let $(\mathcal T,\mathcal F, t ,r)$ be a torsion pair in $\mathcal A$. Then $\mathcal T$ is a coreflective subcategory of $\mathcal A$ with semi-right-exact coreflector $t$, and $\mathcal F$ is a reflective subcategory
of $\mathcal A$ with semi-left-exact reflector $r$, such that the associated orthogonal factorization systems given in {\rm Lemma~\ref{lem:sle-reflective-ofs}} coincide,
i.e., $(\mathcal E_t,\mathcal M_t)=(\mathcal E_r,\mathcal M_r)$, where
$$\mathcal E_r = \{f \in {\rm Mor}(\mathcal A) \mid   r(f)\text{ is isomorphism}\},
\ \ \mathcal M_r =  \{f \in {\rm Mor}(\mathcal A) \mid  t(f)\text{ is isomorphism}\}.$$
Moreover, it is a torsion factorization system.

\vskip5pt

$(2)$ \ There is a one to one correspondence
$$\Phi:\{\text{torsion pairs in} \ \mathcal A\}\longrightarrow \{\text{torsion factorization systems in} \ \mathcal A\}, \ \  (\mathcal{T},\mathcal{F},t,r)\longmapsto (\mathcal E_r,\mathcal M_r)$$
with the inverse $\Psi:(\mathcal{E},\mathcal{M})\longmapsto (\mathcal{T},\mathcal{F}),$
where $$\mathcal{T} = \{X\mid 0\longrightarrow X \ \mbox{is in} \ \mathcal E\},  \ \ \ \ \mathcal{F} = \{X\mid X\longrightarrow 0\ \mbox{is in} \ \mathcal M\}.$$
\end{prop}

\begin{proof}  This is a main result in \cite{RT}, and it is stated for normal torsion factorization systems in homological categories. Since we need to use this result, for convenience we include a direct proof for abelian categories.

\vskip5pt
$(1)$ \
Let $(\mathcal T,\mathcal F,t,r)$ be a torsion pair in the abelian category
$\mathcal A$. Thus, for every $X\in\mathcal A$, there is an exact
sequence
\[
0\longrightarrow t(X)\overset{\varepsilon_{_X}}{\longrightarrow}X
\overset{\eta_{_X}}{\longrightarrow}r(X)\longrightarrow 0,
\]
where $t(X)\in\mathcal T$ and $r(X)\in\mathcal F$.

\vskip5pt

We first show that $\mathcal T$ is a coreflective subcategory of $\mathcal A$. Let $T\in\mathcal T$ and let $f:T\longrightarrow X$ be a morphism. Since $\eta_{_X}f:T\longrightarrow r(X)$ is zero by $\Hom_{\mathcal A}(\mathcal T,\mathcal F)=0$, there is a unique morphism $l:T\longrightarrow t(X)$ such that $f = l\circ \varepsilon_{_X}$. Hence $\Hom_{\mathcal{A}}(T,X)\cong \Hom_{\mathcal{T}}(T,t(X))$, and it follows that $t$ is a coreflector of $\mathcal T$. Moreover, the argument above implies that $\varepsilon_{_X}:t(X) \longrightarrow X$ is the right $\mathcal T$-strong approximation of $X$.

\vskip5pt

We next prove that $t$ is semi-right-exact. Let $\varepsilon_{_X}:t(X)\longrightarrow X$ be the right $\mathcal T$-strong approximation of $X$ and let $h:t(X)\longrightarrow T$ be a morphism with $T\in\mathcal T$.
Consider the pushout diagram
\[
\xymatrix@C=1.1cm@R=0.8cm{
  t(X)\ar@{^{(}->}[r]^-{\varepsilon_{_X}}\ar[d]_-{h} & X\ar@{->>}[r]\ar[d] & r(X)\ar@{=}[d] \\
  T\ar@{^{(}->}[r]^-{c} & Q\ar@{->>}[r] & r(X)
}\]

Since $\varepsilon_{_X}$ is a monomorphism and $\mathcal A$ is abelian, $c$ is also a monomorphism.
Since the second row of the diagram above is an exact sequence with $T \in \mathcal T$ and  $r(X)\in \mathcal F$, there is a unique isomorphism $u:T \longrightarrow t(Q)$ such that $\varepsilon_{_Q}\circ u=c$. Thus $c$ is the right $\mathcal T$-strong approximation of $Q$ and hence $t$ is semi-right-exact.

\vskip5pt

Dually, $\mathcal F$ is a reflective subcategory of $\mathcal A$ and  $r$ is the semi-left-exact reflector.

\vskip5pt

By Lemma~\ref{lem:sle-reflective-ofs} one gets two orthogonal
factorization systems $(\mathcal E_t,\mathcal M_t)$ and $(\mathcal E_r,\mathcal M_r),$
where
\begin{equation*}
\begin{aligned}
\mathcal E_t &= \left\{
f:X\longrightarrow Y\ \middle|\begin{array}{c}
\xymatrix@C=0.9cm@R=0.4cm{
t(X) \ar[r]^-{\varepsilon_{_X}} \ar[d]_{t(f)} & X \ar[d]^{f} \\
t(Y) \ar[r]^-{\varepsilon_{_Y}} & Y
}
\end{array}\text{ is a pushout}
\right\}
\\
\mathcal M_t
&= \{\ f \in {\rm Mor}(\mathcal A)\ \mid \  t(f)\text{ is an isomorphism}\}
\end{aligned}
\end{equation*}
\begin{equation*}
\begin{aligned}
\mathcal E_r &= \{f \in {\rm Mor}(\mathcal A)\ \mid \  r(f)\text{ is an isomorphism}\}
\\
\mathcal M_r
&=
\left\{
f: X\longrightarrow Y\ \middle|\begin{array}{c}
\xymatrix@C=0.9cm@R=0.4cm{
X\ar[r]^-{\eta_{_X}}\ar[d]_f & r(X)\ar[d]^{r(f)}\\
Y\ar[r]^-{\eta_{_Y}} & r(Y)
}
\end{array}\text{is a pullback}
\right\}.
\end{aligned}
\end{equation*}

\vskip5pt

We now show that $\mathcal E_t=\mathcal E_r$.
Let $f:X\longrightarrow Y$ be a morphism in $\mathcal A$. Then one has the following diagram with exact rows
\[
\xymatrix@C=0.9cm@R=0.7cm{
0 \ar[r] & t(X) \ar[r]\ar[d]_{t(f)} & X \ar[r]\ar[d]^{f} & r(X) \ar[r]\ar[d]^{r(f)} & 0
\\
0 \ar[r] & t(Y) \ar[r] & Y \ar[r] & r(Y) \ar[r] & 0.
}
\]
Note that $r(f)$ is an isomorphism if and only if the left-hand-side square is a pushout square. It follows that $f \in \mathcal E_t$ if and only if $f \in \mathcal E_r$.

\vskip 5pt

Since $(\mathcal E_t,\mathcal M_t)$ and $(\mathcal E_r,\mathcal M_r)$ are orthogonal factorization systems with the same left class, one has $(\mathcal E_t,\mathcal M_t) = (\mathcal E_r,\mathcal M_r).$
By definition, both $\mathcal E_r$ and $\mathcal M_r$ have the
two-out-of-three property. Thus $(\mathcal E_r,\mathcal M_r)$ is a torsion factorization system.

\vskip 5pt
$(2)$ \  Let $(\mathcal T,\mathcal F)$ be a torsion pair.
Set $\mathcal{E} = \{f \in {\rm Mor}(\mathcal A)\ \mid \  r(f)\text{ is an isomorphism}\}$ and $\mathcal{M} = \{f \in {\rm Mor}(\mathcal A)\ \mid \  t(f)\text{ is an isomorphism}\}$.
We have shown that $(\mathcal{E},\mathcal{M})$ is a torsion factorization system. Thus $\Phi$ is well-defined.

\vskip 5pt

Conversely, let $(\mathcal E,\mathcal M)$ be a torsion factorization system. Set $\mathcal T:=\{X\mid 0\longrightarrow X\in\mathcal E\}$ and $\mathcal F:=\{X\mid X\longrightarrow0\in\mathcal M\}$.
Now we prove $(\mathcal T,\mathcal F)$ is a torsion pair.
For $T\in\mathcal T$ and $F\in\mathcal F$, every morphism $h:T\longrightarrow F$ is a lifting in the commutative square
\[
\xymatrix@C=1.0cm@R=0.7cm{
0\ar[r]\ar[d] & F\ar[d]\\
T\ar[r]\ar@{-->}[ur]^{h} & 0.
}
\]
By definition of torsion factorization system, the lifting is unique. Since the zero morphism is also a lifting, one has $\Hom_{\mathcal A}(\mathcal T,\mathcal F)=0$, and hence $\mathcal F\subseteq\mathcal T^{\perp_0}$.

\vskip 5pt

We claim that $\mathcal F=\mathcal T^{\perp_0}$.
In fact, let $X\in\mathcal T^{\perp_0}$.
Factorize $0\longrightarrow X$ as $0\longrightarrow T_X\overset{i_X}{\longrightarrow}X$, with $0\longrightarrow T_X$ in $\mathcal E$ and $i_X$ in $\mathcal M$.
Then one has $T_X \in \mathcal T$ by definition. Since $X\in\mathcal T^{\perp_0}$, it follows that $i_X=0$. Thus both $0$ and ${\rm Id}_{T_X}$ are liftings in
\[
\xymatrix@C=1.0cm@R=0.7cm{
0\ar[r]\ar[d] & T_X\ar[d]^{i_X}\\
T_X\ar[r]_{i_X}\ar@{-->}[ur]_0^{{\rm Id}_{T_X}} & X,
}
\]
and hence $T_X=0$.
Thus $0\longrightarrow X$ lies in $\mathcal M$.
Since the composition of $0\longrightarrow X\longrightarrow0$ lies in $\mathcal M$ and $\mathcal M$ has the two-out-of-three property, one has $X\longrightarrow0$ in $\mathcal M$ and hence $X\in\mathcal F$.
Thus $\mathcal T^{\perp_0}\subseteq \mathcal F$ and it follows that $\mathcal F=\mathcal T^{\perp_0}$.
The dual of the claim gives $\mathcal T={}^{\perp_0}\mathcal F$.
The claim implies that $\mathcal T$ is closed under quotients and extensions; and $\mathcal F$ is closed under subobjects and extensions.

\vskip 5pt

Now we prove that every object $X\in\mathcal A$
admits an exact sequence $0\longrightarrow T\longrightarrow X\longrightarrow F\longrightarrow0$ with $T\in\mathcal T$ and $F\in\mathcal F$.
For $X\in\mathcal A$, factorize $0\longrightarrow X$ as $0\longrightarrow T_X\overset{i_X}{\longrightarrow}X$ with $T_X \in \mathcal T$ and $i_X \in \mathcal M$.
Similarly, factorize $X\longrightarrow0$ as $X\overset{p_X}{\longrightarrow}F_X\longrightarrow0$ with $F_X \in \mathcal F$ and $p_X \in \mathcal E$.
For any $s:T\longrightarrow X$ with $T \in \mathcal{T}$, there is a morphism $\tau:T\longrightarrow T_X$ such that the following diagram commutes:
\[
\xymatrix@C=1.0cm@R=0.7cm{
0\ar[r]\ar[d] & T_X\ar[d]^{i_X}\\
T\ar[r]_{s}\ar@{-->}[ur]\ar@{-->}[ur]\ar@{-->}[ur]^{\tau} & X.
}
\]
Thus $i_X$ is the strong right $\mathcal T$-approximation of $X$ (cf. Definition~\ref{defn:strong-approximation}). Similarly, $p_X$ is a left $\mathcal F$-strong-approximation of $X$.
Consider epi--mono factorization of $i_X$: $T_X \overset{q}{\longrightarrow} \operatorname{Im}i_X \overset{j}{\longrightarrow} X$. Since $\operatorname{Im}i_X$ is a quotient of $T_X\in\mathcal T$ and $\mathcal T={}^{\perp_0}\mathcal F$ is closed under quotients, one has $\operatorname{Im}i_X\in\mathcal T$.
Since $i_X$ is strong right $\mathcal T$-approximation of $X$, there is a unique morphism $h:\operatorname{Im}i_X\longrightarrow T_X$ such that $j=i_X\circ h$.
Since $j=i_X\circ h=j\circ q \circ h$ and $j$ is a monomorphism, one gets $q\circ h={\rm Id}_{\operatorname{Im}i_X}$.
Consider the following commutative square:
\[
\xymatrix@C=1.0cm@R=0.7cm{
0\ar[r]\ar[d] & T_X\ar[d]^{q}\\
T_X\ar[r]_{q}\ar@{-->}[ur]^{hq}_{{\rm Id}_{T_X}} & \operatorname{Im}i_X,
}
\]
where $h\circ q$ and ${\rm Id}_{T_X}$ are both liftings, it follows that $h\circ q={\rm Id}_{T_X}$.
Thus $q$ is an isomorphism and $i_X$ is monic.
Dually, $p_X$ is an epimorphism.

\vskip 5pt

It remains to prove that $T_X = \Ker p_X$. Let $k:K\longrightarrow X$ be the kernel of $p_X$.
For every $f:K\longrightarrow F$ with $F\in\mathcal F$, form the pushout
\[
\xymatrix@C=0.9cm@R=0.7cm{
0\ar[r] & K\ar[r]^{k}\ar[d]_{f} & X\ar[r]^{p_X}\ar[d]^{f'} & F_X\ar[r]\ar@{=}[d] & 0\\
0\ar[r] & F\ar[r]^{\alpha} & E\ar[r]^{\beta} & F_X\ar[r] & 0.
}
\]
Since $\mathcal F$ is closed under extensions, $E\in\mathcal F$.
Then there exists a unique morphism $s:F_X\longrightarrow E$ such that $s\circ p_X=f'$.
Hence $\alpha\circ f=f'\circ k=s\circ p_X\circ k=0$.
Since $\alpha$ is monic, one has $f=0$.
Thus $K\in{}^{\perp_0}\mathcal F=\mathcal T$.

\vskip 5pt

Since $\operatorname{Hom}_{\mathcal{A}}(T_X,F_X) = 0$, one has $p_X\circ i_X=0$. There is a unique $a:T_X\longrightarrow K$ with $i_X=k\circ a$.
$K\in\mathcal T$ gives a unique $b:K\longrightarrow T_X$ with $k=i_X\circ b$.
Since $i_X$ and $k$ are monic, one has $b\circ a=1_{T_X}$ and $a\circ b=1_K$.
Therefore $T_X\cong K$, and hence $0\longrightarrow T_X\overset{i_X}{\longrightarrow}X\overset{p_X}{\longrightarrow}F_X\longrightarrow0$ is exact.
Thus $(\mathcal T,\mathcal F, t ,r)$ is a torsion pair, where $t(X)=T_X$ and $r(X)=F_X$.
This shows that $\Phi$ is well-defined.

\vskip 5pt

We now prove that the maps $\Phi$ and $\Psi$ are bijection.

\vskip 5pt

First, let $(\mathcal T,\mathcal F,t,r)$ be a torsion pair. Then $\Phi(\mathcal T,\mathcal F)=(\mathcal E,\mathcal M),$
where
\[
\mathcal E = \{f\in\operatorname{Mor}(\mathcal A)\mid r(f)\text{ is an isomorphism}\},
\ \ \ \ \
\mathcal M = \{f\in\operatorname{Mor}(\mathcal A)\mid t(f)\text{ is an isomorphism}\}.
\]
Applying $\Psi$, one gets a torsion pair $(\mathcal T',\mathcal F')$, with
\[
\mathcal T' = \{X\in\mathcal A\mid 0\longrightarrow X\text{ lies in }\mathcal E\},
\ \ \ \ \
\mathcal F' = \{X\in\mathcal A\mid X\longrightarrow0\text{ lies in }\mathcal M\}.
\]
By the definition of $\mathcal E$, one has
\[
\mathcal T' = \{X\in\mathcal A\mid 0\longrightarrow X\text{ lies in }\mathcal E\} = \{X\in\mathcal A\mid rX\cong 0\} = \mathcal T.
\]
Thus $\mathcal T'=\mathcal T$ and $\mathcal F'=\mathcal F$, and hence $$(\Psi\circ\Phi)(\mathcal T,\mathcal F) = \Psi(\mathcal E,\mathcal M)= (\mathcal T,\mathcal F).$$

\vskip5pt

Conversely, let $(\mathcal E,\mathcal M)$ be a torsion factorization
system. Then $\Psi(\mathcal E,\mathcal M)=(\mathcal T,\mathcal F,t,r)$, where
\[
\mathcal T = \{X\in\mathcal A\mid 0\longrightarrow X\text{ lies in }\mathcal E\},
\ \ \ \ \
\mathcal F = \{X\in\mathcal A\mid X\longrightarrow0\text{ lies in }\mathcal M\}.
\]
Applying $\Phi$ one gets a torsion factorization system $(\mathcal E',\mathcal M')=\Phi(\mathcal T,\mathcal F,t,r),$
where
\[
\mathcal E' = \{f\in\operatorname{Mor}(\mathcal A)\mid r(f)\text{ is an isomorphism}\}, \ \ \ \
\mathcal M' = \{f\in\operatorname{Mor}(\mathcal A)\mid t(f)\text{ is an isomorphism}\}.
\]

\vskip5pt

It suffices to prove that $(\mathcal E',\mathcal M')=(\mathcal E,\mathcal M).$

\vskip5pt

For every $X\in\mathcal A$, choose a factorization $X\overset{\eta_{_X}}{\longrightarrow}F_X \overset{q_{_X}}{\longrightarrow}0,$ with $\eta_{_X}\in\mathcal E$ and $q_{_X}\in\mathcal M$.
By the construction of $(\mathcal T,\mathcal F)$ above, one has $F_X=r(X)$.
Then there exists the following commutative diagram
\[
\xymatrix@C=1.1cm@R=0.8cm{
X \ar[r]^-{\eta_{_X}} \ar[d]_{f} & F_X=r(X) \ar[r]^-{q_{_X}} \ar[d]^{r(f)} & 0 \ar@{=}[d] \\
Y \ar[r]_-{\eta_{_Y}} & F_Y=r(Y) \ar[r]_-{q_{_Y}} & 0.
}
\]
If $f:X \longrightarrow Y$ in $\mathcal E$, since $\eta_{_X}\in\mathcal E$ and $\eta_{_Y}\in\mathcal E$, one has $r(f)\in\mathcal E$ by the two-out-of-three property of $\mathcal E$.
Since $q_{_X}\in\mathcal M$ and $q_{_Y}\in\mathcal M$, one has $r(f)\in\mathcal M$ by the two-out-of-three property of $\mathcal M$.
Thus $r(f)\in \mathcal E\cap\mathcal M=\operatorname{Iso}(\mathcal A)$ and hence $f \in \mathcal E'$.

\vskip5pt

Conversely, if $f\in\mathcal E'$, one has $r(f)$ is an isomorphism.
Since $\eta_{_X}\in\mathcal E$ and $\eta_{_Y}\in\mathcal E$, by the two-out-of-three property of $\mathcal E$ one has $f\in\mathcal E$.
Hence $\mathcal E=\mathcal E'$.
Since $(\mathcal E,\mathcal M)$ and $(\mathcal E',\mathcal M')$ have the same left class, one has $(\Phi\circ\Psi)(\mathcal E,\mathcal M) = (\mathcal E',\mathcal M') = (\mathcal E,\mathcal M).$
This ends the proof.
\end{proof}

\vskip5pt

\begin{prop}\label{prop:exact-nested-torsion-pairs-bi-reflective} \ Let  $(\mathcal T_1, \mathcal F_1, t_1, r_1; \mathcal T_2, \mathcal F_2, t_2,r_2)$ be a twin torsion pair in abelian category $\mathcal A$.
Then $(\CoFib,\Fib,\Weq)$ is a torsion model structure on $\mathcal A$, where
\begin{equation*}
\begin{aligned}
\CoFib & =\{f\in {\rm Mor}(\mathcal A) \ \mid \ r_1(f) \text{ is an isomorphism}\ \}
\\
\Fib &=\{f\in {\rm Mor}(\mathcal A) \ \mid \ t_2(f) \text{ is an isomorphism}\ \}
\\
\Weq &=\{f\in {\rm Mor}(\mathcal A) \ \mid \ H(f) \text{ is an isomorphism}\ \}
\end{aligned}
\end{equation*}
with $ H\cong t_1\circ r_2 \cong r_2\circ t_1: \mathcal A \longrightarrow \mathcal T_1\cap \mathcal F_2.$
\end{prop}

\begin{proof}
For every $X\in\mathcal A$, one has the following commutative diagram with exact rows:
\[
\xymatrix@C=0.9cm@R=0.7cm{
0\ar[r] & t_2(X)\ar[r]^{\varepsilon_2}\ar[d]_{u_{_X}} & X\ar[r]^{\eta_2}\ar@{=}[d] & r_2(X)\ar[r]\ar[d]^{v_{_X}} & 0\\
0\ar[r] & t_1(X)\ar[r]^{\varepsilon_1} & X\ar[r]^{\eta_1} & r_1(X)\ar[r] & 0,
}
\]
where $t_1(X) \in \mathcal T_1, \ t_2(X) \in \mathcal T_2$, $r_1(X) \in \mathcal F_1$ and $r_2(X) \in \mathcal F_2$. Since by definition $\mathcal T_2 \subseteq \mathcal T_1$,
$\mathcal T_2 \cap \mathcal F_1=0$. Thus there is a unique morphism $u_{_X}:t_2(X) \longrightarrow t_1(X)$ such that $\varepsilon_1 \circ u_{_X} =\varepsilon_2$, and there is a unique morphism $v_{_X}:r_2(X) \longrightarrow r_1(X)$ such that $v_{_X} \circ \eta_2=\eta_1$.
By the snake lemma, one gets $\Ker v_{_X}\cong \Coker u_{_X}$. Put $H(X) = \Ker v_{_X}$. Then one has exact sequences:
\begin{equation}
    0\longrightarrow t_2(X)\overset{u_{_X}}{\longrightarrow}t_1(X)\longrightarrow H(X)\longrightarrow0
\end{equation}
\vskip -15pt
\begin{equation}
    0\longrightarrow H(X)\longrightarrow r_2(X)\overset{v_{_X}}{\longrightarrow}r_1(X)\longrightarrow0.
\end{equation}
Since $t_1(X) \in \mathcal T_1$ and $r_2(X) \in \mathcal F_2$, one has $H(X)\in\mathcal T_1\cap\mathcal F_2$. By the functoriality of the connecting morphism one gets a functor $H: \mathcal A \longrightarrow \mathcal T_1\cap \mathcal F_2$.
Since $t_2 (X) \in \mathcal T_2$ and $H(X) \in \mathcal F_2$, by $(6.1)$ one has $H(X) \cong (r_2 \circ t_1)X$.
Similarly, by $(6.2)$ one has $H(X) \cong  (t_1\circ r_2)X$.
Thus $H\cong r_2\circ t_1\cong t_1\circ r_2$.

\vskip 5pt

By Proposition~\ref{prop:abelian-torsion-OFS}, one gets two torsion factorization systems $(\mathcal E_1,\mathcal M_1)$ and $(\mathcal E_2, \mathcal M_2)$, where
\begin{equation*}
\begin{aligned}
\mathcal E_1 & =\{f\in {\rm Mor}(\mathcal A) \ \mid \ r_1(f) \text{ is an isomorphism}\ \}
\\
\mathcal M_1 &=\{f\in {\rm Mor}(\mathcal A) \ \mid \ t_1(f) \text{ is an isomorphism}\ \}
\\
\mathcal E_2 & =\{f\in {\rm Mor}(\mathcal A) \ \mid \ r_2(f) \text{ is an isomorphism}\ \}
\\
\mathcal M_2 &=\{f\in {\rm Mor}(\mathcal A) \ \mid \ t_2(f) \text{ is an isomorphism}\  \}.
\end{aligned}
\end{equation*}

Set $\CoFib:=\mathcal E_1$, $\Fib:=\mathcal M_2$, and $\Weq:=\mathcal M_1\circ\mathcal E_2$.
We claim that $$\Weq = \{f \mid H(f) \text{ is an isomorphism}\}.$$

Let $f\in \Weq = \mathcal M_1\circ\mathcal E_2$. Then $f = m_1\circ e_2$ with $m_1\in\mathcal M_1$ and $e_2\in\mathcal E_2$.
Thus
\[
H(f) = H(m_1)\circ H(e_2) \cong r_2(t_1(m_1))\circ t_1(r_2(e_2)).
\]
By definition $t_1(m_1)$ and $r_2(e_2)$ are isomorphisms, and hence $H(f)$ is an isomorphism.
This proves $\Weq =\mathcal M_1\circ\mathcal E_2\subseteq \{f \mid H(f) \text{ is an isomorphism}\}$.

\vskip 5pt

Conversely, let $H(f)$ be an isomorphism. Factorize $f=m_2\circ e_2$ with $e_2:X\longrightarrow Z$ in $\mathcal E_2$ and $m_2:Z\longrightarrow Y$ in $\mathcal{M}_2$.
Thus $H(e_2)$ is an isomorphism, so is $H(m_2)$. Then $(6.1)$ gives commutative diagram with exact rows:
\[
\xymatrix@C=0.9cm@R=0.7cm{
0\ar[r] & t_2(Z)\ar[r]^{u_Z}\ar[d]_{t_2(m_2)}^{\cong} & t_1(Z)\ar[r]\ar[d]^{t_1(m_2)} & H(Z)\ar[r]\ar[d]^{H(m_2)}_{\cong} & 0\\
0\ar[r] & t_2(Y)\ar[r]_{u_Y} & t_1(Y)\ar[r] & H(Y)\ar[r] & 0.
}
\]
By the snake lemma $t_1(m_2)$ is an isomorphism, so $m_2\in\mathcal M_1$ and $f\in\mathcal M_1\circ\mathcal E_2 =\Weq$.
Thus $\Weq=\{f \mid H(f) \text{ is an isomorphism}\}$, and  then $\Weq$ is closed under compositions.

\vskip 5pt

It follows from  Proposition~\ref{prop:weq-torsion} that $(\CoFib,\Fib,\Weq)$ is a torsion model structure.
\end{proof}

\vskip 5pt

\begin{thm}\label{thm:abelian-bijection-models-torsion-pairs} \ Let $\mathcal A$ be an abelian category.

\vskip 5pt

$(1)$ \ There is a one to one correspondence between torsion model structures on $\mathcal A$ and twin torsion pairs in $\mathcal A$.
Explicitly,
$$\Phi:  \left\{\text{\rm twin torsion pairs} \ \mbox{in} \ \mathcal A\right\} \longrightarrow \left\{\text{\rm torsion model structures on }\mathcal A \right\}$$
$$(\mathcal T_1, \mathcal F_1, t_1, r_1; \mathcal T_2, \mathcal F_2, t_2, r_2)  \longmapsto (\CoFib, \Fib, \Weq)$$
is a bijection, where
\begin{equation*}
\begin{aligned}
\CoFib & =\{f\in {\rm Mor}(\mathcal A) \ \mid \ r_1(f) \text{ is an isomorphism}\ \}
\\
\Fib &=\{f\in {\rm Mor}(\mathcal A) \ \mid \ t_2(f) \text{ is an isomorphism}\ \}
\\
\Weq &=\{f\in {\rm Mor}(\mathcal A) \ \mid \ H(f) \text{ is an isomorphism}\ \}
\end{aligned}
\end{equation*}
with $ H\cong t_1\circ r_2 \cong r_2\circ t_1: \mathcal A \longrightarrow \mathcal T_1\cap \mathcal F_2;$
and the inverse is $$\Psi:(\CoFib, \Fib, \Weq) \longmapsto (\mathcal T_1, \mathcal F_1; \mathcal T_2, \mathcal F_2)$$ where
\begin{equation*}
\begin{aligned}
&\mathcal T_1:=\{X\in \mathcal A \ \mid \ 0\longrightarrow X \ \mbox{is in} \ \CoFib\},\ \ \ \ \ \mathcal F_1:=\{X\in \mathcal A \ \mid \ X\longrightarrow 0 \ \mbox{is in} \ \TFib\}
\\
&\mathcal T_2:=\{X\in \mathcal A \ \mid \ 0\longrightarrow X \ \mbox{is in} \ \TCoFib\},\ \ \ \mathcal F_2:=\{X\in \mathcal A \ \mid \ X \longrightarrow 0 \ \mbox{is in} \ \Fib\}.
\end{aligned}
\end{equation*}

\vskip 5pt

$(2)$ \ Any torsion model structure on $\mathcal A$ is a bi-reflective model structure.

\vskip 5pt

$(3)$ \ Let $(\CoFib, \Fib, \Weq)$ be a torsion model structure given by twin torsion pair $(\mathcal T_1, \mathcal F_1; \mathcal T_2, \mathcal F_2)$.
Then the class of cofibrant objects is $\mathcal T_1$, the class of fibrant objects is $\mathcal F_2$, and the class of trivial objects is
$$\mathcal T_2*\mathcal F_1: = \{X\in \mathcal A \mid \exists\ \text{an exact sequence } 0\longrightarrow T_2 \longrightarrow X \longrightarrow F_1 \longrightarrow 0, \ T_2 \in \mathcal T_2, \ F_1 \in \mathcal F_1\}.$$
The homotopy category $\Ho(\mathcal A)$ is a quasi-abelian category, and it is equivalent to $\mathcal T_1\cap\mathcal F_2$ as a category.
In particular, every torsion class and every torsion-free class can be realized as the homotopy category of a torsion model structure.
\end{thm}

\begin{proof}\ $(1)$ \ By Proposition~\ref{prop:exact-nested-torsion-pairs-bi-reflective}, $\Phi$ is well-defined.

\vskip 5pt

Let $(\CoFib, \Fib, \Weq)$ be a torsion model structure. By Definition \ref{defn:torsion-OFS-and-torsion-model}(1), $(\CoFib,\TFib)$ and $(\TCoFib,\Fib)$ are torsion factorization systems.
By Proposition~\ref{prop:abelian-torsion-OFS}, one has two torsion pairs $(\mathcal T_1,\mathcal F_1,t_1,r_1)$ and $(\mathcal T_2,\mathcal F_2,t_2,r_2)$, where
\begin{equation*}
\begin{aligned}
&\mathcal T_1:=\{X\in \mathcal A \ \mid \ 0\longrightarrow X \ \mbox{is in} \ \CoFib\},\ \ \ \ \ \mathcal F_1:=\{X\in \mathcal A \ \mid \ X\longrightarrow 0 \ \mbox{is in} \ \TFib\}
\\
&\mathcal T_2:=\{X\in \mathcal A \ \mid \ 0\longrightarrow X \ \mbox{is in} \ \TCoFib\},\ \ \ \mathcal F_2:=\{X\in \mathcal A \ \mid \ X \longrightarrow 0 \ \mbox{is in} \ \Fib\}.
\end{aligned}
\end{equation*}
Since $\TCoFib\subseteq\CoFib$, one has $\mathcal T_2\subseteq\mathcal T_1$, and hence $(\mathcal T_1, \mathcal F_1, t_1, r_1; \mathcal T_2, \mathcal F_2, t_2, r_2)$ is a twin torsion pair.
Thus $\Psi$ is well-defined.

\vskip 5pt

Starting from the torsion factorization system $(\CoFib,\TFib)$,  by Proposition~\ref{prop:abelian-torsion-OFS}, one gets the corresponding torsion pair $(\mathcal T_1, \mathcal F_1, t_1, r_1)$;
and for the torsion pair $(\mathcal T_1, \mathcal F_1, t_1, r_1)$, again by Proposition~\ref{prop:abelian-torsion-OFS},
one gets the corresponding torsion factorization system $(\mathcal E_1, \mathcal M_1)$, where
$$\mathcal E_1 =\{f \mid  r_1(f) \text{ is an isomorphism}\}, \ \ \
\mathcal M_1 =\{f \mid  t_1(f) \text{ is an isomorphism}\}.$$ The one to one correspondence in Proposition~\ref{prop:abelian-torsion-OFS} implies
$(\CoFib, \TFib) = (\mathcal E_1, \mathcal M_1).$ Thus
$$\CoFib  =\{f\in {\rm Mor}(\mathcal A) \ \mid \ r_1(f) \text{ is an isomorphism}\ \}.$$
Starting from the torsion factorization system $(\TCoFib, \Fib)$, by the similar argument one gets
 $$\Fib  =\{f\in {\rm Mor}(\mathcal A) \ \mid \ t_2(f) \text{ is an isomorphism}\ \}.$$
Thus, by definition one has  $$(\Phi\circ\Psi)(\CoFib, \Fib, \Weq) = \Phi((\mathcal T_1, \mathcal F_1, t_1, r_1; \mathcal T_2, \mathcal F_2, t_2, r_2)) = (\CoFib, \ \Fib, \ \Weq').$$
By Fact \ref{elementpropmodel}(1) one has $(\Phi\circ\Psi)(\CoFib, \Fib, \Weq) = (\CoFib, \ \Fib, \ \Weq).$

\vskip 5pt

Conversely, let $(\mathcal T_1, \mathcal F_1, t_1, r_1; \mathcal T_2, \mathcal F_2, t_2, r_2)$ be a twin torsion pair in $\mathcal A$.
By Proposition~\ref{prop:exact-nested-torsion-pairs-bi-reflective}, $$\Phi((\mathcal T_1, \mathcal F_1, t_1, r_1; \mathcal T_2, \mathcal F_2, t_2, r_2)) = (\CoFib, \Fib, \Weq)$$ is a torsion model structure on $\mathcal A$, where
\begin{equation*}
\begin{aligned}
\CoFib & =\{f\in {\rm Mor}(\mathcal A) \ \mid \ r_1(f) \text{ is an isomorphism}\ \}
\\
\Fib &=\{f\in {\rm Mor}(\mathcal A) \ \mid \ t_2(f) \text{ is an isomorphism}\ \}
\\
\Weq &=\{f\in {\rm Mor}(\mathcal A) \ \mid \ H(f) \text{ is an isomorphism}\ \}
\end{aligned}
\end{equation*}
with $ H\cong t_1\circ r_2 \cong r_2\circ t_1: \mathcal A \longrightarrow \mathcal T_1\cap \mathcal F_2$.
Then by definition one has $$(\Psi\circ\Phi)((\mathcal T_1, \mathcal F_1, t_1, r_1; \mathcal T_2, \mathcal F_2, t_2, r_2))  = \Psi(\CoFib, \Fib, \Weq) = (\mathcal T_1', \mathcal F_1', t'_1, r'_1; \mathcal T'_2, \mathcal F'_2, t'_2, r'_2)$$
where
\begin{align*}\mathcal T'_1 & =\{X\in \mathcal A \ \mid \ 0\longrightarrow X \ \mbox{is in} \ \CoFib\}
= \{X\in \mathcal A \ \mid \ 0\longrightarrow r_1(X) \ \mbox{is an isomorphism}\}\\ & = \{X\in \mathcal A \ \mid \ r_1(X)\cong 0\} = \mathcal T_1\\
\mathcal F'_2 & =\{X\in \mathcal A \ \mid \ X\longrightarrow 0 \ \mbox{is in} \ \Fib\}
= \{X\in \mathcal A \ \mid \ t_2(X)\longrightarrow 0 \ \mbox{is an isomorphism}\}\\ & = \{X\in \mathcal A \ \mid \ t_2(X)\cong 0\} = \mathcal F_2.
\end{align*}
Therefore
$$(\Psi\circ\Phi)((\mathcal T_1, \mathcal F_1, t_1, r_1; \mathcal T_2, \mathcal F_2, t_2, r_2)) =
(\mathcal T_1', \mathcal F_1', t'_1, r'_1; \mathcal T'_2, \mathcal F'_2, t'_2, r'_2) = (\mathcal T_1, \mathcal F_1, t_1, r_1; \mathcal T_2, \mathcal F_2, t_2, r_2)$$
where the last equality follows from the fact that a torsion pair is uniquely determined by one of its parts. This completes the proof.

\vskip5pt
$(2)$ \ Let $(\CoFib, \Fib, \Weq)$ be a torsion model structure. By (1) it is given by a twin torsion pair $(\mathcal T_1, \mathcal F_1, t_1, r_1; \mathcal T_2, \mathcal F_2, t_2, r_2)$.
By Proposition~\ref{prop:abelian-torsion-OFS}, the two torsion pairs in this twin torsion pair give torsion factorization systems $(\mathcal E_1,\mathcal M_1)$ and $(\mathcal E_2, \mathcal M_2)$,  where
\begin{equation*}
\begin{aligned}
\mathcal E_1 & =\{f\in {\rm Mor}(\mathcal A) \ \mid \ r_1(f) \text{ is an isomorphism}\ \}
\\
\mathcal M_1 &=\{f\in {\rm Mor}(\mathcal A) \ \mid \ t_1(f) \text{ is an isomorphism}\ \}
\\
\mathcal E_2 & =\{f\in {\rm Mor}(\mathcal A) \ \mid \ r_2(f) \text{ is an isomorphism}\ \}
\\
\mathcal M_2 &=\{f\in {\rm Mor}(\mathcal A) \ \mid \ t_2(f) \text{ is an isomorphism}\  \}.
\end{aligned}
\end{equation*}
However, the torsion model structure $(\CoFib, \Fib, \Weq)$ is given by the twin torsion pair \newline $(\mathcal T_1, \mathcal F_1, t_1, r_1; \mathcal T_2, \mathcal F_2, t_2, r_2)$, which means
$\CoFib = \mathcal E_1,   \ \Fib = \mathcal M_2.$
Since $(\mathcal E_1,\mathcal M_1)$ and $(\CoFib,\TFib)$ are torsion factorization systems with the same left class, one has $(\mathcal E_1,\mathcal M_1) = (\CoFib,\TFib).$
Similarly, $(\mathcal E_2,\mathcal M_2) = (\TCoFib,\Fib).$
Hence $$\mathcal M_1\circ\mathcal E_2 = \TFib\circ \TCoFib = \Weq = \{f\in {\rm Mor}(\mathcal A) \ \mid \ H(f) \text{ is an isomorphism}\ \}$$ with $ H\cong t_1\circ r_2 \cong r_2\circ t_1: \mathcal A \longrightarrow \mathcal T_1\cap \mathcal F_2$.

\vskip 5pt

We claim that $(\mathcal T_1,\mathcal F_2)$ is a bi-reflective pair.
One needs to show (A1), (A2) and (A3) in Definition~\ref{defn:bi-reflectivepair}.

\vskip 5pt
\noindent
\textbf{(A1).} \ By Proposition~\ref{prop:abelian-torsion-OFS}(1), $\mathcal T_1$ is the coreflective subcategory of $\mathcal A$ with semi-right-exact coreflector $t_1$, and $\mathcal F_2$ is the reflective subcategory of $\mathcal A$ with semi-left-exact reflector $r_2$.
The associated orthogonal factorization systems are $(\mathcal E_{t_1},\mathcal M_{t_1})$ and $(\mathcal E_{r_2},\mathcal M_{r_2})$, where
\begin{equation*}
\begin{aligned}
\mathcal E_{t_1} &= \left\{
f:X\longrightarrow Y\ \middle|\begin{array}{c}
\xymatrix@C=0.9cm@R=0.4cm{
{t_1}X \ar[r]^-{\varepsilon_{_X}} \ar[d]_{{t_1}(f)} & X \ar[d]^{f} \\
{t_1}Y \ar[r]^-{\varepsilon_{_Y}} & Y
}
\end{array}\text{ is a pushout}
\right\}
\\
\mathcal M_{t_1}
&= \{\ f \in {\rm Mor}(\mathcal A)\ \mid \  {t_1}(f)\text{ is an isomorphism}\}
\end{aligned}
\end{equation*}
 \begin{equation*}
 \begin{aligned}
\mathcal E_{r_2} &= \{f \in {\rm Mor}(\mathcal A)\ \mid \  {r_2}(f)\text{ is an isomorphism}\}
\\
\mathcal M_{r_2}
&=
\left\{
f: X\longrightarrow Y\ \middle|\begin{array}{c}
\xymatrix@C=0.9cm@R=0.4cm{
X\ar[r]^-{\eta_{_X}}\ar[d]_f & {r_2}X\ar[d]^{{r_2}(f)}\\
Y\ar[r]^-{\eta_{_Y}} & {r_2}Y
}
\end{array}\text{is a pullback}
\right\}.
\end{aligned}
\end{equation*}
Note that by construction $\mathcal M_{t_1}=\mathcal M_1$ and $\mathcal E_{r_2}=\mathcal E_2$.
Since $(\mathcal E_1,\mathcal M_1)$ and $(\mathcal E_{t_1},\mathcal M_{t_1})$ are torsion factorization systems with the same right class, one has $(\mathcal E_{t_1},\mathcal M_{t_1}) = (\mathcal E_1,\mathcal M_1)$.
Similarly, $(\mathcal E_{r_2},\mathcal M_{r_2}) = (\mathcal E_2,\mathcal M_2)$.
\vskip 5pt
\noindent
\textbf{(A2).} \ Since $\mathcal E_{r_2} = \mathcal E_2 = \TCoFib $ and $\mathcal M_{t_1} = \mathcal M_1 = \TFib$, it follows that $\mathcal E_{r_2} \perp \mathcal M_{t_1}.$
\vskip 5pt
\noindent
\textbf{(A3).} \
Since $\mathcal E_{r_2} = \mathcal E_2\subseteq \Weq$ and $\mathcal M_{t_1} = \mathcal M_1\subseteq \Weq$, one has $$\mathcal E_{r_2}\circ\mathcal M_{t_1} = \mathcal E_2\circ\mathcal M_1\subseteq \Weq = \mathcal M_1\circ \mathcal E_2 = \mathcal M_{t_1}\circ\mathcal E_{r_2}.$$

\vskip 5pt

By Theorem~\ref{thm:bi-reflective-pairs-bi-reflective-orthogonal-model-structures}, the bi-reflective model structure given by $(\mathcal T_1,\mathcal F_2)$ is $(\mathcal E_{t_1}, \mathcal M_{r_2}, \mathcal M_{t_1}\circ \mathcal E_{r_2})$.
Thus $$(\CoFib, \Fib, \Weq) = (\mathcal E_1, \mathcal M_2, \mathcal M_1\circ \mathcal E_2) = (\mathcal E_{t_1}, \mathcal M_{r_2}, \mathcal M_{t_1}\circ \mathcal E_{r_2})$$ is a bi-reflective model structure.

\vskip5pt
$(3)$ \ By (1), $\mathcal T_1$ is precisely the class of cofibrant objects of $(\CoFib, \Fib, \Weq)$ and $\mathcal F_2$ is precisely the class of fibrant objects of $(\CoFib, \Fib, \Weq)$.

\vskip 5pt

Now we prove that the class of trivial objects is precisely
$$\mathcal T_2*\mathcal F_1 = \{X\in \mathcal A \mid \exists\ \text{an exact sequence } 0\longrightarrow T_2 \longrightarrow X \longrightarrow F_1 \longrightarrow 0, \ T_2 \in \mathcal T_2, \ F_1 \in \mathcal F_1\}.$$
Let $W$ be a trivial object. By definition, $0\longrightarrow W$ is a weak equivalence, i.e., $H(W)\cong (t_1\circ r_2)W\cong (r_2\circ t_1)W \cong 0$.
Let $$0\longrightarrow t_1(W) \longrightarrow W \longrightarrow r_1(W) \longrightarrow 0$$ be the exact sequence with $t_1(W)\in\mathcal T_1$ and $r_1(W)\in\mathcal F_1$.
One has $$0\longrightarrow t_2(t_1(W)) \longrightarrow t_1(W) \longrightarrow r_2(t_1(W)) \longrightarrow 0$$ is the exact sequence with $t_2(t_1(W))\in\mathcal T_2$ and $r_2(t_1(W))\in\mathcal F_2$.
Since $(r_2\circ t_1)W \cong 0$, one has $t_1(W)\cong t_2(t_1(W))\in\mathcal T_2$. Thus $W\in \mathcal T_2*\mathcal F_1$. Conversely, let $X\in \mathcal T_2*\mathcal F_1$. Then there is an exact sequence $$0\longrightarrow T_2 \longrightarrow X \longrightarrow F_1 \longrightarrow 0$$ with $T_2 \in \mathcal T_2, \ F_1 \in \mathcal F_1$.
Since $r_2$ is the left adjoint of the inclusion $\mathcal F_2\hookrightarrow \mathcal A$, it is a right exact functor.
Applying $r_2$ to the above exact sequence, one gets $$r_2(T_2) = 0 \longrightarrow r_2(X) \longrightarrow r_2(F_1) \longrightarrow 0.$$
Since $\mathcal F_1 \subseteq \mathcal F_2$, one has $r_2(F_1)\cong F_1$.
It follows that $r_2(X)\cong r_2(F_1)\cong F_1$.
Thus $t_1(r_2(X))\cong t_1(F_1)\cong 0$. Thus $0\longrightarrow X$ is a weak equivalence, i.e., $X$ is a trivial object.

\vskip 5pt

By Theorem~\ref{thm:orthogonal-ho-CF}, the homotopy category of $(\CoFib, \Fib, \Weq)$  is equivalent to $\mathcal T_1\cap\mathcal F_2$.
By \cite[Theorem 3.2]{T}, $\mathcal T_1\cap\mathcal F_2$ is a quasi-abelian category.

\vskip 5pt

Moreover, let $(\mathcal T,\mathcal F)$ be a torsion pair in an abelian category $\mathcal A$. Consider the torsion model structure $(\CoFib, \Fib, \Weq)$ given by the twin torsion pair $(\mathcal T, \mathcal F; 0, \mathcal A)$.
By the preceding paragraph,  the homotopy category of $(\CoFib, \Fib, \Weq)$ is equivalent to $\mathcal T\cap\mathcal A = \mathcal T$.

\vskip 5pt

Similarly,  the homotopy category of the torsion model structure given by the twin torsion pair $(\mathcal A, 0; \mathcal T, \mathcal F)$ is equivalent to $\mathcal F$.
\end{proof}

\begin{rem}\label{both-TTF-and-torsion} \ $(1)$ \ On an abelian category $\mathcal{A}$, a TTF model structure is not torsion, in general. In fact, a TTF model structure induced by TTF triple $(\mathcal C, \mathcal W, \mathcal F)$
is a torsion model structure if and only if $\mathcal W = 0$.

\vskip5pt
$(2)$ \ Proposition~\ref{prop:exact-nested-torsion-pairs-bi-reflective} can be generalized to a weakly idempotent complete exact category $\mathcal A$.
However, Theorem~\ref{thm:abelian-bijection-models-torsion-pairs} cannot be generalized to $\mathcal A$, in general.
\end{rem}

\subsection{\bf Torsion model structures on posets}

A poset $(P,\leq)$ is a category: objects are elements of $P$, and there is a unique morphism $p\longrightarrow q$ if and only if $p\leq q$.
This category has no initial object or terminal object, in general. We will construct a torsion model structure which is also a bi-reflective model structure on $P$.

\vskip 5pt

Let  $(P,\leq)$ be a poset and $p\in P$. We fix the following notations. Set
$$L_p:=\{a\longrightarrow b\mid b\leq p\}\cup\{\operatorname{Id}_x\mid x\in P\}, \ \ \ \  \ R_p:=\{a\longrightarrow b\mid p\leq a\}\cup\{\operatorname{Id}_x\mid x\in P\}.$$
$$P_{\leq p}:=\{x\in P\mid x\leq p\}, \ \ \ \ P_{\geq p}:=\{x\in P\mid p\leq x\}.$$
For $m,n\in P$ with $m\le  n$, set $P_{[m,n]} :=\{x\in P\mid m\leq x\leq n\}.$

\begin{defn}\ Let $(P,\leq)$ be a poset.
An element $m\in P$ is a \emph{a cut} if, for every $x\in P$, exactly one of $m<x$, $m=x$, and $x< m$ holds.
\end{defn}

\begin{lem}\label{prop:cut-point-torsion-ofs}\ Let $(P,\leq)$ be a poset, and $m$ a cut of $P$. Then $(L_m,R_m)$ is a torsion factorization system in $P$.
\end{lem}

\begin{proof} \ Let $f:x \longrightarrow y$ be a morphism in $P$.
Since both $L_m$ and $R_m$ contain all the identity morphisms,
it is clear that every morphism can be factorized into $f=h\circ g$ with $g\in L_m$ and $h \in R_m$.
For example, if $m<x$, then $f = f\circ {\rm Id}_X$.

\vskip5pt

Given a commutative square
\[
\xymatrix@C=1.2cm@R=0.9cm{
a\ar[r]\ar[d]_l & c\ar[d]^r\\
b\ar[r] & d
}
\]
with $l\in L_m$ and $r\in R_m$, if $l$ and $r$ are not isomorphisms, one has
$a \leq b \leq m \leq c \leq d$.
Then $b\longrightarrow c$ is the unique lifting.
If one of $l$ and $r$ is isomorphism, the lifting is clear.
Hence $L_m\perp R_m$. By Proposition \ref{prop:OFS-criterion}, $(L_m,R_m)$ is an orthogonal factorization system.

\vskip5pt

It is clear that $L_m$ and $R_m$ have the two-out-of-three property. By definition $(L_m,R_m)$ is a torsion factorization system.
\end{proof}

\begin{prop}\label{prop:two-cut-points-poset-model} \ Suppose that $(P,\leq)$ is a poset, $m$ and $n$ are cuts of $P$ with $m \le  n$.
Then $(L_n,R_m,L_m\cup R_n)$ is a torsion model structure and also a bi-reflective model structure, the class $\mathcal C$ of cofibrant objects is $P_{\leq n}$, the class $\mathcal F$ of fibrant objects is $P_{\geq m}$,
and the homotopy category is $P_{[m,n]}$.
\end{prop}
\begin{proof} \ By Lemma~\ref{prop:cut-point-torsion-ofs}, one has torsion factorization systems $(L_m, R_m)$ and $(L_n, R_n)$.
Since $m < n$, one has $L_m \subseteq L_n$.

\vskip5pt

Put $\Weq: = R_n \circ L_m$. By definition one can deduce $\Weq = L_m \cup R_n$. It is clear that $ L_m \cup R_n$ is closed under compositions.
Put $(\CoFib, \Fib, \Weq):=(L_n, R_m, L_m \cup R_n)$. By Proposition~\ref{prop:weq-torsion}, $(\CoFib, \Fib, \Weq) =(L_n, R_m, L_m \cup R_n)$ is a torsion model structure, with
$$\TCoFib = L_n \cap (L_m \cup R_n) = L_m, \ \ \ \TFib =  R_m \cap (L_m \cup R_n) = R_n.$$

\vskip 5pt

To see it is a bi-reflective model structure, one needs to verify conditions {\rm (M1)}, {\rm (M2)}, {\rm (M3)}, and ${\rm (M3')}$ in Definition~\ref{defn:bi-reflective-orthogonal-model-structure}.

\vskip 5pt
{\rm (M1)} \ We claim that the class $\mathcal C$ of cofibrant objects is $P_{\leq n}$. Since the category $P$ has no initial object and terminal object, in general, we need to use Definition \ref{defn:cofibrant-fibrant-objects}.
Let $c \in P_{\leq n}$, i.e., $c \leq n$.
Let $\alpha: x \longrightarrow y $ be a morphism in $ R_n$ and $\beta: c \longrightarrow y$ a morphism.
If $\alpha = {\rm Id}_x$, then ${\rm Id}_x \circ \beta = \beta$.
If $\alpha\ne {\rm Id}_x$, then  $c \leq n \leq x \leq y$. Thus there is a unique $\theta: c \longrightarrow x$ with $\alpha \circ \theta = \beta$.
Thus $c$ is a cofibrant object.
Conversely, let $c \in P$ be a cofibrant object. We need to show $c \leq n$. Otherwise $n < c$, since $n$ is a cut. For the trivial fibration $n \longrightarrow c$ and the identity ${\rm Id}_c$, there is no morphism from $c$ to $n$,
which contradicts the assumption that $c$ is a cofibrant object. This proves the claim.

\vskip 5pt

For any object $x \in P$, one needs to construct a trivial fibration $\varepsilon: c \longrightarrow x$ with $c \in \mathcal C = P_{\leq n}$.
If $x \leq n$, one can take $\varepsilon = \operatorname{Id}_x$.
If $n < x$, one takes $\varepsilon: n \longrightarrow x$.
This proves that $P$ has enough cofibrant objects.

\vskip 5pt

Dually, the class $\mathcal F$ of fibrant objects is $P_{\geq m}$, and $P$ has enough fibrant objects.

\vskip 5pt

{\rm (M2)} \ It is clear that both $\TCoFib = L_m$ and $\TFib  = R_n$ have the two-out-of-three property.

\vskip 5pt

{\rm (M3)} \ Let   $\alpha: c \longrightarrow d$ be a morphism in $\TFib=R_n$ with $c \in \mathcal C = P_{\leq n}$, and $\beta: c \longrightarrow e$ a morphism in $\CoFib=L_n$.

If one of $\alpha$ and $\beta$ is an identity, then one has the corresponding pushout square.

The remaining case does not occur: namely, if $\alpha$ and $\beta$ are not identities, then  $c \leq e \leq n \leq c \leq d$. Thus $c = e = n$ and hence $\beta = {\rm Id}_n$, a contradiction.

\vskip 5pt
${\rm (M3')}$ \ Dual to {\rm (M3)}.

\vskip 5pt

By Theorem~\ref{thm:orthogonal-ho-CF}, one has $\Ho(P) \simeq  P_{\leq n} \cap P_{\geq m} = \{x\in P\mid m\leq x\leq n\}=P_{[m,n]}.$
\end{proof}

\subsection{\bf Torsion model structures on $G\text{-}\mathsf{Set}$} \ We will construct a torsion model structure on the category $G\text{-}\mathsf{Set}$. In general, this model structure is not bi-reflective.
The category $G\text{-}\mathsf{Set}$ has the speciality that it has the initial object $\varnothing$ and the terminal object $*$, where $*$ is the $G$-set with a unique element.
Thus,  $G\text{-}\mathsf{Set}$ has no zero object.

\vskip 5pt

Let $G$ be a group, $N$ a normal subgroup of $G$.
For a $G$-set $X$, denote by $X^N$ the set of fixed points of $N$, i.e., $X^N:=\{x\in X\mid nx=x, \ \forall \ n\in N\}.$
It is clear that  $X^N$ is a $G$-set.

\vskip 5pt

Denote by $\mathcal C_N$ the full subcategory of $G\text{-}\mathsf{Set}$ of $G$-sets $X$ with $X=X^N$.
Then one has a functor $(-)^N:G\text{-}\mathsf{Set}\longrightarrow \mathcal C_N.$

\begin{lem}\label{lem:gset-fixed-point-torsion-ofs}\ Let $N$ be a normal subgroup of group $G$.
Then $\mathcal C_N$ is a coreflective subcategory of $G\text{-}\mathsf{Set}$ such that the coreflector $(-)^N$ is semi-right-exact, and the corresponding orthogonal factorization system $(\mathcal E_N,\mathcal M_N)$ is a torsion factorization system, where
\begin{equation*}
\begin{aligned}
\mathcal{E}_{N} &= \{f:X\longrightarrow Y \ \mbox{in} \  G\text{-}\mathsf{Set} \ \mid  \ f \ \mbox{induces a bijection} \ f|_{X\setminus X^N}:X\setminus X^N
\overset{\cong}{\longrightarrow}Y\setminus Y^N \}\\
\mathcal M_N&= \{f:X\longrightarrow Y \ \mbox{in} \  G\text{-}\mathsf{Set} \ \mid \ f \ \mbox{induces a bijection} \ f^N:X^N\longrightarrow Y^N\}.
\end{aligned}
\end{equation*}
\end{lem}

\begin{proof} \ Let $i:\mathcal C_N\hookrightarrow G\text{-}\mathsf{Set}$ be the inclusion functor.
For every $Z\in\mathcal C_N$ and $X\in G\text{-}\mathsf{Set}$, one has a natural bijection $\operatorname{Hom}_{G\text{-}\mathsf{Set}}(iZ,X) \cong \operatorname{Hom}_{\mathcal C_N}(Z,X^N),\ f\longmapsto f^N.$
Hence  $(i, (-)^N)$ is an adjoint pair, and the counit is given by the inclusion $\varepsilon_{_X}:X^N\longrightarrow X.$
Thus  $\mathcal C_N$ is a coreflective subcategory of $G\text{-}\mathsf{Set}$ with coreflector $(-)^N$.

\vskip 5pt

We will show that the coreflector $(-)^N$ is semi-right-exact.
For every $X\in G\text{-}\mathsf{Set}$, by Fact~\ref{fact:approximation-basic}$(1')$, a right strong-$\mathcal C_N$-approximation of $X$ is of the form $\varepsilon_{_X}:X^N\longrightarrow X$.
Let $h:X^N\longrightarrow Z$ be a morphism with $Z\in\mathcal C_N$. Put $P = Z\overset \cdot \cup (X\setminus X^N).$ Then $P\in G\text{-}\mathsf{Set}$ with the obvious action. It is well-known that the square
\[
\xymatrix@C=1.2cm@R=0.5cm{
X^N\ar[d]_-{\varepsilon_{_X}}\ar[r]^-{h} & Z\ar[d]^-{j_Z} \\
X\ar[r] & P
}
\]
is  a pushout square, where $j_Z$ is the inclusion.  Since $Z=Z^N$ and $(X\setminus X^N)^N=\varnothing$, one has $P^N=Z.$
By Fact~\ref{fact:approximation-basic}$(1'){\rm (i')}$ the induced map $j_Z: Z\longrightarrow P$ is a strong right $\mathcal C_N$-approximation. Thus the coreflector $(-)^N$ is semi-right-exact.

\vskip 5pt

By Lemma~\ref{lem:sle-reflective-ofs}, one obtains an orthogonal factorization system $(\mathcal E_N,\mathcal M_N)$, where
\begin{equation*}
\begin{aligned}
\mathcal{E}_{N} &= \left\{
f:X\longrightarrow Y\ \middle|\begin{array}{c}
\xymatrix@C=0.9cm@R=0.4cm{
X^N \ar[r]^-{\varepsilon_X} \ar[d]_-{f^N} & X  \ar[d]^-f \\
Y^N \ar[r]^-{\varepsilon_Y} & Y
}
\end{array}\text{is a pushout}
\right\}\\
\mathcal M_N&= \{f:X\longrightarrow Y\mid f^N:X^N\longrightarrow Y^N \text{ is a bijection}\}.
\end{aligned}
\end{equation*}

\vskip 5pt
\noindent
We claim
\[
\mathcal E_N=
\{f:X\longrightarrow Y \ \mbox{in} \  G\text{-}\mathsf{Set} \ \mid  \ f \ \mbox{induces a bijection} \ f|_{X\setminus X^N}:X\setminus X^N
\overset{\cong}{\longrightarrow}Y\setminus Y^N \}.
\]
For a $G$-map $f:X\longrightarrow Y$, it is clear that the square
\[
\xymatrix@C=1.1cm@R=0.5cm{
X^N\ar[r]^{\varepsilon_{_X}}\ar[d]_{f^N} & X\ar[d]^-g \\
Y^N\ar@{^(->}[r]^-{\sigma} & Y^N\overset\cdot \cup (X\setminus X^N)
}
\]
is a pushout square, where $g(x) = \begin{cases} f^N(x),  & x\in X^N \\ x, & x\notin X^N.\end{cases}$
\ \ If $f\in \mathcal E_N$, then one gets a bijection $\alpha: Y^N\overset\cdot \cup (X\setminus X^N) \longrightarrow Y$ such that $\alpha\circ \sigma= \varepsilon_Y, \ \alpha \circ g = f.$ This implies that
$f|_{X\setminus X^N}:X\setminus X^N
\longrightarrow Y\setminus Y^N$ is a bijection. Conversely, one can see that $Y = Y^N\overset\cdot \cup (X\setminus X^N)$ and $f\in \mathcal E_N$.

\vskip 5pt

To see that $(\mathcal E_N, \mathcal M_N)$ is a torsion factorization system, it remains to
show that $\mathcal E_N$ has the two-out-of-three property (since $(-)^N$ is a functor, it is clear that $\mathcal M_N$ has the two-out-of-three property).
Let $X\overset{f}{\longrightarrow}Y\overset{g}{\longrightarrow}Z$ be $G$-maps.

If $f, g\in\mathcal E_N$, then it is clear that $g\circ f\in\mathcal E_N$.

If $f, g\circ f\in\mathcal E_N$, then
$
g|_{Y\setminus Y^N} = (g\circ f)|_{X\setminus X^N} \circ \left(f|_{X\setminus X^N}\right)^{-1},
$
so $g\in\mathcal E_N$.

Assume that $g, g\circ f\in\mathcal E_N$.  If $x\in X\setminus X^N$, then $f(x)\in Y\setminus Y^N$ (otherwise, $f(x)\in Y^N$ and $g\circ f(x)\in Z^N$, which contradicts $g\circ f\in\mathcal E_N$),
and hence
$f|_{X\setminus X^N} = \left(g|_{Y\setminus Y^N}\right)^{-1} \circ (g\circ f)|_{X\setminus X^N}$
is a bijection, i.e.,  $f\in\mathcal E_N$.
\end{proof}

\begin{prop}\label{prop:gset-normal-subgroup-torsion-model}\ Let $N$ be the normal subgroup of group $G$.
Then $(\CoFib,\Fib,\Weq)$ is a torsion model structure on $G\text{-}\mathsf{Set}$, where
\[
\begin{aligned}
\CoFib&=\{
f:X\longrightarrow Y\ \mid\
f|_{X\setminus X^N}:X\setminus X^N
\overset{\cong}{\longrightarrow}Y\setminus Y^N
\}\\
\Fib&=\{f:X\to Y\mid f^G:X^G\overset{\cong}{\longrightarrow}Y^G\}\\
\Weq&=\{f:X\longrightarrow Y\ \mid\
f|_{B_N(X)}:B_N(X)\overset{\cong}{\longrightarrow}B_N(Y)\}
\end{aligned}
\]
with $B_N(X):=X^N\setminus X^G$.
The class of cofibrant objects is $\mathcal C_N$, the class of fibrant objects is $\{X \in G\text{-}\mathsf{Set} \ \mid \ |X^{G}|=1\}$,
and the homotopy category is\[
\operatorname{Ho}(G\text{-}\mathsf{Set})\simeq\{X\in(G/N)\text{-}\mathsf{Set}\mid |X^{G/N}|=1\}.
\]
\end{prop}

\begin{proof} \ Proposition~\ref{lem:gset-fixed-point-torsion-ofs} gives torsion factorization systems $(\mathcal E_N,\mathcal M_N)$ and $(\mathcal E_G,\mathcal M_G)$, where

\begin{equation*}
\begin{aligned}
\mathcal{E}_{N} &= \left\{f:X\longrightarrow Y\ \middle|\
f|_{X\setminus X^N}:X\setminus X^N
\overset{\cong}{\longrightarrow}Y\setminus Y^N
\right\}\\
\mathcal M_N&= \{f:X\longrightarrow Y\mid f^N:X^N\longrightarrow Y^N \text{ is a bijection}\}\\
\mathcal{E}_{G} &= \left\{f:X\longrightarrow Y\ \middle|\
f|_{X\setminus X^G}:X\setminus X^G
\overset{\cong}{\longrightarrow}Y\setminus Y^G
\right\}\\
\mathcal M_G&= \{f:X\longrightarrow Y\mid f^G:X^G\longrightarrow Y^G \text{ is a bijection}\}.
\end{aligned}
\end{equation*}

\vskip5pt

Since $X^G\subseteq X^N$, one has $\mathcal E_G\subseteq\mathcal E_N$.
Set $B_N(X):=X^N\setminus X^G$.
We claim that
\[
\mathcal M_N\circ\mathcal E_G
=
\left\{
f:X\longrightarrow Y\ \middle|\
f|_{B_N(X)}:B_N(X)\overset{\cong}{\longrightarrow}B_N(Y)
\right\}.
\]

\vskip5pt

First let $f=m\circ e:X\longrightarrow Y$ be a $G$-map with $e:X\longrightarrow P$ in $\mathcal E_G$ and $m:P\longrightarrow Y$ in $\mathcal M_N$.
Since $e\in\mathcal E_G$, one has a bijection $X\setminus X^G\cong P\setminus P^G$. Therefore $P = P^G\overset\cdot\cup(P\setminus P^G) \cong P^G\overset\cdot\cup(X\setminus X^G)$,
and hence $P^N = (P^G)^N\overset\cdot\cup(X\setminus X^G)^N \cong P^G\overset\cdot\cup(X^N\setminus X^G)$,
it follows that $B_N(X) = X^N\setminus X^G\cong P^N\setminus P^G = B_N(P)$.
This induces a bijection $B_N(e):B_N(X) \longrightarrow B_N(P)$.
Since $m\in \mathcal M_N$, one has $m^N:P^N\longrightarrow Y^N$ is a bijection, it follows that $m^G: P^G\longrightarrow Y^G$ is a bijection.
This induces a bijection $B_N(m):B_N(P) = P^N\setminus P^G\longrightarrow Y^N\setminus Y^G = B_N(Y)$.
Therefore $f = m\circ e$ induces a bijection $B_N(f) = B_N(m)\circ B_N(e): B_N(X)\longrightarrow B_N(Y)$.

\vskip 5pt

Conversely, suppose that $f|_{B_N(X)}:B_N(X)\longrightarrow B_N(Y)$ is a
bijection.
Since $(\mathcal E_G,\mathcal M_G)$ is a torsion factorization system, one can factorize $f = m \circ e$ with $e:X\longrightarrow P$ in $\mathcal E_G$ and $m:P\longrightarrow Y$ in $\mathcal M_G$.
Since $e\in\mathcal E_G$, one has $X\setminus X^G \cong P\setminus P^G$.
Since $m\in \mathcal M_G$, one has $P^G\cong Y^G$.
Thus $P = P^G\overset\cdot\cup (P\setminus P^G) \cong Y^G \overset\cdot\cup (X\setminus X^G)$,
it follows that $P^N\cong (Y^G)^N \overset\cdot\cup (X\setminus X^G)^N = Y^G \overset\cdot\cup (X^N\setminus X^G) = Y^G \overset\cdot\cup B_N(X)$.
Since $Y^N = Y^G\overset\cdot\cup B_N(Y)$ and $B_N(X)\longrightarrow B_N(Y)$ is a bijection, one has $m^N:P^N \cong Y^G \overset\cdot\cup B_N(X) \longrightarrow Y^N = Y^G \overset\cdot\cup B_N(Y)$ is a bijection.
Thus $m\in\mathcal M_N$, and consequently $f=m\circ e\in\mathcal M_N\circ\mathcal E_G$.
This proves the claim.

\vskip 5pt

By the claim, $\mathcal M_N\circ\mathcal E_G$ is closed under compositions.
By Proposition~\ref{prop:weq-torsion}, one gets the torsion model structure $(\mathcal E_N,\mathcal M_G,\mathcal M_N\circ\mathcal E_G)$.

\vskip 5pt

A $G$-set $X$ is cofibrant if and only if $\varnothing\longrightarrow X$ lies in $\mathcal E_N$, i.e.,  $X=X^N$. Thus the class of cofibrant objects is $\mathcal C_N$.

A $G$-set $X$ is fibrant if and only if $X\longrightarrow *$ lies in $\mathcal M_G$, i.e., $|X^G|=1$.
Thus the class of fibrant objects is $\mathcal F_G = \{X \in G\text{-}\mathsf{Set} \ \mid \ |X^{G}|=1\}$.
By Theorem~\ref{thm:orthogonal-ho-CF} one has
$\operatorname{Ho}(G\text{-}\mathsf{Set})\simeq \mathcal C_N \cap \mathcal F_G=
\{X\in(G/N)\text{-}\mathsf{Set}\mid |X^{G/N}|=1\}.$ \end{proof}

\begin{exm} \ The torsion model structure given in Proposition
\ref{prop:gset-normal-subgroup-torsion-model} is not bi-reflective,  in
general. Let $G=\mathbb Z_2$ and $N=\{e\}.$  In this case the torsion model structure $(\CoFib,\Fib,\Weq)$ on $\mathbb Z_2\text{-}\mathsf{Set}$ constructed in Proposition
\ref{prop:gset-normal-subgroup-torsion-model}   is
\[
\begin{aligned}
\CoFib
&=\operatorname{Mor}(\mathbb Z_2\text{-}\mathsf{Set}),\\
\Fib
&=\left\{f:X\longrightarrow Y\ \middle|\
f^{\mathbb Z_2}:X^{\mathbb Z_2}
\overset{\cong}{\longrightarrow}Y^{\mathbb Z_2}\right\},\\
\Weq
&=\left\{f:X\longrightarrow Y\ \middle|\
f|_{X\setminus X^{\mathbb Z_2}}:
X\setminus X^{\mathbb Z_2}
\overset{\cong}{\longrightarrow}
Y\setminus Y^{\mathbb Z_2}\right\}.
\end{aligned}
\]
The class of cofibrant objects is $\mathbb Z_2\text{-}\mathsf{Set}$, the class $\mathcal{F}$ of fibrant objects is the class of $\mathbb Z_2$-sets $X$ with $|X^{\mathbb Z_2}|=1$.
We claim that this torsion model structure is not bi-reflective.

\vskip 5pt

To see this, we need to define a functor $r: \mathbb Z_2\text{-}\mathsf{Set} \longrightarrow \mathcal{F}$ as follows. For any $\mathbb Z_2$-set $X$,  $r(X):=(X\setminus X^{\mathbb Z_2})\overset\cdot\cup *$.
For a $\mathbb Z_2$-map $f: X\longrightarrow Y$, $r(f): (X\setminus X^{\mathbb Z_2})\overset\cdot\cup * \longrightarrow (Y\setminus Y^{\mathbb Z_2})\overset\cdot\cup *$ is defined as
$$r(f) (x) = \begin{cases}
f(x), &f(x)\notin Y^{\mathbb Z_2}, \\
*,&f(x)\in Y^{\mathbb Z_2}
\end{cases}$$
here for convenience we denote  the unique element in
the terminal $\mathbb Z_2$-set $*$ still by $*$.
Then $(r, i)$ is an adjoint pair, the unit is given by
\[
\eta_{_X}:X\longrightarrow r(X),\qquad
x\longmapsto
\begin{cases}
x,&x\in X\setminus X^{\mathbb Z_2},\\
*,&x\in X^{\mathbb Z_2}.
\end{cases}
\]
Thus,  $\mathcal F$ is a reflective subcategory of $\mathbb Z_2\text{-}\mathsf{Set}$ with the reflector $r$.

\vskip5pt

Now, assume otherwise that $(\CoFib,\Fib,\Weq)$ is a bi-reflective model structure. Then by Theorem~\ref{thm:bi-reflective-pairs-bi-reflective-orthogonal-model-structures}, $r$ should be semi-left-exact.
However, $r$ is not semi-left-exact, as explained below.

\vskip5pt

Let $X=\{x_1,x_2\}$ with trivial $\mathbb Z_2$-action, and $h:E\longrightarrow r(X)=*$ be the unique $G$-map where $E=*\overset\cdot\cup \mathbb Z_2$.
Then $E\in\mathcal F$ and any left $\mathcal F$-strong-approximation of $X$ is of the form $\eta_X:X\longrightarrow *$.
The pullback of $\eta_X:X\longrightarrow *$ along $h$ is
\[
\xymatrix@C=0.9cm@R=0.6cm{
P\ar[r]\ar[d]_{p} & X\ar[d]^{\eta_X} \\
E\ar[r]^-h & {*}
}
\]
where  $p:P=E\times X\longrightarrow E$ is the projection.
Since $X$ has two elements, one has $P=E \overset\cdot\cup E= *\overset\cdot\cup*\overset\cdot\cup G\overset\cdot\cup G$ and hence $r(P)=*\overset\cdot\cup \mathbb Z_2\overset\cdot\cup \mathbb Z_2$.
Since $r(E)=*\overset\cdot\cup \mathbb Z_2$,
$p$ is not a left $\mathcal F$-strong-approximation of $P$ (otherwise, by Fact~\ref{fact:approximation-basic}(1)(ii) one must have $r(P)\cong E=*\overset\cdot\cup \mathbb Z_2$).
This shows that $r$ is not semi-left-exact.
\end{exm}

\subsection{\bf Torsion model structures on $\mathsf{TopGrp}$}\ We will construct a torsion model structure on the category $\mathsf{TopGrp}$ of topological groups,
and this model structure is not bi-reflective.
The category $\mathsf{TopGrp}$ has the zero object $\{1\}$.

\vskip 5pt

By \cite[Theorem 50]{BC},  $\mathsf{TopGrp}$ is a homological category. Thus,
in $\mathsf{TopGrp}$ the pullback of an epimorphism along any morphism is again an epimorphism, and one has the following lemma.

\begin{lem} \ \label{lem:topgrp-short-five}\ The short five lemma holds for $\mathsf{TopGrp}$, i.e.,
in a commutative diagram of morphisms in $\mathsf{TopGrp}$ with exact rows
\[
\xymatrix@C=0.8cm{
1\ar[r] & A\ar[r]\ar[d]_{\alpha}^{\cong} & B\ar[r]^{q}\ar[d]_{\beta} & C\ar[r]\ar[d]_{\gamma}^{\cong} & 1\\
1\ar[r] & A'\ar[r] & B'\ar[r]^{q'} & C'\ar[r] & 1
}
\]
if $\alpha$ and $\gamma$ are isomorphisms in $\mathsf{TopGrp}$, then so is $\beta$.
\end{lem}

For $G\in\mathsf{TopGrp}$, let $N(G):=\overline{\{1_G\}}$ be the closure of $1_G$, and  $\Gamma(G)$ the connected component of $1_G$.

\begin{fact}\label{lem:topgrp-N-Gamma}\ The subsets $N(G)$ and $\Gamma(G)$ are closed normal subgroups of $G$, and $N(G)\trianglelefteq\Gamma(G)$.
Every morphism $f:G\longrightarrow H$ restricts to morphisms $N(f):N(G)\longrightarrow N(H)$ and $\Gamma(f):\Gamma(G)\longrightarrow\Gamma(H)$.
Thus $N$ and $\Gamma$ define endofunctors of $\mathsf{TopGrp}$.
\end{fact}

\begin{proof}
Since $d:G\times G\longrightarrow G,(x,y)\mapsto xy^{-1}$ is continuous, one sees that $N(G)$ is a subgroup of $G$.
For $g\in G$, the map $c_g:G\longrightarrow G$, $x \mapsto gxg^{-1}$, is a homeomorphism and satisfies $c_g(1_G)=1_G$.
Thus $N(G)\trianglelefteq G$.

\vskip 5pt

By definition $\Gamma(G)$ is closed. Since the multiplication and the inverse are continuous maps, one has $\Gamma(G)\Gamma(G)^{-1}=d(\Gamma(G)\times\Gamma(G))\subseteq\Gamma(G)$, so $\Gamma(G)$ is a subgroup of $G$.
Since $c_g(\Gamma(G))$ is connected and contains $1_G$, one has $c_g(\Gamma(G))\subseteq\Gamma(G)$, which implies that $\Gamma(G)\trianglelefteq G$.

\vskip 5pt

Since $\Gamma(G)$ is closed, one has $N(G)\subseteq\Gamma(G)$. For every continuous homomorphism $f:G\longrightarrow H$, one has $f(N(G)) = f\bigl(\overline{\{1_G\}}\bigr)\subseteq\overline{\{1_H\}}=N(H)$ and $f(\Gamma(G))\subseteq\Gamma(H)$.
Hence the restrictions $N(f)$ and $\Gamma(f)$ are well-defined and functorial.
\end{proof}

Put $K(G):=G/N(G)$, $\pi_0(G):=G/\Gamma(G)$, and $H_0(G):=\Gamma(G)/N(G)$.
Set $$\mathcal F_N:=\{G\in\mathsf{TopGrp}\mid N(G)= \{1\}\} \ \ \ \mbox{and} \ \  \mathcal F_\Gamma:=\{G\in\mathsf{TopGrp}\mid\Gamma(G)= \{1\}\}.$$

\begin{lem}\label{lem:topgrp}\
${\rm (1)}$\  The full subcategory $\mathcal F_N$ of $\mathsf{TopGrp}$ is a reflective subcategory such that the reflector $K$ is semi-left-exact, and the corresponding orthogonal factorization system $(\mathcal E_N,\mathcal M_N)$ is a torsion factorization system, where
\[
\mathcal{E}_{N} = \{f\mid K(f) \text{ is an isomorphism}\},\ \ \ \
\mathcal M_N= \{f\mid N(f) \text{ is an isomorphism}\}.
\]

\vskip 5pt

${\rm (2)}$\  The full subcategory $\mathcal F_\Gamma$ of $\mathsf{TopGrp}$ is a reflective subcategory such that the reflector $\pi_0$ is semi-left-exact, and the corresponding orthogonal factorization system $(\mathcal E_\Gamma,\mathcal M_\Gamma)$ is a torsion factorization system, where
\[
\mathcal{E}_{\Gamma} = \{f\mid \pi_0(f) \text{ is an isomorphism}\},\ \ \ \
\mathcal{M}_{\Gamma}  = \{f\mid \Gamma(f) \text{ is an isomorphism}\}.
\]
\end{lem}

\begin{proof} \ (1) \ Denote by  $\eta_G: G\longrightarrow K(G)$ the quotient map. The functor $K: \mathsf{TopGrp} \longrightarrow \mathcal F_N$ given by $K(G):=G/N(G)$ is the left adjoint of $i: \mathcal F_N \hookrightarrow  \mathsf{TopGrp}$,
and the unit is given by $\eta_G$. In fact, if $H\in\mathcal F_N$ and $f:G\longrightarrow H$ is a continuous homomorphism, then $f(N(G))\subseteq N(H)= \{1\}$.
By the quotient topology,  this induces a continuous map $\tilde{f}: K(G)\longrightarrow H$ such that $\tilde{f}\circ \eta_G = f$.
Thus one has a natural bijection $\operatorname{Hom}_{\mathcal F_N}(K(G),H) \cong \operatorname{Hom}_{\mathsf{TopGrp}}(G,H).$
Therefore $\mathcal F_N$ is reflective with reflector $K$.

\vskip 5pt

Next, we show that $K$ is semi-left-exact.
It is well-known that in $\mathsf{TopGrp}$, the pullbacks and pushouts exist.
For every $G\in \mathsf{TopGrp}$, by Fact~\ref{fact:approximation-basic}$(1)$, a left strong-$\mathcal F_N$-approximation of $G$ is of the form $\eta_{_G}:G\longrightarrow K(G)$.
Let $h:E\longrightarrow K(G)$ be a morphism with $E\in\mathcal{F}_N$.
Let $P=\{(g,e)\in G\times E\mid \eta_G(g)=h(e)\},$ equipped with the subspace topology, $p:P\longrightarrow E,\  (g,e)\longmapsto e,$ and $q:P\longrightarrow G,\ (g,e)\longmapsto g$ be two projections.
Then $P$ is the pullback of $\eta_G:G\longrightarrow K(G)$ along $h$.
Since $P$ has the subspace topology, then one has \[N(P)=P\cap\bigl(N(G)\times N(E)\bigr) =P\cap\bigl(N(G)\times\{1_E\}\bigr) =N(G)\times\{1_E\} = \Ker p.\]
Hence $p$ induces a continuous bijective homomorphism $\overline p:P/N(P)\longrightarrow E,\ [(g,e)]\longmapsto e.$
Since for open subsets $U_G\subseteq G$ and $U_E\subseteq E$, $p\bigl(P\cap(U_G\times U_E)\bigr) = U_E\cap h^{-1}\bigl(\eta_G(U_G)\bigr)$ is an open set ($\eta_G$ is open by quotient topology), it follows that $p$ is open.
Thus $\overline p$ is an open continuous bijection, and therefore $P/N(P)\cong E$ in $\mathsf{TopGrp}$.

\vskip 5pt

Thus one gets the following commutative diagram with exact rows:
\[
\xymatrix@C=1.0cm@R=0.7cm{
1\ar[r] & N(P)\ar[r]\ar@{=}[d] & P\ar[r]^{p}\ar[d]^{q} & E\ar[d]^{h}\ar[r] & 1\\
1\ar[r] & N(G)\ar[r] & G\ar[r]^-{\eta_{_G}} & K(G)\ar[r] & 1
}
\]
where the right hand side square is pullback.
Therefore $p:P\longrightarrow E\cong K(P)$ is a left $\mathcal F$-strong-approximation of $P$, by Fact~\ref{fact:approximation-basic}$(1)$.
By definition, $K$ is semi-left-exact.

\vskip 5pt

By Lemma~\ref{lem:sle-reflective-ofs}, one gets an orthogonal factorization system $(\mathcal{E}_N,\mathcal M_N)$ in $\mathsf{TopGrp}$, where
\begin{equation*}
\begin{aligned}
\mathcal{E}_N &= \{f \in {\rm Mor}(\mathsf{TopGrp})\ \mid \  K(f)\text{ is an isomorphism}\}
\\
\mathcal M_N
&=
\left\{
f: X\longrightarrow Y\ \mbox{in} \mathsf{TopGrp}\ \middle|\begin{array}{c}
\xymatrix@C=0.9cm@R=0.4cm{
X\ar[r]^-{\eta_{_X}}\ar[d]_f & K(X)\ar[d]^{K(f)}\\
Y\ar[r]^-{\eta_{_Y}} & K(Y)
}
\end{array}\text{is a pullback}
\right\}.
\end{aligned}
\end{equation*}

\vskip 5pt

We claim that $\mathcal{M}_N  = \{f\mid N(f) \text{ is an isomorphism}\}$.

\vskip 5pt

For $f:G\longrightarrow H$, consider the pullback of $\eta_{_H}:H\longrightarrow K(H)$ along $K(f):K(G) \longrightarrow K(H)$, one has the following commutative diagram with exact rows:
\[
\xymatrix@C=0.8cm{
1\ar[r] & N(G)\ar[r]\ar[d]_{N(f)} & G\ar[r]^{\eta_{_G}}\ar[d]_{\varphi} & K(G)\ar[r]\ar@{=}[d] & 1\\
1\ar[r] & N(H)\ar[r]\ar@{=}[d] & Q\ar[r]\ar[d] & K(G)\ar[r]\ar[d]^{K(f)} & 1\\
1\ar[r] & N(H)\ar[r] & H\ar[r]^{\eta_{_H}} & K(H)\ar[r] & 1.
}
\]

\vskip 5pt

If $N(f):N(G)\longrightarrow N(H)$ is an isomorphism, then by Lemma~\ref{lem:topgrp-short-five}, one has $\varphi$ is an isomorphism. Thus $f\in\mathcal M_N$.
Conversely, suppose that $f\in\mathcal M_N$. By definition, one has $\varphi$ is an isomorphism.
Thus $N(f)$ is an isomorphism, see the first row in the diagram above.
This proves the claim.

\vskip 5pt

Since $N$ and $K$ are functors, $\mathcal E_N$ and $\mathcal M_N$ have the two-out-of-three property.
By definition the orthogonal factorization system $(\mathcal E_N,\mathcal M_N)$ is a torsion factorization system.

\vskip 5pt

(2) \ The proof is similar to (1). We omit the details.
\end{proof}

\begin{rem}\label{rem:description-of-F-in-topgrp}\ In fact, $\mathcal F_N$ is the full subcategory of $T_0$ topological groups, and $\mathcal F_\Gamma$ is the full subcategory of totally disconnected topological groups.
\end{rem}

\begin{prop}\label{prop:topological-group-torsion-model-structure}\ There is a torsion model structure $(\CoFib,\Fib,\Weq)$ on $\mathsf{TopGrp}$ given by
\[
\begin{aligned}
\CoFib&=\{f\in\Mor (\mathsf{TopGrp})\mid\pi_0(f)\text{ is an isomorphism}\}\\
\Fib&=\{f\in\Mor (\mathsf{TopGrp})\mid N(f)\text{ is an isomorphism}\}\\
\Weq&=\{f\in\Mor (\mathsf{TopGrp})\mid H_0(f)\text{ is an isomorphism}\}.
\end{aligned}
\]
The class $\mathcal C$ of cofibrant objects is the class of connected topological groups,
the class $\mathcal F$ of fibrant objects is the class of $T_0$-topological groups,
and the homotopy category is
\[
\operatorname{Ho}(\mathsf{TopGrp})\simeq\{G\in\mathsf{TopGrp}\mid\Gamma(G)=G\text{ and }N(G)=1\}.
\]
\end{prop}

\begin{proof} \ Lemma~\ref{lem:topgrp} gives torsion factorization systems $(\mathcal E_N,\mathcal M_N)$ and $(\mathcal E_\Gamma,\mathcal M_\Gamma)$, where
\[
\begin{aligned}
\mathcal{E}_{N}& = \{f\mid K(f) \text{ is an isomorphism} \}, \ \ \ \
\mathcal M_N= \{f\mid N(f) \text{ is an isomorphism} \}, \\
\mathcal{E}_{\Gamma}& = \{f\mid \pi_0(f) \text{ is an isomorphism} \}, \ \ \ \
\mathcal{M}_{\Gamma}  = \{f\mid \Gamma(f) \text{ is an isomorphism} \}.
\end{aligned}
\]

\vskip5pt

For every $G\in\mathsf{TopGrp}$, there are natural short exact
sequences
\[
1\longrightarrow N(G)\longrightarrow\Gamma(G)\longrightarrow H_0(G)
\longrightarrow1,
\qquad
1\longrightarrow H_0(G)\longrightarrow K(G)\longrightarrow\pi_0(G)
\longrightarrow1.
\]
Moreover, $\Gamma(K(G))\cong H_0(G)$ by definition.
If $e:G\longrightarrow G'$ belongs to $\mathcal E_N$, then $K(e)$ is an isomorphism.
Since $H_0(G)=\Gamma(K(G))$, the morphism $K(e)$ induces an isomorphism $H_0(e):H_0(G)\longrightarrow H_0(G')$.
The second short exact sequence then gives that $\pi_0(e)$ is an isomorphism.
Hence $e\in\mathcal E_\Gamma$, and therefore $\mathcal E_N\subseteq\mathcal E_\Gamma$ and $\mathcal M_\Gamma\subseteq\mathcal M_N.$
We claim that
\[
\Weq: = \{\ f\mid H_0(f) \text{ is an isomorphism}\ \} = \mathcal M_\Gamma\circ\mathcal E_N.
\]

\vskip5pt

First, let $f=m\circ e$, where $e\in\mathcal E_N$ and $m\in\mathcal M_\Gamma$.
As above, $H_0(e)$ is an isomorphism.
Since $\Gamma(m)$ is an isomorphism and $N(m)$ is its restriction to the closed normal subgroup $N(G)$, $N(m)$ is also an isomorphism.
Therefore the first short exact sequence and Lemma~\ref{lem:topgrp-short-five} imply that $H_0(m)$ is an isomorphism.
Consequently, $H_0(f)=H_0(m)\circ H_0(e)$ is an isomorphism.
Thus $\mathcal M_\Gamma\circ\mathcal E_N\subseteq\Weq.$

\vskip5pt

Conversely, let $f:G\longrightarrow G'$ belong to $\Weq$, i.e., $H_0(f)$ is an isomorphism.
Since  $(\mathcal E_N,\mathcal M_N)$ is a torsion factorization system, one can factorize $f$ as $f=m\circ e$ with $e\in\mathcal E_N,\ m\in\mathcal M_N.$
Since $H_0(e)$ is an isomorphism, one has $H_0(m)=H_0(f)\circ H_0(e)^{-1}$, which is an isomorphism.
Moreover, $m\in\mathcal M_N$ implies that $N(m)$ is an isomorphism.
Since one has the following commutative diagram with exact rows:
\[
\xymatrix@C=1.0cm@R=0.7cm{
1\ar[r] & N(G)\ar[r]\ar[d]^{N(m)}_{\cong} & \Gamma(G)\ar[r]^-{\eta_{_G}}\ar[d]^{\Gamma(m)} & H_0(G)\ar[d]^{H_0(m)}_{\cong}\ar[r] & 1\\
1\ar[r] & N(G')\ar[r] & \Gamma(G')\ar[r]^-{\eta_{_G'}} & H_0(G')\ar[r] & 1
}
\]
By Lemma~\ref{lem:topgrp-short-five} one has $\Gamma(m)$ is an isomorphism.
Hence $m\in\mathcal M_\Gamma$, and therefore $f=m\circ e\in\mathcal M_\Gamma\circ\mathcal E_N.$
This proves the claim.

\vskip5pt

By the claim, $\Weq=\mathcal M_\Gamma\circ\mathcal E_N$ is closed under
compositions. Proposition~\ref{prop:weq-torsion} therefore gives the
torsion model structure $(\CoFib,\Fib,\Weq) = (\mathcal E_\Gamma,\mathcal M_N,\mathcal M_\Gamma\circ\mathcal E_N).$

\vskip 5pt

A topological group $G$ is cofibrant if and only if $1\longrightarrow G$ lies in $\mathcal E_\Gamma$, i.e., $\Gamma(G)=G$. Thus the class of cofibrant objects consists of all connected topological groups.
It is fibrant if and only if $G\longrightarrow1$ lies in $\mathcal M_N$, i.e., $N(G)=1$. Thus the class of fibrant objects consists of all $T_0$-topological groups.
By Theorem~\ref{thm:orthogonal-ho-CF} one has $\operatorname{Ho}(\mathsf{TopGrp})\simeq\{G\in\mathsf{TopGrp}\mid\Gamma(G)=G\text{ and }N(G)=1\}.$
\end{proof}

\begin{exm}\label{rem:topgrp-not-bi-reflective} \ The torsion model structure $(\CoFib,\Fib,\Weq)$ constructed
in Proposition~\ref{prop:topological-group-torsion-model-structure} is not bi-reflective.

\vskip 5pt

Let $\mathcal C$ be the class of cofibrant objects, i.e., $\mathcal C:=\{G\in\mathsf{TopGrp}\mid \Gamma(G)=G\}$.
Then $\mathcal C$ is a coreflective subcategory of $\mathsf{TopGrp}$ with the coreflector $\Gamma:\mathsf{TopGrp}\longrightarrow\mathcal C$.
Assume otherwise that $(\CoFib,\Fib,\Weq)$ is a bi-reflective model structure.
Then by Theorem~\ref{thm:bi-reflective-pairs-bi-reflective-orthogonal-model-structures}, $\Gamma$ should be semi-right-exact.
However, $\Gamma$ is not semi-right-exact, as explained below.

\vskip 5pt

Consider the topological group $\mathbb Z_2$  and the connected topological group $(\mathbb R,+)\in \mathcal C$. Then $\Gamma(\mathbb Z_2)=\{1\}$, and
one has the pushout square
\[
\xymatrix@C=1.2cm@R=0.8cm{
\{1\}\ar[r]\ar[d] & \mathbb R\ar[d]^{i}\\
\mathbb Z_2\ar[r]_{j} & P
}
\]
where $P=\mathbb R *\mathbb Z_2$ is the coproduct of $\mathbb R$ and $\mathbb Z_2$ in $\mathsf{TopGrp}$. Let $s$ be the nonidentity element of $\mathbb Z_2$.
The $i(\mathbb R)$ and $j(s)i(\mathbb R)j(s)^{-1}$ are connected subspaces of $P$ containing the identity.
Hence both are contained in $\Gamma(P)$.
Choose $0\neq c\in \mathbb R$.
In the coproduct $\mathbb R *\mathbb Z_2$, the element $i(c)$ is represented by a reduced word of length one.
On the other hand, $j(s)i(c)j(s)^{-1}$ is represented by a reduced word of length three.
Therefore $j(s)i(c)j(s)^{-1}\notin i(\mathbb R)$, thus $i(\mathbb R)$ is not a normal subgroup. It follows that  $i(\mathbb R)\subsetneqq\Gamma(P).$

\vskip 5pt

If $\Gamma$ were semi-right-exact, then $i:\mathbb R\longrightarrow P$ would be a strong right $\mathcal C$-approximation of $P$.
However, by Fact~\ref{fact:approximation-basic}${\rm (1')}$, the strong right $\mathcal C$-approximation of $P$ is given by the inclusion $\Gamma(P)\longrightarrow P.$
Let $\varepsilon_{_P}:\Gamma(P)\longrightarrow P$ be the inclusion.
If $i:C\longrightarrow P$ is a right strong-$\mathcal C$-approximation,
then $\varepsilon_{_P}$ should factor through $i$ (cf. Definition~\ref{defn:strong-approximation}),
it follows that $\Gamma(P)\subseteq i(C)$,
which contradicts $i(\mathbb R)\subsetneqq\Gamma(P).$
Thus $\Gamma$ is not semi-right-exact.
\end{exm}

\subsection{\bf A torsion model structure without enough fibrant or cofibrant objects} \
By viewing posets as categories, we will construct torsion model structures on products of three posets.
In particular, we obtain an example with neither cofibrant nor fibrant objects.

\vskip 5pt

Let $P$ and $P'$ be posets.
Equip $P\times P'$ with the order: $(x,y)\leq(x',y')$ precisely when $x\leq x'$ and $y\leq y'$.
Let $\pi:P\times P'\longrightarrow P$ and $\pi':P\times P'\longrightarrow P'$ be the projection functors.

\begin{lem}\label{lem:product-poset-torsion-ofs}
There is a torsion factorization system $(\mathcal E_\pi,\mathcal M_{\pi'})$ on $P\times P'$, where
\[
\begin{aligned}
\mathcal E_\pi
&=\{f\in\operatorname{Mor}(P\times P')
       \mid \pi(f)\text{ is an isomorphism}\}\\
\mathcal M_{\pi'}
&=\{f\in\operatorname{Mor}(P\times P')
       \mid \pi'(f)\text{ is an isomorphism}\}.
\end{aligned}
\]
\end{lem}
\begin{proof}
By construction, one has
\[
\begin{aligned}
\mathcal E_\pi
&=\{f\in\operatorname{Mor}(P\times P')
       \mid \pi(f)\text{ is an isomorphism}\}=\{(x,y)\longrightarrow(x,y')\mid y\leq y'\},\\
\mathcal M_{\pi'}
&=\{f\in\operatorname{Mor}(P\times P')
       \mid \pi'(f)\text{ is an isomorphism}\}=\{(x,y)\longrightarrow(x',y)\mid x\leq x'\}.
\end{aligned}
\]
Every morphism $f:(x,y)\longrightarrow(x',y')$ admits a factorization $(x,y)\longrightarrow(x,y')\longrightarrow(x',y'),$
whose first map belongs to $\mathcal E_\pi$ and whose second map belongs to $\mathcal M_{\pi'}$.

\vskip 5pt

We next verify $\mathcal E_\pi\perp \mathcal M_{\pi'}$. Consider a commutative square
\[
\xymatrix@C=1.4cm@R=0.7cm{
(x,y)\ar[r]\ar[d]_{e} & (u,v)\ar[d]^{m}\\
(x,y')\ar[r] & (u',v),
}
\]
where $e\in\mathcal E_\pi$ and $m\in\mathcal M_{\pi'}$.
In the commutative diagram, one has $x\leq u$ and $y'\leq v$, it follows that $(x,y')\longrightarrow(u,v)$, which gives the lifting.

\vskip 5pt

By Proposition~\ref{prop:OFS-criterion}, $(\mathcal E_\pi,\mathcal M_{\pi'})$ is an orthogonal factorization system.
Since $\mathcal E_\pi$ and $\mathcal M_{\pi'}$ both have two-out-of-three property, one has $(\mathcal E_\pi,\mathcal M_{\pi'})$ is a torsion factorization system.
\end{proof}
\begin{prop}\label{prop:three-posets-torsion-model}
Let $P_{-1}$, $P_0$, and $P_1$ be posets, and equip
$\mathcal P:=P_{-1}\times P_0\times P_1$ with the order: $(x,y,z)\leq(x',y',z')$ precisely when $x\leq x', y\leq y'$ and $z\leq z'$.
For $i\in\{-1,0,1\}$, let $\pi_i:\mathcal P\longrightarrow P_i$
be the projection functor.
Then there is a torsion model structure $(\CoFib,\Fib,\Weq)$ on $\mathcal P$, where
\[
\begin{aligned}
\CoFib
&=\{f\in\operatorname{Mor}(\mathcal P)
       \mid \pi_{-1}(f)\text{ is an isomorphism}\}\\
\Fib
&=\{f\in\operatorname{Mor}(\mathcal P)
       \mid \pi_1(f)\text{ is an isomorphism}\}\\
\Weq
&=\{f\in\operatorname{Mor}(\mathcal P)
       \mid \pi_0(f)\text{ is an isomorphism}\}.
\end{aligned}
\]
\end{prop}
\begin{proof}
By Lemma~\ref{lem:product-poset-torsion-ofs}, one has two torsion factorization systems $(\mathcal E,\mathcal M)$ and $(\mathcal E',\mathcal M')$, where
\[
\begin{aligned}
\mathcal E
&=\{f\in\operatorname{Mor}(\mathcal P)
       \mid \pi_{-1}(f)\text{ is an isomorphism}\} \\
\mathcal M
&=\{f\in\operatorname{Mor}(\mathcal P)
       \mid \pi_0(f)\ \mbox{and}\ \pi_1(f)\text{ are isomorphisms}\},\\
\mathcal E'
&=\{f\in\operatorname{Mor}(\mathcal P)
       \mid \pi_{-1}(f)\ \mbox{and}\ \pi_0(f)\text{ are isomorphisms}\} \\
\mathcal M'
&=\{f\in\operatorname{Mor}(\mathcal P)
       \mid \pi_1(f)\text{ is an isomorphism}\}.
\end{aligned}
\]
It is clear that $\mathcal E'\subseteq \mathcal E$.

\vskip 5pt

We claim that $\mathcal M\circ\mathcal E' = \{f\in\operatorname{Mor}(\mathcal P) \mid \pi_0(f)\text{ is an isomorphism}\}$.
In fact, let $f:(a,b,c)\longrightarrow(a',b',c')$ be a morphism in $\mathcal P$.

First suppose that $f=m\circ e$, where $e\in\mathcal E'$ and $m\in\mathcal M$.
Since $e\in\mathcal E'$, it has the form $e:(a,b,c)\longrightarrow(a,b,d).$
Since $m\in\mathcal M$, it has the form $m:(a,b,d)\longrightarrow(a',b,d).$
Consequently, $\pi_0(f)$ is an isomorphism.

Conversely, suppose that $\pi_0(f)$ is an isomorphism, i.e., $b=b'$. Therefore $f$ admits the factorization $(a,b,c)\overset{e}{\longrightarrow}(a,b,c')
\overset{m}{\longrightarrow}(a',b,c').$
Clearly, $e\in\mathcal E'$ and $m\in\mathcal M$, thus $f=m\circ e\in\mathcal M\circ\mathcal E'$.
This proves the claim.

\vskip 5pt

Since $\pi_0$ is a functor, this class is closed under composition.
Therefore, by Proposition~\ref{prop:weq-torsion}, the triple $(\CoFib,\Fib,\Weq):=(\mathcal E,\mathcal M',\mathcal M\circ\mathcal E')$ is a torsion model structure on $\mathcal P$.
\end{proof}

\begin{exm}\label{exm:Z3-no-enough-CF}\ In this example, $\mathbb Z$ is viewed as a poset $(\mathbb Z,\le)$ with natural order. Put $\mathcal P =(\mathbb Z,\le) \times(\mathbb Z,\le) \times(\mathbb Z,\le) =(\mathbb Z^3,\le)$. For simplicity, we will still denote the poset $(\mathbb Z,\le)$ by $\mathbb Z$ and $(\mathbb Z^3,\le)$ by $\mathbb Z^3$, respectively.
By Proposition~\ref{prop:three-posets-torsion-model}, one has a torsion model structure $(\CoFib,\Fib,\Weq)$ where
\[
\begin{aligned}
\CoFib&=\{f:(a,b,c)\longrightarrow(a',b',c')\ \text{in} \ \operatorname{Mor}(\mathcal P)\ \mid \ a=a'\}\\
\Fib&=\{f:(a,b,c)\longrightarrow(a',b',c')\ \text{in} \ \operatorname{Mor}(\mathcal P)\ \mid \ c=c'\}\\
\Weq&=\{f:(a,b,c)\longrightarrow(a',b',c')\ \text{in} \ \operatorname{Mor}(\mathcal P)\ \mid \ b=b'\}.
\end{aligned}
\]

\vskip 5pt

However, this model structure has no cofibrant objects or fibrant objects.
In fact, for any object $(x,y,z)$, the map $p:(x-1,y,z)\longrightarrow(x,y,z)$ is in $\TFib$.
If $(x,y,z)$ is a cofibrant object in $\mathcal P$ (cf. Definition~\ref{defn:cofibrant-fibrant-objects}), then one has a unique morphism $v$ such that $p\circ v={\rm Id}_{(x,y,z)}$:
\[
\xymatrix@C=1.4cm@R=0.8cm{
& (x-1,y,z)\ar[d]^{p}\\
(x,y,z)\ar@{=}[r]\ar@{-->}[ur]^{v} & (x,y,z).
}
\]
But such morphism $v$ would require $x\leq x-1$, which is a contradiction.
Hence there is no cofibrant object in $\mathcal P$.
Dually, there is no fibrant object in $\mathcal P$.

\vskip 5pt

In this case, one can still calculate the homotopy category. We claim that $\Ho(\mathbb Z^3)\simeq\mathbb Z$.
In fact, let $q:\mathbb Z^3\longrightarrow\mathbb Z,\ (a,b,c)\mapsto b$ be a projection, and $i:\mathbb Z\longrightarrow\mathbb Z^3,\ y\mapsto (0,y,0)$ be an inclusion.
Let $\gamma:\mathbb Z^3\longrightarrow\Ho(\mathbb Z^3)$ be the localization functor.
Since $q$ sends every weak equivalence to an isomorphism, there exists a unique functor $\overline q:\Ho(\mathbb Z^3)\longrightarrow \mathbb Z$ such that $q=\overline q\circ\gamma$.
Put $\overline i:=\gamma \circ i:\mathbb Z\longrightarrow\Ho(\mathbb Z^3)$.
Then one has $\overline q\circ \overline i=\overline q \circ \gamma \circ i= q\circ i= \operatorname{Id}_{\mathbb Z}.$

\vskip 5pt

On the other hand, define $h:\mathbb Z^3\longrightarrow\mathbb Z^3$ by
$h(x,y,z)=(\min\{x,0\},y,\min\{z,0\})$.
Then $h$ is an endofunctor of $\mathcal P =\mathbb Z^3$.
Since $h(x,y,z)\leq(x,y,z)$ for every object $(x,y,z)\in\mathbb Z^3$, there is a natural transformation $\alpha:h\longrightarrow\operatorname{Id}_{\mathbb Z^3}$.
Since $h(x,y,z)\leq(0,y,0)=(i\circ q)(x,y,z)$,  there is a natural transformation $\beta:h\longrightarrow i\circ q$.
Applying the localization functor $\gamma:\mathbb Z^3\longrightarrow\Ho(\mathbb Z^3)$ to these natural transformations, since $\gamma$ sends weak equivalences to isomorphisms, one gets the natural isomorphisms
$\gamma(\alpha):\gamma(h)\overset{\cong}{\longrightarrow}\gamma(\operatorname{Id}_{\mathbb Z^3})=\operatorname{Id}_{\Ho(\mathbb Z^3)},$
and
$\gamma(\beta):\gamma(h)\overset{\cong}{\longrightarrow}\gamma(i\circ q)=\gamma\circ i\circ q=\overline i\circ \overline q\circ \gamma,$
and hence $\overline i\circ\overline q\cong\operatorname{Id}_{\Ho(\mathbb Z^3)}$.
Thus $\overline q$ and $\overline i$ are quasi-inverse, and $\Ho(\mathbb Z^3)\simeq\mathbb Z$.

\end{exm}

\end{document}